\documentclass[11pt]{amsart}

\usepackage{amsmath,amssymb,epsfig}
\usepackage{graphicx}
\usepackage{fullpage}
\usepackage{amscd}
\usepackage{amsfonts, amssymb} 
\usepackage{stmaryrd}
\usepackage[all]{xy}
\usepackage{upgreek}
\usepackage{mathtools}
\usepackage{datetime2} 
\usepackage{latexsym}
\usepackage[usenames]{color}
\usepackage{xcolor}
\usepackage{url}
\usepackage{boxedminipage}
\usepackage{amsbsy}
\usepackage{framed}
\usepackage[colorlinks=false, linktocpage=true]{hyperref}

\usepackage{enumitem}

\usepackage{bm}
\usepackage{tikz-cd}
\newsavebox{\pullback}
\sbox\pullback{%
\begin{tikzpicture}%
\draw (0,0) -- (1ex,0ex);%
\draw (1ex,0ex) -- (1ex,1ex);%
\end{tikzpicture}}

\usepackage{longtable}
\usepackage{hyperref}

\usepackage[backend=biber]{biblatex}

\bibliography{references.bib}

\newcommand{\fg}{\mathfrak g}

\newcommand{\cA}{{\mathcal A}}

\newcommand{\cG}{{\mathcal G}}
\newcommand{\cL}{{\mathcal L}}
\newcommand{\cM}{{\mathcal M}}
\newcommand{\cN}{{\mathcal N}}

\newcommand{\cW}{{\mathcal W}}
\newcommand{\cP}{{\mathcal P}}
\newcommand{\cO}{{\mathcal O}}

\newcommand{\cS}{{\mathcal S}}

\newcommand{\cX}{{\mathcal X}}
\newcommand{\cY}{{\mathcal Y}}
\newcommand{\cZ}{{\mathcal Z}}
\newcommand{\R}{\mathbb{R}}

\newcommand{\Z}{\mathbb{Z}}

\usepackage{mathrsfs}

\DeclareMathAlphabet{\mathpzc}{OT1}{pzc}{m}{it}

\newcommand{\cin}{C^\infty}

\newcommand{\coker}{\mathrm{coker}}

\newcommand{\catname}[1]{\mathsf{#1}}

\DeclareMathOperator{\colim}{colim}

\DeclareMathOperator{\Hom}{Hom}
\DeclareMathOperator{\id}{id}

\DeclareMathOperator{\cone}{cone}
\DeclareMathOperator{\cocone}{cocone}
\DeclareMathOperator{\Tot}{Tot}

\numberwithin{equation}{section}

\theoremstyle{definition}
\newtheorem{thm}{Theorem}[section]
\newtheorem{lemma}[thm]{Lemma}
\newtheorem{theorem}[thm]{Theorem}

\newtheorem{proposition}[thm]{Proposition}
\newtheorem*{corollary*}{Corollary}
\newtheorem*{claim*}{Claim}

\newtheorem{definition}[thm]{Definition}
\newtheorem{remark}[thm]{Remark}
\newtheorem{example}[thm]{Example}
\newtheorem{notation}[thm]{Notation}
\newtheorem {construction}[thm]{Construction}

\begin{document}

\begin{flushright}
    \today \\[2em] 
\end{flushright}
\title{ Derived Symplectic Reduction in Differential Geometry} 
 
\author{Nikolay Sheshko}
\address{Department of Mathematics, University of Illinois Urbana-Champaign, Urbana, IL 61801, USA}
\date{}

\setcounter{tocdepth}{1}

\begin{abstract} In this article we prove a derived version of the Marsden-Weinstein-Meyer symplectic reduction theorem. We model the symplectic quotient as a dg-groupoid. We then construct the reduced symplectic form inside the Bott--Shulman complex of the groupoid. Finally, we show that the reduced form induces an isomorphism between total tangent and cotangent complexes of the groupoid.

\end{abstract}
\maketitle
\tableofcontents

\section{Introduction}

In classical mechanics, the symmetries of the system are often used to study its dynamics. An example of this is Noether's theorem, which states that symmetries lead to conservation laws. In mathematical language, the phase space of a system is modelled by a symplectic manifold $(M, \omega)$. The symplectic form allows one to naturally formulate the Hamilton equations of motion. In particular given the Hamiltonian $H \colon M \to \R$ of a system, its dynamics is encoded in a vector field $X_H$, which is determined uniquely from the Hamilton equation $\iota_{X_H} \omega  = dH$. The notion of symmetry of a system is modelled by a Hamiltonian action of a Lie group $G$ on the phase space. Such an action comes equipped with a $G$-equivariant moment map $\mu \colon M \to \fg^*$, which encodes the Hamiltonians corresponding to the infinitesimal symmetries of the system. To be more precise, for any element of the Lie algebra $X \in \fg$ there is a corresponding vector field $X^{\sharp}$ on the manifold $M$ given by the infinitesimal action of $X$, and the moment map satisfies Hamiltonian equations with respect to that vector field: $\iota_{X^{\sharp}} \omega = d\langle \mu, X\rangle$.

Symplectic reduction is a mathematical procedure of constructing a lower dimensional phase space out of a system equipped with a Hamiltonian group action. In its classical form it was developed independently by Marsden--Weinstein and Meyer \cite{MW,MEYER}. The theorem states that given a Hamiltonian $G$-action on a symplectic manifold $(M,\omega)$ with moment map $\mu\colon M\to\fg^*$ for which $0$ is a regular value, the quotient $\mu^{-1}(0)/G$ is again a symplectic manifold, provided that the action of the Lie group $G$ on the zero set $\mu^{-1}(0)$ is free and proper. The symplectic form on the reduced space is uniquely determined by the condition that its pullback to the zero set $\mu^{-1}(0)$ coincides with the restriction of the original form $\omega$ to the zero-level set $\mu^{-1}(0)$.

However, more often than not, the requirement that $0$ is a regular value is too strong. Even simple Hamiltonian systems can exhibit a singularity. In that case the resulting reduced space is typically not a smooth manifold. Various approaches have been developed to deal with such singularities. The reduced space has been studied as a stratified symplectic space \cite{SL}, as a differential space \cite{SNIATYCKI}, or via algebraic/Poisson methods \cite{SW,AGJ}.

The purpose of this paper is to give a \emph{derived} model for symplectic reduction that works for an arbitrary moment map, without imposing regularity assumptions. This approach is inspired by recent developments in derived differential geometry \cite{CARCHEDI, BLX} (see section \ref{sec:overview} for an overview) and similar results in derived algebraic geometry \cite{PECHARICH}. While other methods dealing with singularities often rely on additional assumptions, such as properness of the group action or specific properties of the moment map, the derived approach does not require any such conditions. This makes it useful conceptually, as it can be applied in a general setting. However, particular applications remain to be explored. 

The central idea of the derived approach is to replace the singular quotient $\mu^{-1}(0)/G$ with a differential geometric derived object that retains the full geometric information of the reduction. In this article we model the zero-level set $\mu^{-1}(0)$ as a dg-manifold $\cZ$ constructed as a derived vanishing locus of the moment map $\mu$, following the approach to homotopy intersections of Carchedi and Behrend-Liao-Xu \cite{CARCHEDI, BLX}. This model encodes the singularities of the zero set $\mu^{-1}(0)$ in the cochain complex associated to the dg-manifold $\cZ$. The group action of the Lie group $G$ on the dg-manifold $\cZ$ is encoded in an action groupoid $\cZ/G:= \{G\times \cZ \rightrightarrows \cZ\}$, whose Bott--Shulman complex (which is a triple complex in the derived case) plays the role of differential forms on the (stacky) quotient. We then construct a closed form $\omega_{\mathrm{red}}$ of total degree $2$ in this complex which pulls back to the restricted form $\iota_\cZ^* \omega$ under the projection $\pi \colon \cZ \to \cZ/G$, i.e. $\pi^* \omega_{\mathrm{red}} = \iota^*_\cZ \omega$ (Propositions \ref{prop:closed} and \ref{prop:red}), and show that it induces an isomorphism between the tangent and cotangent complexes of $\cZ/G$, giving the derived analogue of non-degeneracy (Theorem \ref{thm:nondeg}).

A result similar to ours was announced at the Workshop on Hamiltonian Geometry and Quantization in 2024 by Sjamaar in collaboration with Cueca, Dorsch and Zhu \cite{TORONTO}. Their strategy involves modelling the symplectic quotient as a derived Lagrangian intersection of stacks $[M/G]\times_{[\fg^*/G]} [\{0\}/G]$. Their reduced form agrees with the form $\omega_{\mathrm{red}}$ of this paper.

\subsection*{Organization of the paper}
The paper is organized as follows. In Sections \ref{sec:overview} and \ref{sec:dgmanifolds}, we provide a brief overview of derived differential geometry and the theory of dg-manifolds. In Section \ref{sec:kahler} we describe tangent and cotangent modules of dg-manifolds. In Sections \ref{sec:actiongroupoid} and \ref{sec:dgalg} we construct the dg-action groupoid and the corresponding dg-Lie algebroid. In Section \ref{sec:bottshulman} we build the Bott--Shulman complex, introduce the reduced 2-form and show that it is closed. Finally, in Section \ref{sec:symplectic} we prove that the reduced form induces a non-degenerate pairing between the tangent and cotangent complexes of the reduced space. The paper concludes with a notation index and three appendices, covering, respectively, the de~Rham complex of a dg-manifold, fiber products of dg-manifolds, and a dg-analogue of the Serre--Swan theorem.

\subsection*{Acknowledgments}

I would like to thank my advisor, Eugene Lerman, for his patient guidance throughout this project. I would also like to thank Miquel Cueca, Wilmer Smilde and \v{Z}an Grad for useful discussions. 

\section{Overview of Derived Differential Geometry} \label{sec:overview}

Derived differential geometry is a relatively new field compared to its algebraic counterpart. It aims to enlarge the category of manifolds, which lacks arbitrary fiber products and quotients. In particular, if we have two maps $f\colon M \to L$ and $g \colon N \to L$ between smooth manifolds, the fiber product $M \times_L N$ is guaranteed to be a smooth manifold only if the maps $f$ and $g$ are transverse, i.e. if for any points $x \in M$ and $y \in N$ such that  $f(x) = g(y)$, we have \[df_x(T_xM) + dg_y(T_yN) = T_{f(x)} L. \] However, in general when the maps $f$ and $g$ are not transverse, their fiber product can be an arbitrarily singular topological space. For example, by a theorem of Whitney, any closed subset of a manifold is a zero-level set of some smooth map $f\colon M \to \R$, and therefore can be realised as the topological fiber product $M\times_{\R} \{0\}$. 

On the other hand, the ($\infty$-)category of derived manifolds has arbitrary \emph{homotopy} fiber products. There are several approaches to constructing such a category. The foundational work is a paper of Spivak \cite{SPIVAK}. There, he defined derived manifolds as homotopy sheaves of homotopical $\cin$-rings. Borisov and Noel \cite{BORISOVNOEL} simplified Spivak's model by using simplicial $\cin$-rings. Carchedi and Roytenberg \cite{CARCHEDIROYTENBERG, CARCHEDI} proposed another model based on dg-$\cin$-rings. A more geometric approach to derived geometry was developed by Behrend, Liao and Xu \cite{BLX}, who defined derived manifolds as bundles of curved $L_\infty[1]$-algebras. Recently, Taroyan \cite{TAROYAN} showed that all these models are (Quillen) equivalent.

Another approach to derived manifolds, and the one we will be using in this paper, is based on dg-manifolds of negative amplitude. Dg-manifolds originally appeared as a geometric model for the Batalin--Vilkovisky formalism  (see \cite{SCHWARZ}, \cite{AKSZ}). They are smooth manifolds equipped with an extended algebra of functions, possibly containing elements of non-zero degrees, together with a cohomological vector field of degree $1$. The category of dg-manifolds of negative amplitude is equivalent to the category of curved $L_\infty[1]$-bundles of Behrend, Liao and Xu \cite[Proposition 2.13]{BLX}, and therefore can serve as a model for derived geometry. A classical example of a dg-manifold is the derived vanishing locus of a section $s$ of a vector bundle $E \to M$, which is given by the dg-manifold $E[-1]$ with a sheaf of sections of the graded vector bundle $\wedge^{-\bullet} E^*$ concetrated in negative degrees and a cohomological vector field $\iota_s$ given by contraction with the section $s$ (see Construction~\ref{ctr:zerosection}). All the information about the singularities of the zero-level set $s^{-1}(0)$ is contained in the cohomology of the cochain complex $(\Gamma(\wedge^{-\bullet} E^*), \iota_s)$. In particular, when the section $s$ intersects the zero section transversely, the cohomology is concentrated in degree 0 and is isomorphic to smooth functions on the algebra of smooth functions $C^{\infty}(s^{-1}(0))$ on the zero set.

In general, homotopy fiber products of dg-manifolds are defined as a homotopy pullback. Behrend, Liao and Xu \cite{BLX} showed that the category of dg-manifolds has a structure of a \emph{category of fibrant objects}. This is a weaker notion than that of a model category, as it has only two classes of distinguished morphisms: fibrations and weak equivalences. However, this structure is enough to form homotopy fiber products. In particular for any dg-manifold $\cX$ there is a path object $\mathrm{Path}(\cX)$ which factors the diagonal map $\Delta \colon \cX \to \cX \times \cX$: \[\cX \xrightarrow{w} \mathrm{Path}(\cX) \xrightarrow{F_\Delta} \cX\times \cX\] so that $w$ is a weak equivalence, $F_\Delta$ is a fibration, and $F_\Delta \circ w = \Delta$. 

Now, suppose we have two maps of dg-manifolds $f\colon \cX \to \cL$ and $g \colon \cY \to \cL$. The ordinary fiber product of the dg-manifolds $\cX$ and $\cY$, if it exists, can be formed using a pullback diagram \[\begin{tikzcd}
        \cX\times_\cL \cY \arrow[r] \arrow[d] \arrow[dr, phantom, "\usebox\pullback" , very near start, color=black] & \cL  \arrow[d, "{\Delta}"] \\
        \cX \times \cY \arrow[r, "{(f,g)}"] & \cL \times \cL. 
        \end{tikzcd}\]
The problem is that the diagonal map $\Delta$ is not a fibration, and therefore the existence of the pullback is not guaranteed. However, replacing the base dg-manifold $\cL$ by a weakly equivalent path object $\mathrm{Path}(\cL)$, and the diagonal $\Delta$ by a fibration $F_\Delta \colon \mathrm{Path}(\cL) \to \cL\times \cL$, we obtain the homotopy fiber product of dg-manifolds $\cX$ and $\cY$ over $\cL$ as a pullback \[\begin{tikzcd}
        \cX\times_\cL^h \cY \arrow[r] \arrow[d] \arrow[dr, phantom, "\usebox\pullback" , very near start, color=black] & \mathrm{Path}(\cL) \arrow[d, "{F_\Delta}", twoheadrightarrow]  \\
        \cX \times \cY \arrow[r, "{(f,g)}"] & \cL \times \cL.
        \end{tikzcd}\]
Such homotopy fiber products are unique up to weak equivalence.

Another source of singularities, which we have to consider when working with symplectic reduction, is quotients by group actions. The category of manifolds is not closed under taking quotients by group actions. In particular, if a Lie group $G$ acts on a manifold $M$, the quotient $M/G$ is guaranteed to be a smooth manifold only if the action is free and proper. This problem can be resolved by considering the quotient as a differentiable stack $[M/G]$. Every differentiable stack $\mathfrak{X}$ can be presented by some groupoid $\cG = \{\Gamma_1 \rightrightarrows \Gamma_0\}$, which plays the role of an atlas for $\mathfrak{X} \cong [\Gamma_0/\Gamma_1]$. Any two Morita equivalent groupoids present equivalent stacks (see \cite{BEHRENDXU}). In particular, if a Lie group $G$ acts on a manifold $M$, the action groupoid $\{G\times M \rightrightarrows M\}$ presents the quotient stack $[M/G]$. The search for a category which is closed under arbitrary homotopy quotients motivates the study of higher Lie groupoids and higher differentiable stacks (see for example \cite{HIGHERSTACKS}).

Combining derived manifolds with higher differentiable stacks leads to the category of \emph{derived stacks} which is closed under arbitrary homotopy limits and homotopy colimits. In differential geometry the construction of derived stacks has been carried out formally, using techniques of Toën and Vezzosi (see Taroyan \cite{TAROYAN}, Alfonsi and Young \cite{ALFONSIYOUNG}). In practice, it is often more convenient to work with a specific model for the given derived stack. In the context of the symplectic quotient $\mu^{-1}(0)/G$, it will be convenient to model the derived zero-level set $\mu^{-1}(0)$ as a dg-manifold $\cZ$, and the quotient as an action groupoid $\cZ/G = \{G \times \cZ \rightrightarrows \cZ\}$.

\section{Dg-manifolds} \label{sec:dgmanifolds}

In the literature there are different definitions of dg-manifolds with varying levels of rigour. Usually dg-manifolds are defined to be concentrated purely in negative or purely positive degrees. In this case, dg-manifolds of negative amplitude model derived intersections, while dg-manifolds of positive amplitude model infinitesimal quotients. However, in practice it is sometimes useful to consider dg-manifolds of mixed amplitude. For example, even if one considers a dg-manifold of negative amplitude, its cotangent bundle will have functions of positive degrees. Introducing mixed degrees comes with additional technical problems as now we have to deal with power series of variables instead of polynomials. Kotov and Salnikov \cite{KTSL} gave a rigorous definition of $\Z$-graded dg-manifolds with unbounded amplitude using filtrations of graded modules. However, since we are working with bounded amplitudes, for our purposes it will be enough to use the definition from Vysoky's monograph \cite{VYSOKY} on graded geometry. We refer the reader to Vysoky's work for definitions and facts about graded geometry. 

\begin{remark}
    All $\Z$-graded objects (like graded vector spaces, graded algebras or graded modules) introduced in this paper are defined as collections of objects indexed by integers. For example, a graded vector space $V$ is a collection $\{V_i\}_{i\in \Z}$ of vector spaces. This notation is consistent with the one used by Vysoky in \cite{VYSOKY}. 

    In this notation a graded-commutative algebra $A$ is a collection $\{A_i\}_{i\in \Z}$ of vector spaces, together with a multiplication map $\mu = \{\mu_{ij}\}_{i,j \in \Z}$, where $\mu_{ij} \colon A_i \otimes A_j \to A_{i+j}$ are bilinear maps. The multiplication $\mu$ is required to satisfy the usual graded-commutativity and associativity properties.

    Note that with this notation, \emph{all elements of a graded vector space or a graded algebra are homogeneous}, i.e. they have a well-defined degree. In particular, if $V = \{V_i\}_{i\in \Z}$ is a graded vector space, then we write $v\in V$ if there is some index $i \in \Z$ such that $v \in V_i$. We also write $\lvert v \rvert = i$ to denote the degree of an element $v$.
\end{remark}

Let $V = \{V_i\}_{i\in \Z}$ be a $\Z$-graded vector space. The symmetric algebra $\mathbf{Symm}_{\R}(V)$ is the free graded-commutative algebra generated by the space $V$. It is given by the quotient of the tensor algebra $T_{\R}(V)$ by the ideal generated by elements of the form $v\otimes w - (-1)^{\lvert v \rvert \cdot \lvert w \rvert} w \otimes v$ for any elements $v,w \in V$. The degree $n$ symmetric power $\mathbf{Symm}^n_{\R}(V)$ is the subspace of $\mathbf{Symm}_{\R}(V)$ spanned by monomials with $n$ factors. The algebra $\mathbf{Symm}_{\R}(V)$ is graded by the sum of the degree of the factors.

Let $A$ be a graded-commutative algebra (over $\R$). We define \[\overline{\mathbf{Symm}}(V, A) = \prod_{n\geq 0}  A \otimes_{\R} \mathbf{Symm}^n_{\R} V\] --- the extended symmetric algebra of formal power series in elements of the space $V$ with values in the algebra $A$ (see discussion after \cite[Proposition 1.21]{VYSOKY}). The product is taken in the category of graded vector spaces.

\begin{remark}[see {\cite[Remark 1.23]{VYSOKY}}]
    \label{rmk:polvsseries}
    In some cases, the extended symmetric algebra $\overline{\mathbf{Symm}}(V, A)$ coincides with the simpler symmetric algebra of polynomials with values in the algebra $A$ \[\mathbf{Symm}(V, A) = \bigoplus_{n\geq 0}  A \otimes_{\R} \mathbf{Symm}^n_{\R} V.\] For example, if the space of coefficients $A$ is concentrated in degree 0 and the graded vector space $V$ is concentrated in either purely positive or purely negative degrees, then power series and polynomials coincide: $\overline{\mathbf{Symm}}(V, A) = \mathbf{Symm}(V, A)$. The introduction of the extended symmetric algebra is necessary to deal with cases in which the space $V$ has mixed degrees.
\end{remark}

\begin{notation} Let $(n_i)_{i\in \Z}$ be a sequence of non-negative integers with finitely many non-zero elements. We define a graded vector space \[\R^{(n_i)} = \{ \R^{n_i}\}_{i \in \Z}.\] We also write \[\R^{(n_i)}_* = \{ \R^{n_i} \}_{i \in \Z \setminus \{0\}}\] for the space $\R^{(n_i)}$ with zero component removed.
\end{notation}

In the next definition, we introduce a local model for dg-manifolds.

\begin{definition} \label{def:graded_domain}
    Consider a presheaf of power series with coefficients in smooth functions \[\cin_{(n_i)}(U) = \overline{\mathbf{Symm}}(\R^{(n_i)}_*, C^{\infty}(U))\] on open subsets $U \subset \R^{n_0}$. In \cite[Proposition 3.1]{VYSOKY} it is shown that this presheaf is in fact a sheaf. For any open $U \subset \R^{n_0}$ we call the pair $(U, \cin_{(n_i)}\vert_U)$ a \emph{graded domain}.
\end{definition}

\begin{definition}
    \label{def:dgmanifold}
   A \emph{dg-manifold of bounded amplitude} $\cX$ is a triple $(\cX^0, \mathcal{O}_{\cX}, \delta)$, where $\cX^0$ is a real differentiable manifold and $\mathcal{O}_{\cX}$ is a sheaf of graded-commutative algebras over $\cX^0$ and $\delta \colon \mathcal{O}_{\cX} \to \mathcal{O}_{\cX}$ is an $\R$-linear map of sheaves satisfying the following conditions:
\begin{itemize}
    \item Locally the sheaf $\mathcal{O}_{\cX}$ is isomorphic to a graded domain \[(U, \mathcal{O}_\cX \vert_U) \cong (V, \cin_{(n_i)}\vert_V)\] (as in Definition \ref{def:graded_domain}) for some sequence of non-negative integers with at most finitely many non-zero elements $(n_i)_{i\in \Z}$ and open subset $V \subset \R^{n_0}$. We require that the sequence $(n_i)_{i\in \Z}$ is the same for all local charts of $\cX$.
    \item The map $\delta$ is a derivation of degree 1, i.e. for any open subset $U \subset \cX^0$ and any homogeneous elements $f,g \in \mathcal{O}_{\cX}(U)$ we have \[\delta\vert_U(f\cdot g) = \delta\vert_U(f)\cdot g + (-1)^{\lvert f \rvert} f \cdot \delta\vert_U(g).\]
    \item $\delta$ is a cohomological vector field, i.e. \[\delta^2 = 0.\]
\end{itemize} 

   A  \emph{morphism of dg-manifolds} $f\colon \cX \to \cY$ is a pair $(\underline{f}, f^*)$, where $\underline{f}\colon \cX^0 \to \cY^0$ is a smooth map and $f^*\colon \mathcal{O}_{\cY} \to \underline{f}_* \mathcal{O}_{\cX}$ is a morphism of sheaves of dg-algebras over $\cY^0$, i.e. for any open subset $U \subset \cY^0$ there is a morphism of graded-commutative algebras \[f^*\vert_U \colon \mathcal{O}_{\cY}(U) \to  \mathcal{O}_{\cX}(\underline{f}^{-1}(U))\] such that $f^*\vert_U \circ \delta_{\cY}\vert_U = \delta_{\cX}\vert_{\underline{f}^{-1}(U)} \circ f^*\vert_U$.
\end{definition}

\begin{notation}
    \label{not:globalfunctions}
Given a dg-manifold $\cX = (\cX^0, \mathcal{O}_{\cX}, \delta)$ of bounded amplitude, we will denote the dg-algebra of global sections $\mathcal{O}_{\cX}(\cX^0)$ by $C^{\infty}(\cX)$.
\end{notation}

The following construction will be extensively used throughout the paper. In particular, we will use it to model the derived zero-level set of the moment map. 

\begin{construction} \label{ctr:zerosection}
    Let $E \to M$ be a vector bundle and let $s \in \Gamma(E)$ be a section. Then there is a dg-manifold $E[-1]$ with the sheaf $\cO_{E[-1]}$ of sections of  \begin{equation} \label{eq:symmvb} \mathbf{Symm}_\R((E[-1])^*) \cong \wedge^{-\bullet} E^*,\end{equation} i.e. in degree $-k$ the sheaf $\cO_{E[-1]}$ consists of sections of $\wedge^k E^*$. Notice that when taking a symmetric power of a graded vector bundle, generators of odd degree anti-commute, so since the bundle $E[-1]$ is concentrated in degree $+1$\footnote{Here we use the convention that $E[k]_i = E_{i+k}$.} and the dual bundle $(E[-1])^* \cong E^*[1]$ is concentrated in degree $-1$, one recovers the exterior algebra $\wedge^{-\bullet} E^*$ in~\eqref{eq:symmvb}. Also, since the bundle $(E[-1])^*$ is concentrated in a single degree, it is enough to use the ordinary symmetric algebra instead of power series (see Remark \ref{rmk:polvsseries}). The cohomological vector field $\delta = \iota_s$ is the contraction with $s$, i.e. for any open $U \subset M$, any smooth function $f \in \cin(U)$ and any sections $\alpha_1, \ldots, \alpha_k$ in $\Gamma(E^*\vert_U)$ we have \[\iota_s(f\cdot \alpha_1 \wedge \ldots \wedge \alpha_k) = \sum_{i=1}^k (-1)^{i-1} f\langle s, \alpha_i \rangle\cdot\alpha_1 \wedge \ldots \wedge \hat{\alpha}_i\wedge\ldots \wedge\alpha_{k}.\]
    
    Following Behrend, Liao and Xu \cite{BLX}, we will call such dg-manifolds \emph{quasi-smooth}. Quasi-smooth dg-manifolds appear as homotopy fiber products of smooth manifolds taken using the structure of a category of fibrant objects on dg-manifolds (see section \ref{sec:overview}). In particular, since arbitrary homotopy fiber products exist in the category of dg-manifolds, we can define the homotopy fiber product $M \times_{s,E,0}^h M$ of the section $s$ with the zero section $0$. Then the dg-manifold $E[-1]$ is a model for this homotopy fiber product, or, more precisely, there is a morphism $w  \colon E[-1] \to M \times_{s,E,0}^h M$, which is a weak equivalence, i.e. the pullback $w^*$ induces a quasi-isomorphism between the corresponding sheaves of cochain complexes (see \cite[Proposition 3.30]{BLX} for the proof). 

\end{construction}

\section{Differential forms and tangent complex} \label{sec:kahler}

In order to formulate symplectic reduction in the language of dg-manifolds, we will require the appropriate notions of differential forms and vector fields. When generalizing notions from smooth manifolds to dg-manifolds, one should use point-free definitions from the smooth world, since dg-manifolds are not sets of points, but rather ringed spaces. Therefore, we construct one-forms on dg-manifolds through a universal property, which is a direct generalization of the universal property satisfied by the module of one-forms on an ordinary manifold (see Remark \ref{rmk:manifold_kahler}). Then, we construct the de Rham complex as the exterior algebra of the module of one-forms. This approach works for arbitrary $\Z$-graded dg-manifolds of bounded amplitude. In Proposition \ref{prop:kahler_ex} we show that our definition of the module of one-forms is equivalent to the definition given by Pridham \cite{PRIDHAM}, who defines one-forms in terms of generators and relations.

\begin{remark}[Modules vs.\ sheaves]
\label{rmk:modules_vs_sheaves}
    Throughout this paper, we work at the level of dg-modules of global sections over the dg-algebra of functions $C^\infty(\cX)$ on the dg-manifold $\cX$, rather than at the level of sheaves of dg-modules over the function sheaf $\cO_\cX$. This choice is made for expository convenience, as it allows us to formulate constructions purely algebraically without carrying sheaf-theoretic data through every proof. In Appendix~\ref{sec:serreswan} we prove the dg-analogue of the Serre--Swan theorem, which states that global section functor induces an equivalence between the category of locally free sheaves of dg-modules of finite rank and the category of projective dg-modules of finite type. Therefore, one can always recover the underlying sheaf of dg-modules, once the module of global sections is shown to be projective and of finite type. 
\end{remark}

\begin{remark} \label{rmk:manifold_kahler}
For an ordinary manifold $M$, the module of de Rham one-forms $\Omega^1(M)$ satisfies the following universal property: for any $C^{\infty}(M)$-module $\mathcal{M}$, and any $\cin$-derivation\footnote{A $\cin$-derivation is a degree 0 case of Definition \ref{def:derivation}.} $d' \colon C^{\infty}(M) \to \mathcal{M}$, there is a unique map $\phi \colon \Omega^1(M) \to \mathcal{M}$ of $C^{\infty}(M)$-modules, making the following diagram commute:
\begin{center}
\begin{tikzcd}C^{\infty}(M) \arrow[r, "d'"] \arrow[rd, "d"] & \mathcal{M}                                                   \\
                                                  & \Omega^1(M) \arrow[u, "\phi"', dotted]
\end{tikzcd}
\end{center}
where $d$ is the usual exterior differential. In the language of $\cin$-algebraic geometry this says that the module $\Omega^1(M)$ is the \emph{module of K\"ahler differentials} of the $\cin$-ring $C^{\infty}(M)$ (for an overview of $\cin$-algebraic geometry see \cite[Chapter 2]{JOYCE}). The full proof of this claim can be found in \cite[Theorem 7.7]{JOYCE}. 
\end{remark}

We define the one-forms on dg-manifolds through a universal property similar to the one of $\cin(M) \xrightarrow{d} \Omega^1(M)$ in Remark \ref{rmk:manifold_kahler}.  

\begin{definition} \label{def:dgmodule}
    A \emph{(left) dg-$C^{\infty}(\cX)$-module} $(\mathcal{M}, \delta_{\mathcal{M}})$ is a graded $\cin(\cX)$-module $\mathcal{M} = \{\mathcal{M}_i\}_{i\in \Z}$ together with a degree 1 map $\delta_{\mathcal{M}} \colon \mathcal{M} \to \mathcal{M}$ satisfying $\delta_{\mathcal{M}}^2 = 0$ and a degree 0 multiplication $C^{\infty}(\cX) \times \mathcal{M} \to \mathcal{M}$ satisfying the usual module axioms and the compatibility condition \[\delta_{\mathcal{M}}(f\cdot m) = \delta_\cX(f)\cdot m + (-1)^{\lvert f \rvert} f\cdot \delta_{\mathcal{M}}(m)\] for any function $f \in C^{\infty}(\cX)$ and any element $m \in \mathcal{M}$.
\end{definition}

\begin{definition} \label{def:shift}
    Given a graded $C^{\infty}(\cX)$-module $\mathcal{M}$, we define a \emph{$k$-shifted $\cin(\cX)$-module} $\mathcal{M}[k]$ by \[\mathcal{M}[k]_i = \mathcal{M}_{i+k}\] for any integer $i \in \Z$. The $\cin(\cX)$-module structure on $\mathcal{M}[k]$ is defined by the formula \[ f \cdot_{\mathcal{M}[k]} m = (-1)^{k \cdot \lvert f \rvert} f \cdot_{\mathcal{M}} m\] for any element $f \in C^{\infty}(\cX)$ and any $m \in \mathcal{M}[k]$.

    \emph{A degree $n$ map of graded $C^{\infty}(\cX)$-modules} $\phi \colon \mathcal{M} \to \mathcal{N}$ is a degree 0 map of graded $C^{\infty}(\cX)$-modules $\phi \colon \mathcal{M} \to \mathcal{N}[n]$. Equivalently, it is a $\R$-linear map satisfying $\phi(\mathcal{M}_i) \subset \mathcal{N}_{i+n}$ and $\phi(f \cdot m) = (-1)^{n \cdot \lvert f \rvert} f \cdot \phi(m)$ for any  $f \in C^{\infty}(\cX)$ and any $m \in \mathcal{M}$.

    In case $\mathcal{M}$ is a dg-$C^{\infty}(\cX)$-module with inner differential $\delta_{\mathcal{M}}$, we define the inner differential on $\mathcal{M}[k]$ by the formula \[\delta_{\mathcal{M}[k]} = (-1)^k \delta_{\mathcal{M}}.\] 
\end{definition}

\begin{definition}
    \label{def:derivation}
    A \emph{degree $n$ $\cin$-derivation} of a dg-algebra $C^{\infty}(\cX)$ with values in a graded $\cin(\cX)$-module $\mathcal{M}$ is a degree $n$ map $d' \colon C^{\infty}(\cX) \to \mathcal{M}$ of graded vector spaces, satisfying the following identities:
    \begin{enumerate}
        \item  The equation $d'(fg) = d'(f) g + (-1)^{n \cdot \lvert f \rvert} f d'(g)$ is satisfied for any elements $f,g \in C^{\infty}(\cX)$.
        \item For any degree 0 elements $a_1, \ldots, a_k \in C^{\infty}(\cX)_0$ and $f \in \cin(\R^k)$ we have \[d'(f(a_1, \ldots, a_k)) = \sum_{i=1}^k \frac{\partial f}{\partial x_i}(a_1, \ldots, a_k) d'(a_i).\]
    \end{enumerate}
\end{definition}

\begin{definition} \label{def:kahler}
    Given a dg-manifold $\cX$, the \emph{module of K\"ahler differentials} of $\cX$ is a pair $\left((\Omega^1_{C^{\infty}(\cX)}, \delta^1_\cX),d\right)$, where $(\Omega^1_{C^{\infty}(\cX)}, \delta^1_\cX)$ is a $C^{\infty}(\cX)$-dg-module and $d\colon C^{\infty}(\cX) \to \Omega^1_{C^{\infty}(\cX)}$ is an $\R$-linear map such that:
\begin{itemize}
    \item $d$ is compatible with inner differentials, i.e. \[d \circ \delta_\cX = \delta^1_\cX \circ d.\]
    
    \item $d$ is a degree 0 $\cin$-derivation of $\cin(\cX)$ with values in $\Omega^1_{C^{\infty}(\cX)}$ (see Definition \ref{def:derivation}).
    
    \item The pair $(\Omega^1_{C^{\infty}(\cX)}, d)$ satisfies the following graded K\"ahler differentials universal property: if $\mathcal{M}$ is a graded $C^{\infty}(\cX)$-module, and $d' \colon C^{\infty}(\cX) \to \mathcal{M}$ is a $\cin$-derivation of degree $n$, then there is a unique map of graded $\cin(\cX)$-modules $\phi \colon \Omega^1_{C^{\infty}(\cX)} \to \mathcal{M}$ of degree $n$ (see Definition \ref{def:shift}), making the following diagram (of graded $\cin(\cX)$-modules) commute:

    \begin{center}
\begin{tikzcd}
C^{\infty}(\cX) \arrow[r, "d'"] \arrow[rd, "d"] & \mathcal{M}                                                   \\
                                                  & \Omega^1_{C^{\infty}(\cX)} \arrow[u, "\phi"', dotted]
\end{tikzcd}
\end{center}
\end{itemize}
\end{definition}

In Proposition \ref{prop:kahler_ex}  we show that the module of K\"ahler differentials exists for any dg-manifold.

Next, we construct the dg-module of vector fields on a dg-manifold $\cX$.

\begin{notation}
    \label{not:gradedhom}
    For any two dg-$\cin(\cX)$-modules $\mathcal{M}$ and $\mathcal{N}$, we denote by $\mathrm{Hom}^{\mathrm{gr}}_{C^{\infty}(\cX)}(\mathcal{M}, \mathcal{N})$ the $\R$-vector space of degree 0 morphisms of graded $C^{\infty}(\cX)$-modules from $\mathcal{M}$ to $\mathcal{N}$: \[\mathrm{Hom}^{\mathrm{gr}}_{C^{\infty}(\cX)}(\mathcal{M}, \mathcal{N}) = \{ f \colon \mathcal{M} \to \mathcal{N} \mid f \text{ is } \cin(\cX)\text{-linear and } f(\mathcal{M}_k) \subseteq \mathcal{N}_k \text{ for all } k \in \Z \}.\] Notice that elements of the space $\mathrm{Hom}^{\mathrm{gr}}_{C^{\infty}(\cX)}(\mathcal{M}, \mathcal{N})$ are not necessarily compatible with inner differentials, i.e. they are not necessarily morphisms of dg-$\cin(\cX)$-modules.
\end{notation}

\begin{definition}
     We define an \emph{internal Hom dg-$\cin(\cX)$-module} $\underline{\mathrm{Hom}}_{C^{\infty}(\cX)} (\mathcal{M}, \mathcal{N})$ as follows:
     
     \begin{itemize}
        \item \emph{As a graded module} \[\underline{\mathrm{Hom}}_{C^{\infty}(\cX)} (\mathcal{M}, \mathcal{N})_n = \mathrm{Hom}^{\mathrm{gr}}_{\cin(\cX)}(\mathcal{M}, \mathcal{N}[n]).\] See Definition \ref{def:shift} for shift conventions.
        \item \emph{The inner differential} is defined by the formula $\delta_{\underline{\mathrm{Hom}}}(f) = \delta_{\mathcal{N}} \circ f - (-1)^{\lvert f \rvert} f \circ \delta_{\mathcal{M}}$ for any $f \in \underline{\mathrm{Hom}}_{C^{\infty}(\cX)} (\mathcal{M}, \mathcal{N})$. 
     \end{itemize}

     This is the standard construction of an internal Hom in homological algebra (see for example \cite[Section 2.7.5]{WEIBEL}).
\end{definition}

\begin{definition}
    \label{def:tangentmodule}
We define the \emph{tangent module} $T_{C^{\infty}(\cX)}$ of a dg-manifold $\cX$ to be the internal Hom module \[T_{C^{\infty}(\cX)} :=\underline{\mathrm{Hom}}_{C^{\infty}(\cX)} (\Omega^1_{C^{\infty}(\cX)}, C^{\infty}(\cX)).\] 

\end{definition}

\begin{remark}
    We identify the tangent module $T_{C^{\infty}(\cX)}$ of $\cX$ with the module $\mathrm{Der}(C^{\infty}(\cX))$ of all $\cin$-derivations of $C^{\infty}(\cX)$ with values in $C^{\infty}(\cX)$ graded by their degree and with the inner differential defined by the graded commutator \[\mathcal{L}_{\delta_\cX}(D) = [\delta_\cX, D]\] for any derivation $D \in \mathrm{Der}(C^{\infty}(\cX))$. This identification comes from the map \[d^*\colon  T_{C^{\infty}(\cX)} \to \mathrm{Der}(C^{\infty}(\cX)), \quad d^*(\phi) = \phi\circ d.\] This map is bijective due to the universal property of K\"ahler differentials. Furthermore, this correspondence is compatible with inner differentials $\delta_{\underline{\mathrm{Hom}}}$ on $T_{C^{\infty}(\cX)}$ and $\cL_{\delta_\cX}$ on $\mathrm{Der}(C^{\infty}(\cX))$. Indeed, for any element $\phi \in T_{C^{\infty}(\cX)}$ we have
    \[
       d^*(\delta_{\underline{\mathrm{Hom}}}\phi) = (\delta_\cX \circ \phi - (-1)^{|\phi|} \phi \circ \delta^1_\cX) \circ d = \delta_\cX \circ d^*(\phi) - (-1)^{|\phi|} d^*(\phi) \circ \delta_\cX = [\delta_\cX, d^*(\phi)].
   \]
\end{remark}

\begin{notation} \label{not:tangent_cotangent}
In what follows we will often write $T_{\cX}$ instead of $T_{C^{\infty}(\cX)}$ and $T^*_{\cX}$ instead of $\Omega^1_{C^{\infty}(\cX)}$ and refer to them as the \emph{tangent and cotangent modules} of the dg-manifold $\cX$.
\end{notation}

\begin{definition}
    \label{def:dgformsmodule}
The \emph{dg-$\cin(\cX)$-module of $p$-forms} on a dg-manifold $\cX$ is defined as $\Omega^p_{C^{\infty}(\cX)} = \bigwedge^p_{\cin(\cX)} \Omega^1_{C^{\infty}(\cX)}$ with the de Rham differential $d \colon \Omega^p_{C^{\infty}(\cX)} \to \Omega^{p+1}_{C^{\infty}(\cX)}$. The details of the construction are presented in Theorem \ref{thm:derham_diff}. The \emph{de Rham complex} of $\cX$ is defined as the collection of dg-$\cin(\cX)$-modules $\Omega^\bullet_{C^{\infty}(\cX)} = \{\Omega^p_{C^{\infty}(\cX)}\}_{p\geq 0}$ together with the de Rham differential $d$. The de Rham complex is a double complex of $\cin(\cX)$-modules, where the first grading is given by form degree $p$ and the second grading is given by the internal grading of the dg-module $\Omega^p_{C^{\infty}(\cX)}$. Furthermore, the assignment $\cX \mapsto (\Omega^\bullet_{C^{\infty}(\cX)}, d)$ is contravariantly functorial, i.e. for any morphism of dg-manifolds $f\colon \cX \to \cY$ there is an induced morphism of double complexes \[f^* \colon (\Omega^\bullet_{C^{\infty}(\cY)}, d) \to (\Omega^\bullet_{C^{\infty}(\cX)}, d).\] See Proposition \ref{prop:pullback_derham} for the proof.
\end{definition}

We now turn to particular examples and compute tangent and cotangent modules for a class of quasi-smooth dg-manifolds, that correspond to sections of trivial vector bundles. This will be the main class of examples we will be working with in the paper. Let us first fix some notation.

\begin{notation}
    \label{not:pullback}
For any dg-manifold $\cX$ with the underlying sheaf of functions $\cO_\cX$ concentrated in non-positive degrees, there is a natural dg-map $\iota_\cX \colon \cX \to \cX^0$ given by the identity on the underlying manifold and the inclusion of sheaves $\iota_\cX^* \colon \mathcal{O}_{\cX^0} \to \mathcal{O}_\cX$ in degree 0. 
In what follows for any module $\mathcal{M}$ over the algebra of smooth functions $\cin(\cX^0)$ we denote by \[\iota_\cX^* \mathcal{M} = C^{\infty}(\cX) \otimes_{C^{\infty}(\cX^0)} \mathcal{M}\] the pullback dg-module over the dg-algebra $\cin(\cX)$ of functions on the dg-manifold $\cX$. The inner differential on $\iota_\cX^* \mathcal{M}$ is inherited from the dg-algebra $\cin(\cX)$: \[\delta(m\otimes f) = m \otimes \delta_\cX(f)\] for any elements $m \in \mathcal{M}$ and $f \in \cin(\cX)$.

Furthermore, given a graded vector space $V$, we denote \[V_\cX =  \cin(\cX) \otimes_{\R} V\] to be the module of sections of the trivial vector bundle over the dg-manifold $\cX$ with fiber $V$. The module $V_\cX$ has a dg-structure with the differential defined by $\delta(f\otimes v) = \delta_\cX(f)\otimes v$ for any function $f \in C^{\infty}(\cX)$ and element $v \in V$, making it into a dg-$\cin(\cX)$-module.
\end{notation}

\begin{definition} \label{def:cone_cocone}
    Given a morphism of dg-$\cin(\cX)$-modules $f\colon (\mathcal{M}, \delta_{\mathcal{M}}) \to (\mathcal{N}, \delta_{\mathcal{N}})$, we define the \emph{cone of the dg-morphism} $f$ as the dg-$\cin(\cX)$-module $\cone(f) = (\mathcal{M}[1] \oplus \mathcal{N}, \delta)$ with the inner differential defined by \[ \delta = \begin{pmatrix} -\delta_{\mathcal{M}} & 0 \\ f & \delta_{\mathcal{N}} \end{pmatrix}. \]
    
    Dually, we define the \emph{cocone of a dg-morphism} $f$ as a dg-$\cin(\cX)$-module $\cocone(f) = (\mathcal{M} \oplus \mathcal{N}[-1], \delta)$ with the inner differential defined by \[ \delta = \begin{pmatrix} \delta_{\mathcal{M}} & 0 \\ f & -\delta_{\mathcal{N}} \end{pmatrix}.\] 
\end{definition}

The following explicit computation of tangent and cotangent modules is a standard result in derived geometry (see for example \cite[Proposition 2.6]{SEOL} for a more general statement). However, the proof via the universal property of K\"ahler differentials is, to the best of our knowledge, new.

\begin{proposition}
  \label{prop:kahler}
    Let $V$ be a finite dimensional vector space, and $\cX$ be a quasi-smooth dg-manifold constructed from a trivial bundle $M\times V$ over some manifold $M$ and a section $\sigma \colon M \to V$ (see Construction~\ref{ctr:zerosection}). Then, the cotangent module $T^*_\cX$ (Notation \ref{not:tangent_cotangent}) of $\cX$ is isomorphic to \begin{equation} \label{eq:cotangent} T^*_\cX \cong \cone(V^*_\cX \xrightarrow{(d\sigma)^*}  \iota^*_\cX T^*_M), \end{equation} where the map $(d\sigma)^*$ is defined by $(d\sigma)^*(\alpha) = d\langle \sigma, \alpha \rangle$ for any element $\alpha \in V^*$ and is extended to the whole dg-module $V^*_\cX$ by $C^{\infty}(\cX)$-linearity. Here $T^*_M$ is the module of one-forms on the underlying manifold $M$. 

    Dually, the tangent module $T_\cX$ is isomorphic to \begin{equation} \label{eq:tangent} T_\cX \cong \cocone(\iota^*_\cX T_M \xrightarrow{-d\sigma} V_\cX) ,\end{equation} where the map $d\sigma$ is dual to $(d\sigma)^*$ and the module $T_M$ is the module of vector fields on $M$.
\end{proposition}
\begin{proof} The proof will consist of three steps: in the first step we will explicitly construct the exterior derivative $d\colon C^{\infty}(\cX) \to T^*_\cX$. In the second step we will check that the pair $(T^*_\cX, d)$ satisfies the universal property from Definition \ref{def:kahler}. Finally, in the third step we will compute the tangent module $T_\cX$ as the internal Hom $\underline{\mathrm{Hom}}_{C^{\infty}(\cX)}(T^*_\cX, C^{\infty}(\cX))$.

\begin{itemize}[leftmargin=*]
    
\item \emph{Constructing the exterior differential.} Recall that the dg-algebra of functions on $\cX$ satisfies $C^{\infty}(\cX) = \cin(M, \bigwedge^{-\bullet} V^*) \cong C^{\infty}(M)\otimes_\R \bigwedge^{-\bullet}V^*$. It is generated as a graded algebra over $C^{\infty}(M)$ by elements of the form $\alpha \in V^*$. 

 On the level of graded $\cin(\cX)$-modules, the cone module $T^*_\cX$ from equation \eqref{eq:cotangent} is given by \[T^*_\cX =  \iota^*_\cX T^*_M \oplus V^*_\cX[1] \cong   (\cin(\cX)\otimes_{\cin(M)} T^*_M) \oplus (\cin(\cX)\otimes V^* [1]).\] In order to distinguish between degree $-1$ elements $\alpha \in V^*$ of $C^{\infty}(\cX)$ and degree $-1$ elements $\tilde{\alpha} = 1\otimes \alpha \in V^*_\cX[1]$ of $T^*_\cX$ we will denote the latter with a tilde.
 
 We construct the exterior derivative $d\colon C^{\infty}(\cX)\to T^*_\cX$ using formula \[d(f\cdot\alpha_1 \wedge \ldots \wedge \alpha_k) = (\alpha_1 \wedge \ldots \wedge \alpha_k) \otimes df +\sum_{i=1}^k (-1)^{k-i} (f\cdot\alpha_1 \wedge \ldots \wedge \hat{\alpha}_i\wedge\ldots \wedge \alpha_{k}\otimes \tilde{\alpha}_i) \] for any smooth function $f \in C^{\infty}(M)$ and elements $\alpha_i \in V^*$. 

 Note that our candidate for the module of K\"ahler differentials of the dg-manifold $\cX$, given by $T^*_\cX = \cone(V^*_\cX \xrightarrow{(d\sigma)^*}  \iota^*_\cX T^*_M)$ is generated as a graded $C^{\infty}(\cX)$-module by elements of $T^*_M$ of degree 0 and $V^*$ of degree -1. 

We denote the interior differential of the cone module $T^*_\cX$ by $\delta^1_{\cX}$ (explicitly it is given by a matrix as in Definition \ref{def:cone_cocone}). According to Definition \ref{def:kahler}, we should check that the exterior derivative commutes with inner differentials \[d \circ \delta_{\cX} = \delta^1_{\cX} \circ d.\] Since the difference $d \circ \delta_\cX - \delta^1_\cX \circ d$ is a derivation of degree $1$ with values in $T^*_\cX$ (the failure of the Leibniz rule for the two composite operators cancels), it suffices to check the equality on generators of $C^{\infty}(\cX)$: for a smooth function $f \in C^{\infty}(M) \subseteq \cin(\cX)_0$  we have $d(\delta_{\cX} f) = 0 = \delta^1_{\cX}(df)$, since $f$ and $df$ have degree 0, and for an element $\alpha \in V^* \subseteq \cin(\cX)_{-1}$ we have \[d(\delta_{\cX} \alpha) = d\langle \sigma, \alpha \rangle = (d\sigma)^*(\alpha) = \delta^1_{\cX}(d\alpha).\]

\item \emph{Showing universal property.} Having constructed the exterior differential $d\colon C^{\infty}(\cX) \to T^*_\cX$, we will now check the universal property of the pair $(T^*_\cX, d)$. Suppose $\mathcal{M} = \{ \mathcal{M}_i \}_{i\in \Z}$ is a graded $C^{\infty}(\cX)$-module, and $v \colon C^{\infty}(\cX) \to \mathcal{M}$ is a derivation of degree $n$. We want to construct a unique map of $C^{\infty}(\cX)$-modules $\phi \colon T^*_\cX \to \mathcal{M}$ of degree $n$ which completes the diagram

\begin{center}
\begin{tikzcd}
C^{\infty}(\cX) \arrow[r, "v"] \arrow[rd, "d"] & \mathcal{M}                                                   \\
                                                  & T^*_\cX \arrow[u, "\phi"', dotted]
\end{tikzcd}
\end{center}
In degree 0 the functions on the dg-manifold $\cX$ are just smooth functions on the manifold $M$, i.e. $\cin(\cX)_0 = \cin(M)$, and $v_0 \colon \cin(M) \to \mathcal{M}_n$ is a $\cin(M)$-derivation. By the universal property of the module of one-forms $T^*_M$ of the manifold $M$ (see Remark \ref{rmk:manifold_kahler}), we get a unique map $\phi_0 \colon T^*_M \to \mathcal{M}_n$ of $\cin(M)$-modules. We then extend it to the map $\phi \colon T^*_\cX \to \mathcal{M}$ by setting \[\phi(\theta)=\phi_0(\theta) \quad \text{for} \quad \theta \in T^*_M\] and \[\phi(\tilde{\alpha}) = v(\alpha) \quad \text{for} \quad \tilde{\alpha} = d\alpha \in V^*_\cX[1] \] and extend to the whole cotangent module $T^*_\cX$ by $\cin(\cX)$-linearity. This extension is well defined, since the degree zero part $\phi_0$ is a $C^{\infty}(M)$-module map, and the derivation $v$ is $\R$-linear.

We are left with checking that the map $\phi$ is unique. Suppose that $\psi \colon T^*_\cX\to \mathcal{M}$ is another graded $C^{\infty}(\cX)$-module map of degree $n$, such that $v=\psi \circ d$. Since the degree 0 part $\phi_0$ is unique by the universal property of one-forms $T^*_M$, we have $\psi_0 = \phi_0$. Let $\alpha \in V^*$ be a generator of the dg-algebra $\cin(\cX)$ and $\tilde{\alpha} = d\alpha \in V^*_\cX[1]$. Then we have $\psi(\tilde{\alpha}) = \psi(d\alpha) = v(\alpha) = \phi(\tilde{\alpha})$. Therefore the maps $\psi$ and $\phi$ coincide on generators of the cotangent module $T^*_\cX$, and hence they are equal.

\item \emph{Computing the tangent module.} To compute the tangent module, note that each element $\partial \in T_\cX = \underline{\mathrm{Hom}}_{C^{\infty}(\cX)}(T^*_\cX, C^{\infty}(\cX))$ is determined by its action on generators of the cotangent module $T^*_\cX$ over the dg-algebra $\cin(\cX)$, i.e. on elements in the module of smooth one-forms $T^*_M$ of degree 0 and on one-forms $V^*$ of derived (internal) degree $-1$. In particular, any element $\partial \in T_\cX$ of degree $2$ and higher must send both types of generators to zero, since the graded algebra $\cin(\cX)$ is concentrated in non-positive degrees, and therefore the vector field $\partial$ must vanish. Let $\partial \in T_{\cX, 1}$ be a derivation of degree $1$. Then it maps smooth one-forms $T^*_M$ to $\cin(\cX)_1 = 0$, and derived degree $-1$ one-forms $V^*$ to $ \cin(\cX)_0 = \cin(M)$. Therefore, $T_{\cX, 1} \cong C^{\infty}(M) \otimes V$, where under this isomorphism an element $1\otimes \xi \in \cin(M) \otimes V$ corresponds to contraction $\iota_\xi \colon \cin(M)\otimes \wedge^{-\bullet} V^* \to \cin(M)\otimes \wedge^{-\bullet + 1} V^*$. A derivation $\partial \in T_{\cX, n}$ of degree $n\leq 0$ maps smooth one-forms $T^*_M$ to $\cin(\cX)_{n} = C^{\infty}(M) \otimes \bigwedge^{-n} V^*$ as a $\cin(M)$-linear map and derived degree $-1$ one-forms $V^*$ to $\cin(\cX)_{n-1} = C^{\infty}(M) \otimes \bigwedge^{-n+1} V^*$. Therefore, \[T_{\cX, n} \cong (T_M \otimes \bigwedge^{-n} V^*) \oplus (C^{\infty}(M) \otimes \bigwedge^{-n+1} V^*)\otimes V.\]

The inner differential $\mathcal{L}_{\iota_\sigma}$ of the tangent module $T_\cX$ acts on the generators $W \in T_M$ of degree 0 and vector fields $\iota_\xi \in T_{\cX , 1}$ of degree 1 (contractions with $\xi \in V$) as follows: \begin{align*} \mathcal{L}_{\iota_\sigma}(W) &= [\iota_\sigma, W]  = -\iota_{W(\sigma)} \quad \text{for } W \in T_M \\
  \mathcal{L}_{\iota_\sigma}(\iota_\xi) &= [\iota_\sigma, \iota_\xi] =0 \quad \text{for } \xi \in V.
\end{align*}

Therefore, both the underlying graded module and the differential on the tangent module $T_\cX$ are isomorphic to those on the cocone module $\cocone(\iota^*_\cX T_M \xrightarrow{-d\sigma} V_\cX)$, which proves that they are isomorphic as dg-$\cin(\cX)$-modules.
\end{itemize}
\end{proof}

\begin{remark} 
    The tangent and cotangent modules $T_\cX$ and $T^*_\cX$ of a quasi-smooth dg-manifold $\cX$ computed in Proposition \ref{prop:kahler} are both graded projective modules over the graded algebra $\cin(\cX)$ (see Definition \ref{def:graded_projective}). This can be seen from the explicit formulas \eqref{eq:cotangent} and \eqref{eq:tangent}. In particular, the tangent module $T_\cX$ is a direct sum of a free module $V_\cX[-1]$ and the pullback $\iota^*_\cX T_M$ of vector fields on the underlying manifold $M$, which is projective as a module over the algebra of smooth functions $\cin(M)$, and hence its pullback is also projective as a module over $\cin(\cX)$. Same reasoning applies also for the cotangent module $T^*_\cX$. 
    
    Furthermore, since the vector space $V$ is finite dimensional, the modules $T_\cX$ and $T^*_\cX$ are finitely generated (see Definition \ref{def:finitely_gen}) as well. Therefore, under the Serre--Swan correspondence (see Theorem \ref{thm:serre_swan}), the modules $T_\cX$ and $T^*_\cX$ correspond to the tangent and cotangent dg-vector bundles over the dg-manifold $\cX$.
\end{remark}

\section{Action groupoid in dg-manifolds} \label{sec:actiongroupoid}

We now switch to the study of singular symplectic reduction. Let $(M, \omega)$ be a symplectic manifold with a Hamiltonian $G$-action and a $G$-equivariant moment map $\mu \colon M \to \fg^*$. We construct a derived version of the zero-level set of the moment map $\mu$, by considering a dg-manifold (see Construction~\ref{ctr:zerosection}) \begin{equation} \label{eq:zerolevelset} \cZ = (M, \cin(M)\otimes_\R \bigwedge^{-\bullet} \fg, \iota_{\mu}) \end{equation} (by abuse of notation we write the global algebra of functions instead of the sheaf). As discussed in Construction~\ref{ctr:zerosection}, this is a model for the homotopy fiber product $M \times^h_{\mu, \fg^*, 0} M$. The algebra $\cin(\cZ)$ is generated over smooth functions $C^{\infty}(M)$ by elements $X \in \fg$ of degree $-1$. The differential acts on these generators by $\iota_{\mu}(X) = \mu^X$, where we write $\mu^X = \langle \mu, X \rangle$.

The reduced space $M_0 = \mu^{-1}(0)/G$ will be modeled as an action groupoid in dg-manifolds. In the case of manifolds, Lie groupoids serve as atlases for differentiable stacks. Therefore, the groupoid corresponding to the $G$-action on the derived zero set $\cZ$ can be considered to be a derived stack model for $M_0$ (see Section~\ref{sec:overview} for discussion).

We first give a general definition of groupoids in dg-manifolds and their nerves, and then specialize to the action groupoid corresponding to the $G$-action on the dg-manifold $\cZ$. The notion of a dg-groupoid was already considered by Mehta (see \cite{QGROUPOIDS}), however, he used a  definition of graded manifolds that contained an issue, which was subsequently corrected by Vysoky in \cite{VYSOKY} (see the Introduction of \cite{VYSOKY} for a discussion). Therefore we think it is important to give a self-contained definition of groupoids in dg-manifolds and their nerves within Vysoky's framework. Another small difference from Mehta's approach is that we construct a groupoid internal to the category of dg-manifolds, while Mehta constructs a groupoid in graded manifolds and then equips it with a multiplicative cohomological vector field. However, these two approaches are equivalent, since a multiplicative cohomological vector field induces a dg-structure on the whole nerve of the graded groupoid, making it into a groupoid in dg-manifolds (see \cite[Theorem 4.1]{QGROUPOIDS}).

\begin{definition}
    We say that a map of dg-manifolds $f=(\underline{f}, f^*)\colon \cX \to \cY$ is a \emph{submersion} if for any point $m \in \cX^0$ the induced map on tangent spaces $T_m f \colon T_m \cX \to T_{\underline{f}(m)} \cY$ is surjective. Here, $T_m \cX$ denotes the graded vector space of derivations of the structure sheaf $\cO_\cX$ at the point $m\in \cX^0$ (see Definition \ref{def:tangent_space}).

    We say that a map of dg-manifolds $f\colon \cX \to \cY$ is a \emph{surjective submersion} if it is a submersion and the underlying map of manifolds $\underline{f}\colon \cX^0 \to \cY^0$ is surjective.
\end{definition}

\begin{remark}
    Our definition of surjective submersion is the same as the definition of \emph{fibration} used to define the structure of a category of fibrant objects on the category of dg-manifolds in \cite{BLX}.
\end{remark}

\begin{definition}
    A \emph{groupoid in dg-manifolds} or a \emph{dg-groupoid} is a groupoid object in the category of dg-manifolds, such that the source and target maps are surjective submersions. More explicitly, it consists of a pair of dg-manifolds $\cG$ and $\cX$, together with surjective submersions $s,t \colon \cG \to \cX$, a multiplication map $m\colon \cG\times_{s, \cX, t} \cG \to \cG$ (this fiber product exists by Proposition \ref{prop:dg_fiber_products}), a unit map $u \colon \cX \to \cG$ and an inverse map $i\colon \cG \to \cG$, satisfying the usual groupoid axioms.
\end{definition}

\begin{construction} \label{ctr:actiongroupoid_general}
Let $G$ be a Lie group and $\cX$ be a dg-manifold with a $G$-action. This means that there is a map of dg-manifolds $a\colon G\times \cX \to \cX$ such that the following diagrams commute: \begin{center}
\begin{tikzcd}G\times G \times \cX \arrow[r, "m_G\times \mathrm{id}"] \arrow[d, "\mathrm{id}\times a"] & G\times \cX \arrow[d, "a"] \\
G\times \cX \arrow[r, "a"] & \cX
\end{tikzcd}
\end{center}
and \begin{center}
\begin{tikzcd}\{e\} \times \cX \arrow[r, "u_G\times \mathrm{id}"] \arrow[rd, "\cong"'] & G\times \cX \arrow[d, "a"] \\
& \cX   
\end{tikzcd}
\end{center}
where $m_G\colon G\times G \to G$ is the group multiplication map and $u_G\colon \{e\} \to G$ is the unit map. Then there exists a groupoid in dg-manifolds $\cX/G = \{G\times \cX \rightrightarrows \cX\}$, called the \emph{action groupoid} (the notation refers to the fact that the groupoid $\cX/G$ models the quotient stack $[\cX/G]$). The source and target maps are given by $s = \mathrm{pr}_\cX$ and $t = a$. To define the groupoid multiplication map $m \colon (G\times \cX) \times_\cX (G\times \cX) \to G\times \cX$ we first take a map $\overline{m}\colon G\times \cX \times G \times \cX \to G\times \cX$ defined as multiplication on the $G$-components and projection on the $\cX$-components. There is a natural inclusion of the fiber product into the product as a dg-submanifold (see Proposition \ref{prop:dg_fiber_products}). Therefore we can restrict the map $\overline{m}$ to the fiber product to get the groupoid multiplication $m$. The unit and inverse maps come from the corresponding maps on the Lie group $G$. 
\end{construction}

\begin{construction} \label{ctr:actiongroupoid}
Let $\cZ$ be the derived zero set of the moment map $\mu \colon M \to \fg^*$ as in equation \eqref{eq:zerolevelset}. The $G$-action on $\cZ$ is defined in terms of the $G$-action on $M$ and the adjoint action on the Lie algebra $\fg$. More explicitly, we define the action map $a = (\underline{a}, a^*)\colon G\times \cZ \to \cZ$ as follows. 

\begin{itemize}[leftmargin=*] \item The underlying map $\underline{a} \colon G \times M \to M$ is the ordinary action of the Lie group $G$ on the manifold $M$. 
    \item The pullback map $a^* \colon \cin(\cZ) \to \cin(G \times \cZ)$ is defined by its action on generators $X \in \fg$ of $\cin(\cZ)$. The element $a^*(X) \in \cin(G\times \cZ)_{-1} = \cin(G \times M) \otimes \fg$ is given by \[a^*(X)(g,m) = \mathrm{Ad}_{g^{-1}} X.\]  In order to check that the pullback $a^*$ commutes with differentials, it is enough to check that the equation $a^* \circ \iota_{\mu} = \iota_{\mu} \circ a^*$ (here, abusing the notation, we denote the inner differential on $G\times \cZ$ also by $\iota_\mu$) holds on generators $X \in \fg$. We have $a^* \iota_{\mu}(X)(g,m) = a^* \mu^X (g,m) = \mu^X(gm)$. On the other hand $\iota_{\mu} a^*(X)(g,m) = \mu^{\mathrm{Ad}_{g^{-1}}X}(m) = \mu^X(gm)$ by $G$-equivariance of the moment map $\mu$.
\end{itemize}

It is straightforward to check that $a$ defines a $G$-action. Therefore, we can apply Construction~\ref{ctr:actiongroupoid_general} to get the action groupoid \[\cZ/G = \{G\times \cZ \rightrightarrows \cZ\}\] corresponding to the $G$-action $a$ on the derived zero set $\cZ$.
\end{construction}

Next, we define the nerve of a groupoid in dg-manifolds. The notion of a nerve of an internal category goes back to Grothendieck (see \cite{GROTHENDIECK}).

\begin{definition} \label{def:nerve}
    Let $(\cG \rightrightarrows \cX, s, t, m, u, i)$ be a groupoid in dg-manifolds. Its \emph{nerve} is the simplicial dg-manifold $N_\bullet(\cG) \colon \Delta^{\mathrm{op}} \to \catname{dgMan}$ defined as follows.
    \begin{itemize}
        \item In degree $0$: $N_0(\cG) = \cX$.
        \item In degree $1$: $N_1(\cG) = \cG$.
        \item In degree $n \geq 2$: $N_n(\cG) = \cG^{(n)}$, the $n$-fold iterated fiber product
        \[
            \cG^{(n)} = \underbrace{\cG \times_{s,\cX,t} \cG \times_{s,\cX,t} \cdots \times_{s,\cX,t} \cG}_{n},
        \]
        which exists by repeated application of Proposition \ref{prop:dg_fiber_products}. We denote by $\mathrm{pr}_k \colon \cG^{(n)} \to \cG$ the $k$-th projection morphism, for $1 \leq k \leq n$.
    \end{itemize}
    The \emph{face maps} $\partial_k \colon N_n(\cG) \to N_{n-1}(\cG)$ for $n\geq 1$ and $0 \leq k \leq n$ are the morphisms of dg-manifolds defined for $n = 1$  by

    \[\partial_0 = s, \quad \partial_1 = t\]

    and for $n \geq 2$ via the universal property of the fiber product by the following formulas:
    \[
        \partial_k = \begin{cases}
            (\mathrm{pr}_2, \ldots, \mathrm{pr}_n) & k = 0, \\
            (\mathrm{pr}_1, \ldots, \mathrm{pr}_{k-1}, m \circ (\mathrm{pr}_k, \mathrm{pr}_{k+1}), \mathrm{pr}_{k+2}, \ldots, \mathrm{pr}_n) & 1 \leq k \leq n-1, \\
            (\mathrm{pr}_1, \ldots, \mathrm{pr}_{n-1}) & k = n,
        \end{cases}
    \]
    where for $k=0$ and $k=n$ the result lands in the $(n-1)$-th fiber product $\cG^{(n-1)}$ because projections $\mathrm{pr}_k$ already satisfy $s \circ \mathrm{pr}_k = t \circ \mathrm{pr}_{k+1}$ by construction of the iterated fiber product, and for $1 \leq k \leq n-1$ the compatibilities $s \circ m \circ (\mathrm{pr}_k, \mathrm{pr}_{k+1}) = s \circ \mathrm{pr}_{k+1}$ and $t \circ m \circ (\mathrm{pr}_k, \mathrm{pr}_{k+1}) = t \circ \mathrm{pr}_k$ follow from the groupoid axioms.

    The \emph{degeneracy maps} $\varsigma_k \colon N_n(\cG) \to N_{n+1}(\cG)$ for $0 \leq k \leq n$ are the morphisms of dg-manifolds defined by
    \[
        \varsigma_k = \begin{cases}
            (u \circ t \circ \mathrm{pr}_1,\, \mathrm{pr}_1, \ldots, \mathrm{pr}_n) & k = 0, \\
            (\mathrm{pr}_1, \ldots, \mathrm{pr}_k,\, u \circ s \circ \mathrm{pr}_k,\, \mathrm{pr}_{k+1}, \ldots, \mathrm{pr}_n) & 1 \leq k \leq n,
        \end{cases}
    \]
    where the inserted unit $u \circ s \circ \mathrm{pr}_k$ (resp.\ $u \circ t \circ \mathrm{pr}_1$) is composable with its neighbours because $s \circ u = \mathrm{id}_\cX = t \circ u$ by the unit axiom. All face and degeneracy maps are morphisms of dg-manifolds, and they satisfy the simplicial identities by the groupoid axioms.
\end{definition}

\begin{construction}
    \label{ctr:actiongroupoid_nerve}
    Let $\cX/G$ be the action groupoid from Construction~\ref{ctr:actiongroupoid_general}. Then there is a natural identification \[N_n(\cX/G) \cong G^n \times \cX.\] It is constructed as follows. Let $\alpha_k \colon G^n \times \cX \to \cX$ for $ 1 \leq k \leq n+1$ be a family of maps defined recursively by \[\alpha_{n+1} = \mathrm{pr}_\cX \quad \text{and} \quad \alpha_k = a \circ (\mathrm{pr}_k, \alpha_{k+1}) \quad \text{for } 1 \leq k \leq n.\] Define the map $f_k = (\mathrm{pr}_k, \alpha_{k+1}) \colon G^n \times \cX \to G \times \cX$ for $1 \leq k \leq n$. Since $s \circ f_k = \alpha_{k+1}$ and $t \circ f_{k+1} = \alpha_{k+1}$, the map $\phi_n = (f_1, \ldots, f_n) \colon G^n \times \cX \to N_n(\cX/G) = (G\times \cX)^{(n)}$ is well defined.

    Conversely, denoting by $\mathrm{pr}^G_k$ and $\mathrm{pr}^\cX_k$ the $G-$ and $\cX$-components of the $k$-th projection $\mathrm{pr}_k \colon (G\times \cX)^{(n)} \to G\times \cX$, there is an inverse map \[ \psi_n := (\mathrm{pr}^G_1, \ldots, \mathrm{pr}^G_n, \mathrm{pr}^\cX_n) \colon N_n(\cX/G) \to G^n \times \cX.\] 

    Furthermore, $G^{\bullet} \times \cX$ has the natural structure of a simplicial dg-manifold, where $\partial_0$ drops the first $G$-factor, $\partial_k$ for $1 \leq k \leq n-1$ multiplies the $k$-th and $(k+1)$-th $G$-factors, and $\partial_n$ applies the action of the last $G$-factor on $\cX$. The degeneracy maps insert the unit element of $G$ in the corresponding position. It is straightforward to check that the maps $\phi_n$ are compatible with face and degeneracy maps, and hence they define an isomorphism of simplicial dg-manifolds $N_\bullet(\cX/G) \cong G^{\bullet} \times \cX$.
\end{construction}

\begin{definition} \label{def:cochains}
    Given a groupoid in dg-manifolds $\cG \rightrightarrows \cX$, we define a \emph{complex of groupoid cochains} $(C^{\bullet}(\cG), \partial)$ to be a cochain complex of dg-vector spaces whose terms are the dg-algebras $C^p(\cG) = C^{\infty}(N_p(\cG))$, and \[\partial^* = \sum_{i=0}^{p+1} (-1)^i \partial_i^* \colon C^p(\cG) \to C^{p+1}(\cG).\] 
\end{definition}

\section{Dg-Lie algebroid of an action groupoid} \label{sec:dgalg}

In this section we compute the Lie algebroid of the groupoid $\cZ/G$ defined in Construction~\ref{ctr:actiongroupoid}. It will play an essential role in the construction of tangent and cotangent complexes of the dg-groupoid $\cZ/G$. We will use the following definition, which we take from Nuiten \cite[Definition 3.1.1]{NUITEN}. It is a dg-version of the classical notion of a Lie--Rinehart algebra.

\begin{definition}
    A dg-Lie algebroid $(A,\delta_A)$ over a dg-manifold $(\cX,\mathcal{O}_\cX)$ is a dg-$C^{\infty}(\cX)$-module, equipped with a $\mathbb{R}$-linear dg-Lie algebra structure and an anchor map $\rho \colon A \to T_{\cin(\cX)}$ which is both a map of $C^{\infty}(\cX)$-modules and dg-Lie algebras, such that \[[X,a\cdot Y] = (-1)^{|X||a|}a[X,Y] + \rho(X)(a) \cdot Y \] for all elements $X,Y \in A$ and functions $a \in C^{\infty}(\cX)$.
\end{definition}

In order to define the dg-Lie algebroid of a groupoid in dg-manifolds, we will need the following notion.

\begin{definition}
    Given a map of dg-manifolds $f\colon \cX \to \cY$, a \emph{$f$-derivation} of the dg-algebra $\cin(\cY)$ with values in the target dg-algebra $\cin(\cX)$ is an $\mathbb{R}$-linear map $\alpha \colon \cin(\cY) \to \cin(\cX)$ such that \[\alpha(g\cdot h) = \alpha(g)f^*(h) + (-1)^{|\alpha||g|}f^*(g)\alpha(h)\] for all $g,h \in \cin(\cY)$. We denote the space of $f$-derivations by $\mathrm{Der}_f(\cin(\cY), \cin(\cX))$.
\end{definition}

\begin{remark}
    A careful reader might notice that we do not impose the chain rule on $f$-derivations, unlike in Definition \ref{def:derivation}. In fact, it can be shown that the chain rule is automatically satisfied for $f$-derivations using a local description similar to that of \cite[Proposition 4.14]{VYSOKY}. However, this will play no role in what follows.
\end{remark}

The following construction appeared in the work of Mehta \cite[Section 3.2]{MEHTATHESIS}. This is a dg-analogue of the construction of the Lie-Rinehart algebra of a Hopf algebroid described in \cite[Section~5]{HOPF}.

\begin{construction}[dg-Lie algebroid of a groupoid in dg-manifolds] \label{ctr:dgalgebroid}
    Given a dg-groupoid $\cG \rightrightarrows \cX$ there is a $\cin(\cX)$-module $A$ defined as the space of $u$-derivations of the dg-algebra $\cin(\cG)$ which vanish along the source map $s$: \[ A = \mathrm{Der}^s_u(\cin(\cG), \cin(\cX)) \] where \[ \mathrm{Der}^s_u(\cin(\cG), \cin(\cX)) = \left\{ \alpha : \cin(\cG) \to \cin(\cX) \ \middle|\begin{array}{l} \alpha \text{ is a }  u\text{-derivation}, \\[6pt] \alpha(s^*h) = 0 \text{ for any } h \in \cin(\cX). \end{array} \right\}\]

    The differential $\delta_A$ on the dg-module $A$ is defined by \[\delta_A(\alpha)(f) =\delta_\cX(\alpha(f)) - (-1)^{|\alpha|}\alpha(\delta_\cG(f)).\] It preserves the subspace of $u$-derivations vanishing along the source map $s$, since both pullback maps $s^*$ and $u^*$ commute with differentials. Furthermore, there is a dg-$\cin(\cX)$-module map $\rho \colon A \to T_{\cin(\cX)}$ defined by $\rho(\alpha)(f) = \alpha(t^*f)$ for $f \in \cin(\cX)$. 
    
The $\cin(\cX)$-module $A$ carries a natural structure of a dg-Lie algebroid over the dg-manifold $\cX$, with anchor map $\rho$ defined above. The bracket on the dg-module $A$ is given on $u$-derivations $\alpha$ and $\beta$ by
\[
    [\alpha, \beta](f) = \alpha(\tilde\beta(m^*f)) - (-1)^{|\alpha||\beta|}\beta(\tilde\alpha(m^*f)),
\]
where $\tilde\alpha, \tilde\beta \colon \cin(\cG^{(2)}) \to \cin(\cG)$ are unique $\varsigma_1$-derivations (here $\varsigma_1$ is the degeneracy map from Definition \ref{def:nerve}) which satisfy \[\tilde\alpha \circ \mathrm{pr}_1^* = 0 \quad \text{ and } \quad \tilde\alpha \circ \mathrm{pr}_2^* = t^* \circ \alpha,\] and similarly for $\tilde\beta$ (see \cite[Section 3.2.1]{MEHTATHESIS}).
\end{construction}

\begin{remark}
    \label{rmk:smooth_algebroid}
    The algebroid $A$ constructed above can be identified with the pullback $u^* \ker(ds)$ along the unit, where $ds \colon T_\cG \to s^* T_\cX$ is the tangent map of the source. To make this precise, there is a $\cin(\cX)$-module map
    \[
        \phi \colon T_\cG \otimes_{\cin(\cG)} \cin(\cX) \to \mathrm{Der}_u(\cin(\cG), \cin(\cX)),
        \qquad \phi([X \otimes f]) = f \cdot (u^* \circ X),
    \]
    which restricts to a map
    \[
        \phi \colon u^* \ker(ds) \to \mathrm{Der}^s_u(\cin(\cG), \cin(\cX)) = A.
    \]
    In the smooth case, the left side is the pullback of the module of sections of $\ker(ds)$, while the right side is the module of sections of the pullback bundle $u^* \ker(ds)$ and the map $\phi$ is an isomorphism by a standard argument (see e.g.\ \cite[Theorem 12.42]{NESTRUEV}). An analagous line of reasoning applies also in the dg-setting, however, since we will not need this result in what follows, we will omit the details.
\end{remark}

\begin{notation}
    \label{not:dg_alg_not}
Before proceeding, let us fix some notation. Let $\cZ$ be the derived zero-level set of the moment map $\mu$ with global dg-algebra of functions given by $\cin(\cZ) = \cin(M) \otimes  \wedge^{-\bullet} \fg$ (see \eqref{eq:zerolevelset}). Recall from Proposition~\ref{prop:kahler} that \[T^*_\cZ \cong \cone(\fg_\cZ \xrightarrow{(d\mu)^*} \iota^*_\cZ T^*_M ) = \fg_\cZ[1] \oplus  \iota^*_\cZ T^*_M, \text{ and } T_{\cZ} \cong \cocone(\iota^*_\cZ T_M \xrightarrow{-d\mu} \fg^*_\cZ) = \iota^*_\cZ T_M \oplus \fg^*_\cZ[-1].\] Here $\fg_\cZ$ and $\fg^*_\cZ$ denote the trivial dg-modules $\cin(\cZ) \otimes \fg$ and $\cin(\cZ) \otimes \fg^*$, respectively, while $\iota^*_\cZ T_M$ and $\iota^*_\cZ T^*_M$ denote the pullbacks of the module of vector fields $T_M$ and the module of one-forms $T^*_M$ on $M$ along the inclusion $\iota_\cZ \colon \cZ \to M$ (see Notation~\ref{not:pullback}). 

Given an element $X \in \fg$ we will denote by $dX\in T^*_{\cZ, -1}$ the one-form of derived degree $-1$  which is the image of $X$ under the exterior differential $d \colon \cin(\cZ)\to T^*_\cZ$. With this notation, the cotangent module $T^*_\cZ$ is generated over the dg-algebra of functions $\cin(\cZ)$ by smooth one-forms $\theta \in T^*_M$ and derived degree -1 one-forms $dX$ for $X\in \fg$. 

Given a covector $\sigma \in \fg^*$ we will denote by $\iota_\sigma \in T_{\cZ , 1}$ the degree 1 vector field on $\cZ$ given by contraction with $\sigma$. In particular $\iota_\sigma$ acts on derived degree -1 elements $X \in \fg \subseteq \cin(\cZ)_{-1}$ by $\iota_\sigma(X) = \langle \sigma, X \rangle$. The dg-module $T_{\cZ}$ is generated over the dg-algebra $\cin(\cZ)$ by smooth vector fields $W \in T_M$ and derived degree 1 vector fields $\iota_\sigma$ for $\sigma \in \fg^*$.
\end{notation}
\begin{proposition}
    \label{prop:algebroid}
    The underlying dg-$\cin(\cZ)$-module of the dg-Lie algebroid of the action groupoid $\cZ/G$ is isomorphic to $\cin(\cZ) \otimes \fg$. The anchor map \begin{equation} \label{eq:anchor} \rho \colon \cin(\cZ) \otimes \fg \to T_{\cZ} = \iota^*_\cZ T_M \oplus \fg^*_\cZ[-1] \end{equation} is given by \[\rho = \rho_0 + \eta,\] where \begin{equation} \label{eq:anchor_zero} \rho_0 \colon \cin(\cZ) \otimes \fg \to \iota^*_\cZ T_M\end{equation} is the infinitesimal action $\fg \to T_M$ pulled back along the inclusion $\iota_\cZ \colon \cZ \to M$, and \begin{equation} \label{eq:eta} \eta \colon \cin(\cZ) \otimes \fg \to \fg^*_\cZ[-1]\end{equation} is a chain homotopy between $-d\mu \circ \rho_0$ and $0$ given by $\eta(1\otimes X) = \mathrm{ad}^*_X$ and extended $\cin(\cZ)-linearly$, where $\mathrm{ad}^*_X$ is viewed as an element of the space $\fg \otimes \fg^*\subseteq \cin(\cZ)_{-1} \otimes \fg^*$.
\end{proposition}
\begin{proof}
\begin{itemize}[leftmargin=*]
\item \emph{Computing the Lie algebroid module $A$}. In order to show that $A \cong \cin(\cZ) \otimes \fg$, we first establish that any $u$-derivation $\alpha \in A$ is determined by its restriction to the subalgebra $\cin(G \times M) \subseteq \cin(G \times \cZ)$ of degree zero elements.

Since $\cin(G \times \cZ) = \cin(G \times M) \otimes \wedge^{-\bullet}\mathfrak{g}$, every function $f \in \cin(G \times \cZ)$ can be written as a finite sum
\[
    f = \sum_{I} f_I X_I, \qquad f_I \in \cin(G \times M),
    \quad X_I = X_{i_1} \wedge \cdots \wedge X_{i_{|I|}} \in \wedge^{|I|}\mathfrak{g},
\]
where the generators $X_i \in \mathfrak{g}$ have degree $-1$. The generators $X_I$ lie in the subspace $s^*\cin(\cZ)$ since the source map $s = \mathrm{pr}_\cZ$ acts as the identity on the $\mathfrak{g}$-part. By the Leibniz rule and the vanishing $\alpha(s^*\cin(\cZ)) = 0$,
\begin{equation} \label{eq:algebroid_derivations}
    \alpha(f_I X_I) = \alpha(f_I) \cdot u^*(X_I) = \alpha(f_I) \cdot X_I,
\end{equation}
so the derivation $\alpha$ is determined by its restriction $\alpha\vert_{\cin(G \times M)}$ to the subalgebra $\cin(G \times M) \subseteq \cin(G \times \cZ)$ in degree zero. On this subalgebra, $\alpha\vert_{\cin(G \times M)}$ is an ordinary $u$-derivation of $\cin(G \times M)$ with values in the dg-algebra $\cin(\cZ)$, vanishing on the subspace $1 \otimes \cin(M)$ and therefore it corresponds to a section of the smooth Lie algebroid $M \times \fg$ with values in $\cin(\cZ)$ (see Remark \ref{rmk:smooth_algebroid}).

Now, the module of sections of the trivial bundle $M \times \fg$ can be identified with $\cin(M) \otimes \fg$, and therefore \[A \cong (\cin(M)\otimes \fg) \otimes_{\cin(M)} \cin(\cZ) \cong \cin(\cZ) \otimes \fg.\]

Finally, we establish the differential $\delta_A$ on $A$. Using equation \eqref{eq:algebroid_derivations} we compute: \begin{align*}\delta_A(\alpha)(f_I X_I) &= \iota_\mu(\alpha(f_I) X_I) - (-1)^{|\alpha|} \alpha(f_I \iota_\mu(X_I))\\ &= (\iota_\mu \alpha(f_I))X_I + (-1)^{|\alpha|} \alpha(f_I) \iota_\mu(X_I) - (-1)^{|\alpha|} \alpha(f_I) \iota_\mu(X_I)\\ &= (\iota_\mu \alpha(f_I))X_I.\end{align*} Therefore, under the identification $A \cong \cin(\cZ) \otimes \fg$, the differential $\delta_A$ is just $\iota_\mu$ on $\cin(\cZ)$ trivially extended to $\cin(\cZ) \otimes \fg$. 

\item \emph{Computing the anchor map $\rho \colon A \to T_\cZ$}. We compute the anchor map $\rho \colon \cin(\cZ) \otimes \fg \to T_\cZ$ by evaluating it on generators of the dg-$\cin(\cZ)$-module $\cin(\cZ) \otimes \fg$. Let $\{\overline{E}_i\}$ be a left-invariant frame on the Lie group $G$ so that $\{E_i = \overline{E}_i(e)\}$ is a basis of the Lie algebra $\fg$, and $\{\sigma_i\}$ is the corresponding dual basis of $\fg^*$. Since the anchor $\rho$ is a $\cin(\cZ)$-module map, it is determined by its values on elements of the form $1 \otimes E_i$.

Under the correspondence $A \cong \cin(\cZ) \otimes \fg$ established in the first part of the proof, the element $1 \otimes E_i \in \cin(\cZ) \otimes \fg$ corresponds to the $u$-derivation $\alpha_i \in A$ given on functions $f \in \cin(G \times \cZ)$ by
\[ \alpha_i(f) \;=\; (\overline{E}_i \otimes \mathrm{id})(f) \big|_\cZ, \]
i.e. the action of the vector field $\overline{E}_i \otimes \mathrm{id} \in T_{G \times \cZ}$ on $f$, restricted along the unit $u \colon \cZ \to G \times \cZ$. (One verifies directly that $\alpha_i$ is a $u$-derivation vanishing on the subspace $s^*\cin(\cZ)$, and that on the degree-zero subalgebra $\cin(G \times M)$ it recovers the section $1 \otimes E_i$ of the smooth Lie algebroid $M \times \fg$.)

The anchor sends the $u$-derivation $\alpha_i$ to the derivation $\rho(1 \otimes E_i) \in T_\cZ$ defined by
\[ \rho(1 \otimes E_i)(f) \;=\; \alpha_i(t^*f) \;=\; (\overline{E}_i \otimes \mathrm{id})(t^*f) \big|_\cZ, \]
for $f \in \cin(\cZ)$. We compute the action of $\rho(1\otimes E_i)$ on generators of the dg-algebra $\cin(\cZ)$ of degree 0 and -1 separately.

For degree 0 generators $f \in \cin(M) = \cin(\cZ)_0$, we have $t^*f(g, m) = f(g m)$, so we recover the smooth case:
\[ \rho(1 \otimes E_i)(f) \;=\; (\overline{E}_i \otimes \mathrm{id})(t^*f) \big|_\cZ \;=\; \rho_0(1 \otimes E_i)(f), \]
where $\rho_0 \colon \cin(\cZ) \otimes \fg \to \iota_\cZ^* T_M \subseteq T_\cZ$ is the infinitesimal action of the Lie algebra $\fg$ on the module of smooth vector fields $T_M$ on $M$, extended $\cin(\cZ)$-linearly.

For degree $-1$ generators $E_j \in \fg \subseteq \cin(\cZ)_{-1}$, write the adjoint action as \[\mathrm{Ad}_g = \sum_{k,l} c_{kl}(g) E_k \otimes \sigma_l,\] so that
\[
    t^* E_j \;=\; \sum_k c_{kj}(g^{-1}) \, E_k
\]
as an element of $\cin(G) \otimes \fg \subseteq \cin(G \times \cZ)_{-1}$. Since the vector field $\overline{E}_i \otimes \mathrm{id}$ acts trivially on the $\fg$-part of the dg-algebra $\cin(G \times \cZ) = \cin(G \times M) \otimes \wedge^{-\bullet}\fg$, we have:
\[ (\overline{E}_i \otimes \mathrm{id})(t^* E_j)
    \;=\; \sum_k \overline{E}_i\bigl(c_{kj}(g^{-1})\bigr) E_k.\]
Restricting along the unit map $u$ (i.e.\ evaluating at $g = e$) gives
\[\rho(1 \otimes E_i)(E_j) \;=\; \sum_k \bigl(-\overline{E}_i c_{kj}\bigr)(e) \, E_k. \] Notice that coefficients $\overline{E}_i c_{kj}(e)$ are exactly the structure constants $c^k_{ij}$ of the Lie algebra $\fg$. Putting together all contributions, we obtain
\[\rho(1\otimes E_i) \; = \; \rho_0(1\otimes E_i)  - \sum_{k,j} c^k_{ij} E_k \otimes \iota_{\sigma_j}.\] Under the identification $\fg \otimes \fg^* \cong \mathfrak{gl}(\fg^*)$ given by $(v \otimes \xi)(\sigma) =\langle \sigma, v \rangle \xi$, we get \[\rho(1\otimes E_i) \; = \; \rho_0(1\otimes E_i) + \mathrm{ad}^*_{E_i}\] where the endomorphism $\mathrm{ad}^*_{E_i}$ of the dual space $\fg^*$ is viewed as an element of the space $\fg \otimes \fg^* \subseteq \cin(\cZ)_{-1} \otimes \fg^* \subseteq T_{\cZ, 0}$. Therefore, we have established the formula \[\rho = \rho_0 + \eta\] with $\eta(X) = \mathrm{ad}^*_X$.

It remains to verify that the map $\eta$ is a chain homotopy between $-d\mu \circ \rho_0$ and $0$, i.e. that \[-d\mu \circ \rho_0 = \mathcal{L}_{\iota_\mu} \circ \eta +\eta \circ \delta_A.\] Notice that for $1 \otimes X \in \cin(\cZ) \otimes \fg$ we have $\delta_A(1 \otimes X) = 0$ and under our conventions $\mathcal{L}_{\iota_\mu}(\mathrm{ad}^*_X) = - \mathrm{ad}^*_X(\mu)$. Therefore, the homotopy equation reduces to
\[
    -\iota_{\rho_0(1 \otimes X)} \, d\mu = -\mathrm{ad}^*_X(\mu),
\]
which is the infinitesimal form of the $G$-equivariance of the moment map (see e.g.\ \cite[Lemma~5.2.5]{MCDUFF}).
\end{itemize}
\end{proof}

\begin{remark}
    \label{rmk:BRST}

    Any dg-Lie algebroid $A$ over a dg-manifold $\cX$ gives rise to a bigraded Chevalley--Eilenberg complex $\mathrm{CE}(A) = \wedge^\bullet_{\cin(\cX)} A^*$ (see \cite[Definition 4.1.16]{NUITEN}). In the case of the dg-Lie algebroid $A$ of the action groupoid $\cZ/G$, we have $A^* \cong \cin(\cZ) \otimes \fg^*$, so the Chevalley--Eilenberg complex $\mathrm{CE}(A)$ is isomorphic to the bigraded complex $\cin(\cZ) \otimes \wedge^\bullet \fg^* \cong \cin(M) \otimes \wedge^{-\bullet} \fg \otimes \wedge^\bullet \fg^*$ with the Koszul differential $\iota_\mu$ and the Chevalley--Eilenberg differential $d_{CE}$ corresponding to the Lie algebra $\fg$. This bigraded complex coincides with the classical BRST complex of the Hamiltonian $G$-space as defined by Kostant and Sternberg \cite{KOSTANTBRST}. In this sense, the BRST complex is the infinitesimal counterpart of the action groupoid $\cZ/G$. This point of view will be developed further in Remark~\ref{rmk:BRST_form}.
    
    The complex $\mathrm{CE}(A)$ can also be thought of as a dg NQ-manifold $(M, \wedge^\bullet A^*, d_{\mathrm{CE}})$ in the sense of Pridham \cite{PRIDHAM}. In the case of an algebroid of a Hamiltonian $G$-space $A \cong \cin(\cZ) \otimes \fg$, the corresponding dg NQ-manifold is what Pridham calls the derived infinitesimal quotient $\cZ/\fg$ (see \cite[Example 1.38]{PRIDHAM}).

\end{remark}

We will now discuss the $G$-equivariance of the anchor map $\rho$ of the dg-Lie algebroid of $\cZ/G$. This will be important later when we construct the tangent and cotangent complexes of $\cZ/G$.

\begin{construction}($G$-action)
    \label{ctr:gaction}
Let $a_g \colon \cZ \to \cZ$ denote the action of an element $g \in G$ on $\cZ$. Define the left action $\Phi_g = (a_{g^{-1}})^*$ on the dg-algebra $\cin(\cZ)$. It is given on degree 0 functions by $\Phi_g(f) = f \circ a_{g^{-1}}$ for $f \in \cin(M)$, and on degree $-1$ functions by $\Phi_g(X) = \mathrm{Ad}_{g} X$ for $X \in \fg$. 

The action $\Phi_g$ on $\cin(\cZ)$ induces an action of $G$ on the module of derivations $T_\cZ$ by conjugation: \[g \cdot D = \Phi_g \circ D \circ \Phi_{g^{-1}} = (a_{g^{-1}})^* \circ D \circ (a_g)^* \] for $D \in T_\cZ$. We evaluate this action on two type of generators from Notation \ref{not:dg_alg_not}. For a smooth vector field $W \in T_M$, we have \[
    (g \cdot W)(f) = W(f \circ a_g) \circ a_{g^{-1}} = g_* W(f), \] which is the standard pushforward action of $G$ on vector fields. For a degree 1 contraction $\iota_\sigma$ with $\sigma \in \fg^*$ and an element $X \in \fg \subseteq \cin(\cZ)_{-1}$, we have \[g\cdot \iota_\sigma(X) = \Phi_g(\iota_\sigma(\mathrm{Ad}_{g^{-1}} X)) = \langle \sigma, \mathrm{Ad}_{g^{-1}} X \rangle = \iota_{\mathrm{Ad}^*_{g} \sigma}(X)\] with the middle $\Phi_g$ acting trivially since the value is a scalar. Therefore, $g\cdot \iota_\sigma = \iota_{\mathrm{Ad}^*_{g} \sigma} \in \fg^*_\cZ[-1]$. To summarize, the action of $G$ on $T_\cZ$ is given by the pushforward action on the smooth vector fields in $\iota^*_\cZ T_M$ and the coadjoint action on the degree 1 contractions in $\fg^*_\cZ[-1]$.

Recall that the dg-Lie algebroid of $\cZ/G$ is by definition $A = \mathrm{Der}_u^s(\cin(G \times \cZ), \cin(\cZ))$. The dg-manifold $G \times \cZ$ carries a left $G$-action given by $C_g = (\mathrm{Ad}_g, a_g)$. With this action the source and target maps of $\cZ/G$ become $G$-equivariant, i.e. $s\circ C_g = a_g \circ s$ and $t \circ C_g = a_g \circ t$. Denote by $\Psi_g = (C_{g^{-1}})^*$ the induced action on the dg-algebra $\cin(G \times \cZ)$. Then we define a $G$-action on $A$ by \[g \cdot \alpha = \Phi_g \circ \alpha \circ \Psi_{g^{-1}} = (a_{g^{-1}})^* \circ \alpha \circ (C_g)^*.\] Notice that restricting this action to the degree zero component $\cin(M) \otimes \fg$ of $A$ recovers the standard adjoint action of the action groupoid $G \times M$ on its Lie algebroid $M \times \fg$. Therefore, for $1\otimes X \in \cin(\cZ) \otimes \fg$ we have $g \cdot (1\otimes X) = 1 \otimes \mathrm{Ad}_g X$. This uniquely determines the $G$-action on $A$ by $\cin(\cZ)$-linearity, i.e. $g \cdot (f \otimes X) = \Phi_g(f) \otimes \mathrm{Ad}_g X$ for $f \in \cin(\cZ)$ and $X \in \fg$.

\end{construction}

\begin{lemma}
    \label{lem:anchor_equivariant}
    The anchor map $\rho = A \to T_\cZ$ of the dg-Lie algebroid $A$ of $\cZ/G$ is $G$-equivariant with respect to the $G$-actions defined in Construction~\ref{ctr:gaction}.
\end{lemma}
\begin{proof}
    Recall that the anchor map $\rho$ is defined by $\rho(\alpha)(f) = \alpha(t^*f)$ for $\alpha \in A$ and $f \in \cin(\cZ)$. Therefore using the $G$-equivariance of the target map $t$ we obtain: \[\rho(g\cdot \alpha) = (g \cdot \alpha)\circ t^* = \Phi_g \circ \alpha \circ \Psi_{g^{-1}} \circ t^* = \Phi_g \circ \alpha \circ t^* \circ \Phi_{g^{-1}} = g \cdot \rho(\alpha).\] 
\end{proof}

\section{Bott--Shulman complex and the reduced form} \label{sec:bottshulman}

In the smooth case, to any simplicial manifold $M_\bullet$ one can associate a double complex $\Omega^{\bullet}(M_\bullet)$, called the Bott--Shulman complex, whose total cohomology computes the de Rham cohomology of the geometric realization of $M_\bullet$. We will now define a dg-version of the Bott--Shulman complex and introduce the reduced form $\omega_{\mathrm{red}}$ as an element of this complex. In its essence, the construction of the Bott--Shulman complex is a particular case of Dold--Kan correspondence applied to the de Rham complex.

\begin{construction}[Unnormalized Dold--Kan functor for cosimplicial objects] \label{ctr:doldkan}
    Let $\mathcal{A}$ be an abelian category. A \emph{cosimplicial object} in $\mathcal{A}$ is a functor $A^\bullet \colon \Delta \to \mathcal{A}$ from the simplex category $\Delta$ to $\mathcal{A}$. There is an unnormalized Dold--Kan functor \[ C \colon \catname{CoSimp}(\mathcal{A}) \to \catname{CoCh}_{\geq 0}(\mathcal{A}) \] from the category of cosimplicial objects of the category $\mathcal{A}$ to the category of non-negatively graded cochain complexes in $\mathcal{A}$. On objects, the functor $C$ is given by $C(A^\bullet)^n = A^n$ and the differential is given by the alternating sum of coface maps (see \cite[§8.4]{WEIBEL}).
\end{construction}

\begin{construction} \label{ctr:bottshulman}
In Appendix \ref{sec:derham} we showed that the de Rham complex of a dg-manifold gives a functor \[\Omega^{\bullet}_{\bullet} \colon \catname{dgMan}^{\mathrm{op}} \to \catname{DoubleComp}(\catname{Vec})\] to the category of double complexes in vector spaces. Suppose we are given a simplicial dg-manifold $X_\bullet \colon \Delta^{\mathrm{op}} \to \catname{dgMan}$. Then, composing the functor $X_\bullet$ with the de Rham functor, we obtain a cosimplicial object in the category of double complexes \[\Omega^{\bullet}_{\bullet} \circ X_\bullet^{\mathrm{op}} \colon \Delta \to \catname{DoubleComp}(\catname{Vec}).\] Applying the unnormalized Dold--Kan functor (see Construction~\ref{ctr:doldkan}) we obtain a cochain complex in double complexes, i.e. a triple complex in vector spaces. Explicitly, we get a triple complex $\Omega^{\bullet}_{\cin(X_\bullet), \bullet}$, with differentials $d$, $\delta$ and $\partial^*$ coming from the de Rham differential, the internal differential of the dg-manifolds and the alternating sum of pullbacks along face maps respectively.
\end{construction}

\begin{definition} \label{def:bottshulman}
Given a dg-groupoid $\mathcal{G}$, we can apply Construction~\ref{ctr:bottshulman} to the nerve $N_\bullet \mathcal{G}$ and obtain a triple \emph{Bott--Shulman complex} $(\Omega^{\bullet}_{\cin(N_\bullet \mathcal{G}), \bullet}, d, \delta, \partial^*)$. We define the corresponding \emph{de Rham complex} as the total complex \begin{equation} \label{eq:derham} \mathrm{DR}^m(\mathcal{G}) = \prod_{i+j+k = m} \Omega^i_{C^{\infty}(\mathcal{G}^{(j)}), k}.\end{equation}
\end{definition}

\begin{remark}[Regular case]
    \label{rmk:regular_reduction}
    When $0$ is a regular value of the moment map $\mu\colon M \to \fg^*$, the preimage $Z = \mu^{-1}(0)$ is a smooth submanifold of the manifold $M$ and the action Lie groupoid $G \times Z \rightrightarrows Z$ models the quotient $Z/G$. The restricted form $\iota^*_Z \omega\in \Omega^2(Z)$ satisfies $(s^* - t^*)\iota^*_Z \omega = 0$ with respect to the source and target maps. This equation exactly implies that the restricted form $\iota^*_Z \omega$ descends to the reduced space $Z/G$.

    In the derived setting, the analogous equation fails on the nose: for the action groupoid $G \times \cZ \rightrightarrows \cZ$, the restricted form $\iota^*_\cZ \omega$ no longer satisfies $(s^* - t^*)\iota^*_\cZ \omega = 0$. However, it satisfies this equation \emph{up to homotopy} in the Bott--Shulman complex: there exists a form $h \in \Omega^2_{\cin(G \times \cZ), -1}$ of simplicial degree $1$ and derived degree $-1$ with
    \[
        \iota_\mu(h) \;=\; (s^* - t^*)\iota^*_\cZ \omega.
    \]
    In the next construction we identify this homotopy explicitly as $h = d\theta$, where the form $\theta$ is the Maurer--Cartan form on $G$.
\end{remark}

\begin{construction}[The reduced form]
    \label{ctr:reducedform}
    We now define the 2-form on the groupoid $\cZ/G$ (see Construction~\ref{ctr:actiongroupoid}) that will play the role of the symplectic form on the quotient. The Bott--Shulman complex of the action dg-groupoid $\cZ/G$ in form degree $p$, simplicial degree $r$ and derived degree $q$ has the form \[\Omega^p_{\cin(G^r\times \cZ), q} \cong \bigoplus_{k+l = p}\left(\wedge^{q - k} \fg \otimes_\R \left(  \Omega^l(G^r\times M) \otimes_\R S^k \fg\right)\right)\] with differentials $d$, $\delta = \iota_\mu$ and $\partial^*$. If we choose local coordinates $\{x_i\}$ on the manifold  $G^r\times M$ and a basis $\{E_i\}$ of the Lie algebra $\fg$, then an element of $\Omega^p_{\cin(G^r\times \cZ), q}$ can be written as a sum of terms of the form \[(fE_{m_1} \wedge \cdots \wedge E_{m_{q-k}}) \, dx_{i_1} \wedge \cdots \wedge dx_{i_l} dE_{j_1} \cdots dE_{j_k}\] where $f$ is a smooth function on $G^r\times M$ and elements $E_{m_1}, \ldots, E_{m_{q-k}}$ are of derived degree $-1$, while $dx_{i_1}, \ldots, dx_{i_l}$ are of derived degree $0$ and form degree $1$ and $dE_{j_1}, \ldots, dE_{j_k}$ are of derived degree $-1$ and form degree $1$. The internal differential $\delta = \iota_\mu$ acts on derived degree 0 forms by $\iota_\mu(dx_i) = 0$ and on derived degree -1 forms by $\iota_\mu(dE_j) = d \langle \mu, E_j \rangle$. Then it extends to the whole Bott--Shulman complex as a derivation of derived degree 1 and form degree 0.

    The simplicial differential $\partial^*$ is given by the alternating sum of pullbacks along face maps. In particular, in simplicial degree 0 we have $\partial^* = s^* - t^* \colon \Omega^p_{\cin(\cZ)} \to \Omega^p_{\cin(G \times \cZ)}$. In simplicial degree 1 we have $\partial^* = \partial_0^* - \partial_1^* + \partial_2^* \colon \Omega^p_{\cin(G \times \cZ)} \to \Omega^p_{\cin(G\times G\times \cZ)}$, where $\partial_0, \partial_1, \partial_2$ are the face maps of the nerve of the groupoid $\cZ/G$.

    The reduced form $\omega_{\mathrm{red}} \in \textrm{DR}^2(\cZ/G)$ consists of two components. For the first component we take the original symplectic form $\omega \in \Omega^2(M)$ and pull it back to $\iota^*_\cZ \omega \in\Omega^2_{\cin(\cZ), 0} \subseteq \textrm{DR}^2(\cZ/G)$ using the inclusion $\iota_\cZ \colon \cZ \to M$. For the second component we define the one-form of derived degree -1 $\theta \in \Omega^1(G\times M) \otimes \fg \subseteq\Omega^1_{\cin(G \times \cZ), -1} \subseteq \textrm{DR}^1(\cZ/G)$ as the standard (left) Maurer--Cartan form on $G$ pulled back to $G \times M$. We set \[\omega_{\mathrm{red}} = \iota^*_\cZ \omega + d\theta.\] Notice that the exterior differential of the Maurer--Cartan form $\theta$ is taken in the derived sense, so if we write $\theta = \sum \theta_i \otimes E_i$ in some basis $\{E_i\}$ of the Lie algebra $\fg$, then $d\theta = \sum d \theta_i \otimes E_i + \theta_i \otimes dE_i \in \Omega^2_{C^{\infty}(G\times \cZ),-1}\subseteq \mathrm{DR}^2(\cZ/G)$.
\end{construction}

In order to show that the reduced form $\omega_{\mathrm{red}}$ is a closed form in the Bott--Shulman complex, we will need the following lemma.
\begin{lemma}
    \label{lem:theta_mult}
    The two-form $d\theta$ is multiplicative, i.e. \[\partial^* d\theta = 0.\]
\end{lemma}
\begin{proof}
    Let $m\colon G\times G \to G$ be the multiplication map, and $\textrm{pr}_1,\textrm{pr}_2 \colon G\times G \to G$ be the natural projections. Then $\theta$ satisfies \begin{equation} \label{multiplicative} (m^*\theta)_{(g_1,g_2)} = \textrm{Ad}_{g_2^{-1}}\textrm{pr}_1^* \theta +  \textrm{pr}_2^* \theta. \end{equation} To see this we take $\overline{E}_1 , \overline{E}_2$ to be left invariant vector fields on the Lie group $G$ and denote $E_1 = \overline{E}_1(e), E_2 = \overline{E}_2(e)$ their values at the identity. Then we compute:
    \begin{multline*}
        (m^*\theta)_{(g_1,g_2)}(\overline{E}_1, \overline{E}_2) = (\theta)_{g_1g_2}(R_{g_2}\overline{E}_1(g_1) + L_{g_1}\overline{E}_2(g_2)) \\  = E_2 + (\theta)_{g_1g_2}(L_{g_1g_2}\textrm{Ad}_{g_2^{-1}} E_1) = E_2 + \textrm{Ad}_{g_2^{-1}} E_1
    \end{multline*}

    Now consider the form $\theta$ as an element of $\Omega^1(G\times M) \otimes \fg \subseteq \Omega^1_{\cin(G \times M), -1}$. Recall that we have simplicial differentials $\partial_0, \partial_1, \partial_2 \colon G\times G \times \cZ \to G \times \cZ$, where $\partial_0$ is the projection onto the second factor, $\partial_1$ is the multiplication map, and $\partial_2$ is the action of the second factor on the dg-manifold $\cZ$. Equation~\eqref{multiplicative} implies that one-form $\theta$ is multiplicative. Indeed, if we write $\theta = \sum \theta_i \otimes E_i$ for some basis $E_i$ of $\fg$, then \begin{align*}
        \partial_0^* \theta &= \sum_i \textrm{pr}_2^*\theta_i \otimes E_i \\
        \partial_1^* \theta &= \sum_i m^*\theta_i \otimes E_i\\
        \partial_2^* \theta_{(g_1,g_2,m)} &= \sum_i \textrm{pr}_1^* \theta_i \otimes \textrm{Ad}_{g_2^{-1}} E_i.
    \end{align*}
    Therefore Equation~\eqref{multiplicative} is equivalent to $\partial_1^* \theta = \partial_0^* \theta + \partial_2^* \theta$, which is the multiplicativity condition for $\theta$. Since the exterior differential $d$ commutes with simplicial differential $\partial^*$, we get that the two-form $d\theta$ is multiplicative as well.
\end{proof}

\begin{proposition} \label{prop:closed}
    The form $\omega_{\mathrm{red}}$ is closed with respect to total differential $D = d + (-1)^p \iota_\mu +(-1)^{p+q} \partial^*$ (integers $p$ and $q$ denote the form degree and the derived degree, respectively) in the de Rham complex $\mathrm{DR}^2(\cZ/G)$.
\end{proposition}
\begin{proof}
    It is clear that the equation $d(\iota^*_\cZ \omega + d\theta) = 0$ is satisfied. From Lemma \ref{lem:theta_mult} we know that $\partial^*d\theta = 0$. Hence we only need to show that \[(s^* - t^*) \iota^*_\cZ \omega = \iota_{\mu} d\theta\]

    Note that the pullback form $\iota^*_\cZ \omega$ is just the ordinary symplectic form $\omega$ on the manifold $M$ sitting in $\Omega^2(M) = \Omega^2_{\cin(\cZ), 0}$ and pullbacks along the source and target maps $s^*$ and $t^*$ act on derived degree 0 forms just like smooth pullbacks. The computation then proceeds as follows. Let $\overline{E}_i$ be a left invariant frame of the Lie group $G$ and $E_i = \overline{E}_i(e)$ be the basis of the Lie algebra $\fg$. Then the differential of the action map $t \colon G\times M \to M$  is given by \[dt_{(g,m)}(\overline{E}_i, v) = (R_{g^{-1}} \overline{E}_i(g))^{\sharp}(gm)+ dg_m(v) = (\textrm{Ad}_{g} E_i)^{\sharp}(gm) + dg_m(v) = dg_m(E_i^{\sharp} + v).\]

    We can now compute 
    \begin{align*}
        (s^* - t^*)\omega_{(g,m)}(\overline{E}_i,\overline{E}_j,v_1,v_2) &=\omega_m(v_1,v_2) - \omega_{gm}(dg_m(E_i^{\sharp}+v_1),dg_m(E_j^{\sharp}+v_2)) \\ &= \omega_m(v_1, v_2) - \omega_m(E_i^{\sharp}+v_1, E_j^{\sharp}+v_2) \\ &= -(\omega_{m}(E_i^{\sharp},v_2)+\omega_{m}(E_i^{\sharp},E_j^{\sharp})+\omega_{m}(v_1,E_j^{\sharp})) \\ &= d_m\langle \mu,  E_j\rangle (v_1) - d_m\langle \mu,  E_i\rangle (v_2)  - \langle \mu, [E_i, E_j]\rangle(m) \\ &= d\langle \mu, \theta \rangle_{(g,m)}(\overline{E}_i, \overline{E}_j, v_1, v_2).
    \end{align*}
    Here in the second equality we used that the symplectic form $\omega$ is $G$-invariant. We get that $(s^* - t^*)\omega = d \langle \mu, \theta \rangle$. But since the inner differential $\iota_{\mu}$ commutes with the exterior differential $d$, we get that\[(s^* - t^*)\omega = d \langle \mu, \theta \rangle = d \iota_\mu(\theta) = \iota_\mu d\theta.\]

\end{proof}

The form $\omega_{\mathrm{red}}$ is reduced in the following sense. Let $\overline{\cZ}= \{ \cZ \rightrightarrows \cZ\}$ be the unitgroupoid. Then there is a morphism of dg-groupoids, the canonical atlas map $\pi \colon \overline{\cZ} \to \cZ/G $ given by $\pi = (u, \id)$, where $u \colon \cZ \to G\times \cZ$ is the unit map. Then one has the following result:

\begin{proposition} \label{prop:red}
    Let $\omega\vert_{\overline{\cZ}}$ denote the form $\iota^*_\cZ \omega$ viewed as an element of $\mathrm{DR}^2(\overline{\cZ})$. Then we have $\pi^*\omega_{\mathrm{red}} = \omega\vert_{\overline{\cZ}}$.
\end{proposition}

\begin{proof}
  It is enough to show that $u^* d\theta = 0$. Write $\theta = \sum_i \theta_i \otimes E_i$, with one-forms $\theta_i \in \Omega^1(G)$. Then we have $u^*\theta_i = 0$ for all indices $i$, and so $u^* d\theta = du^* \theta = 0.$
\end{proof}

\section{The reduced form is symplectic} \label{sec:symplectic}

In this section we show that the reduced form $\omega_{\mathrm{red}}$ on the dg-groupoid $\cZ/G$ is non-degenerate in a suitable sense. To make this precise we need to introduce the notion of tangent and cotangent complexes of the dg-groupoid $\cZ/G$, which are the derived analogues of the tangent and cotangent bundles of a smooth manifold. 

\begin{notation}[Recap]
    Let us stop for a moment and recall what we had constructed so far, and the notation that we used. We fix a symplectic manifold $(M, \omega)$ with a Hamiltonian $G$-action and equivariant moment map $\mu \colon M \to \fg^*$. 

    We constructed derived zero-level set $\cZ$ with global dg-algebra of functions $\cin(\cZ) =\cin(M) \otimes \wedge^{-\bullet}\fg$ and differential $\iota_\mu$ (\eqref{eq:zerolevelset}). The dg-algebra $\cin(\cZ)$ is generated over $\cin(M)$ by elements $X \in \fg$ of derived degree $-1$. There is a natural embedding $\iota_\cZ \colon \cZ \to M$, which induces the inclusion $\iota^*_\cZ \colon \cin(M) \hookrightarrow \cin(\cZ)$ in degree 0 (see Notation~\ref{not:pullback}). We then constructed an action groupoid modelling the derived symplectic quotient $\cZ/G = \{G \times \cZ \rightrightarrows \cZ\}$ (see Construction~\ref{ctr:actiongroupoid}).
   
    We also introduced some dg-modules that we are going to use. Dg-modules \[\fg_\cZ = \cin(\cZ) \otimes \fg, \qquad \fg^*_\cZ = \cin(\cZ) \otimes \fg^*\] are trivial dg-modules over the dg-algebra $\cin(\cZ)$. They inherit an inner differential $\iota_\mu$ from $\cin(\cZ)$.
    Modules $\iota_\cZ^* T_M$ and $\iota_\cZ^* T^*_M$ are the pullbacks of the module of vector fields and module of one-forms of $M$ respectively to the derived zero set $\cZ$ (see Notation~\ref{not:pullback}). The pullback modules get their inner differential $\iota_\mu$ from the dg-algebra $\cin(\cZ)$ as well.

In Proposition~\ref{prop:kahler} we computed tangent and cotangent modules $T_\cZ$ and $T^*_\cZ$ of the dg-manifold $\cZ$ using the explicit cocone/cone descriptions
          \[
              T_\cZ \cong \cocone\bigl(\iota^*_\cZ T_M \xrightarrow{-d\mu} \fg^*_\cZ\bigr),
              \qquad
              T^*_\cZ \cong \cone\bigl(\fg_\cZ \xrightarrow{(d\mu)^*} \iota^*_\cZ T^*_M\bigr).
          \]
 Here the map $(d\mu)^*$ is defined by $(d\mu)^*(X) = d\langle \mu, X\rangle$ for $X \in \fg$ and extends by $\cin(\cZ)$-linearity, and the map $d\mu$ is defined dually by $\langle d\mu(V), X \rangle = \iota_V d\langle \mu, X\rangle$ for $V \in T_M$ and $X \in \fg$.

 Given $X \in \fg$ we write $dX \in T^*_\cZ$ for the derived one form degree $-1$, which is the image of $X$ under the de Rham differential $d$. One forms $dX$ in derived degree $-1$, together with smooth forms $\beta \in T^*_M$ generate $T^*_\cZ$ as a graded module over $\cin(\cZ)$.
 Similarly, given $\sigma \in \fg^*$ we write $\iota_\sigma \in T_\cZ$ for the corresponding contraction $\iota_\sigma \colon \cin(M) \otimes \wedge^\bullet \fg \to \cin(M) \otimes \wedge^{\bullet -1} \fg$ viewed as a degree 1 derivation on $\cin(\cZ)$ (see Notation~\ref{not:dg_alg_not}). Derivations $\iota_\sigma$ in derived degree $1$ and smooth vector fields $W \in T_M$ generate $T_\cZ$ as a graded module over $\cin(\cZ)$.

In Proposition~\ref{prop:algebroid} we computed the dg-Lie algebroid of the action dg-groupoid $\cZ/G$ to be $A \cong \fg_\cZ$. We also computed the anchor map to be $\rho = \rho_0 + \eta \colon \fg_\cZ \to T_\cZ$, where $\rho_0$ is the infinitesimal action $\fg \to T_M$ extended over $\cin(\cZ)$ to $\rho_0 \colon \fg_\cZ \to \iota^*_\cZ T_M$ and $\eta\colon \fg_\cZ \to \fg^*_\cZ[-1]$ is the homotopy given by the coadjoint action $\eta(X) = \mathrm{ad}^*_X$ for $X \in \fg\subseteq \fg_{\cZ, 0}$. Furthermore in Lemma~\ref{lem:anchor_equivariant} we showed that the anchor $\rho$ is a $G$-equivariant map of dg-modules.

Finally, in Section~\ref{sec:bottshulman} we constructed the reduced form $\omega_{\mathrm{red}} = \iota^*_\cZ \omega + d\theta$ of total degree 2 in the Bott--Shulman complex of the action groupoid $\cZ/G$, where $\theta$ is the Maurer--Cartan form on the Lie group $G$ pulled back to $G \times \cZ$ (Construction~\ref{ctr:reducedform}). We showed that it is closed in the Bott--Shulman complex (see Proposition~\ref{prop:closed}) and that it is reduced (see Proposition~\ref{prop:red}).
\end{notation}

\begin{construction}
\label{ctr:tangent_complex_action}
    The \emph{tangent complex} of the action groupoid $\cZ/G$ is the
    $G$-equivariant two-term complex of $\cin(\cZ)$-modules
    \[
        T_{\cZ/G} \;\colon\; \fg_\cZ \xrightarrow{\rho} T_\cZ,
    \]
    with the dg-$\cin(\cZ)$-module $\fg_\cZ$ in degree $-1$ and tangent module $T_\cZ$ in degree $0$, where $\rho = \rho_0 + \eta$ is the anchor map of the dg-Lie algebroid of $\cZ/G$ (Proposition~\ref{prop:algebroid}) and the action of $G$ on $T_{\cZ/G}$ is as in Lemma~\ref{ctr:gaction}.

    Dually, the \emph{cotangent complex} of $\cZ/G$  is the
    $G$-equivariant complex \[T^*_{\cZ/G} \;\colon\; T^*_\cZ \xrightarrow{\rho^*} \fg^*_\cZ,\]
    with the cotangent module $T^*_\cZ$ in degree 0 and $\fg^*_\cZ$ in degree 1. The cotangent complex $T^*_{\cZ/G}$ carries the dual $G$-action.
\end{construction}

\begin{remark}
    In the non-graded setting, there is a one-to-one correspondence between two-term complexes with representations up to homotopy and VB-groupoids (see \cite{VB}), which are vector bundles in the category of Lie groupoids. By design, VB-groupoids represent vector bundles over stacks. For example, the tangent Lie groupoid $TG \rightrightarrows TX$ (i.e., the tangent bundle of the stack $[X/G]$) corresponds to the adjoint representation up to homotopy on the tangent complex. Therefore, the representation up to homotopy is an essential structural component; without it, we cannot recover the corresponding vector bundles over the stack. 
    
    For $G$-action groupoids the representation up to homotopy structure on the complex is just an action of the group $G$. Therefore, since developing general theory of representations up to homotopy in the dg-setting is beyond the scope of this paper, we included the $G$-equivariance in the definition of tangent and cotangent complexes in Construction \ref{ctr:tangent_complex_action}.
\end{remark}

\begin{remark}[Regular case]
    Suppose, as in Remark~\ref{rmk:regular_reduction}, that $0$ is a regular value of the moment map $\mu$ and that the Lie group $G$ acts freely on the zero-level set $Z = \mu^{-1}(0)$, so that the reduced space $Z/G$ is a smooth manifold. The restricted form $\iota^*_Z\omega \in \Omega^2(Z)$ induces a $G$-equivariant map from the tangent complex of the action Lie groupoid $G \times Z \rightrightarrows Z$ to its cotangent complex:

    \begin{center}
\begin{tikzcd}
Z \times \mathfrak{g} \arrow[d, "0"] \arrow[r, "\rho"] & TZ \arrow[d, "(\iota^*_Z \omega)^{\flat}"] \arrow[r] & 0 \arrow[d, "0"]        \\
0 \arrow[r]                                            & T^*Z \arrow[r, "\rho^*"]                 & Z \times \mathfrak{g}^*
\end{tikzcd}
\end{center}
    where $\rho$ is the anchor map of the action Lie algebroid $Z \times \fg \to TZ$. In fact, this map of complexes is a quasi-isomorphism, i.e. the restricted form $\iota^*_Z \omega$ induces an isomorphism of vector bundles $\coker(\rho) \cong \ker(\rho^*)$ on $Z$, which is equivalent to the classical non-degeneracy of the reduced symplectic form on the manifold $Z/G$.

    In the derived setting, the analogous map between tangent and cotangent complexes $T_{\cZ/G}$ and $T^*_{\cZ/G}$ induced by the "restricted" form $\iota^*_\cZ \omega$ alone is no longer a chain map. In particular, for the anchor map $\rho \colon \fg_\cZ \to T_\cZ$, the composition $(\iota^*_\cZ \omega)^{\flat} \circ \rho$ is not zero, but homotopic to zero via the homotopy $\alpha$ induced by contraction with the two-form $d\theta$. There is also a dual homotopy $\alpha^*$ between $\rho^* \circ (\iota^*_\cZ \omega)^{\flat}$ and $0$. Since differential forms $\iota^*_\cZ \omega$ and $d\theta$ have the same total degree, together they induce a map between \emph{total} complexes $\Tot(T_{\cZ/G})$ and $\Tot(T^*_{\cZ/G})$. Surprisingly, this map turns out to be an isomorphism, rather than just a quasi-isomorphism, which is the content of Theorem \ref{thm:nondeg}.
\end{remark}

\begin{remark}
    In what follows we will need a notion of a contraction of a form with a vector field. In particular, given a dg-manifold $\cX$ and a derived two-form $\beta \in \Omega^2_{\cin(\cX), q}$ of derived degree $q$, there is a well-defined dg-$\cin(\cX)$-module map \[\beta^{\flat} \colon T_\cX \to T^*_\cX[q], \qquad \beta^{\flat}(V) = \iota_V \beta\] for $V \in T_\cX$. The contractions with vector fields for dg-manifolds are established in the Appendix in Proposition~\ref{prop:contractions}. 
\end{remark}

\begin{construction}
    \label{ctr:inducedmap}
We construct the map induced by the symplectic form $\omega_{\mathrm{red}}  = \iota^*_\cZ \omega + d\theta \in \mathrm{DR}^2(\cZ/G)$  of total degree 0 between tangent and cotangent complexes. The motivation for this construction comes from the correspondence between multiplicative forms and IM-forms in the non-graded setting (see \cite{XU} for quasi-symplectic groupoids and \cite{IMFORMS} for the general case).

First, contraction with the restricted symplectic form $\iota^*_\cZ\omega$ induces a map $(\iota^*_\cZ \omega)^{\flat} \colon T_\cZ \to T^*_\cZ$. Next, the contraction with the differential of the Maurer--Cartan form $d\theta$ induces a map $d\theta^{\flat} \colon T_{G\times \cZ} \to T^*_{G\times \cZ}[-1]$. Since the form $d\theta$ is multiplicative (see Lemma~\ref{lem:theta_mult}), it restricts to the map on units: $\alpha \colon T_\cZ \to \fg^*_\cZ[-1]$. Indeed, if we write $d\theta = \sum_i d\theta_i \otimes E_i + \theta_i \otimes dE_i$, then for vector fields $\mathrm{id} \otimes V \in T_{G\times M} \subseteq T_{G \times \cZ, 0}$ we have $d\theta^\flat(\mathrm{id} \otimes V) = 0$, and for vector fields $\iota_{\sigma_i}$ induced by contractions with a dual basis $\sigma_i \in \fg^*$ we have $ d\theta^\flat(\iota_{\sigma_i}) = \theta_i$. Hence, restricting to units, for a vector field $V \in T_M \subseteq T_{\cZ, 0}$, we have \[\alpha(V) = (d\theta)^\flat(\mathrm{id}\otimes V)\vert_\cZ = 0\] and for the vector field $\iota_{\sigma_i} \in T_{\cZ, 1}$ we have \[\alpha(\iota_{\sigma_i}) = (d\theta)^\flat (\iota_{\sigma_i})\vert_\cZ = \theta_i(e) = \sigma_i.\] These equations determine $\alpha$ uniquely by $\cin(\cZ)$-linearity. With respect to the decomposition $T_\cZ = \iota^*_\cZ T_M \oplus \fg^*_\cZ[-1]$, the map $\alpha$ is zero on the first component and a $\cin(\cZ)$-linear isomorphism $\alpha \colon \fg^*_\cZ[-1] \xrightarrow{\sim} \fg^*_\cZ[-1]$ on the second component.

The dual map $\alpha^* \colon \fg_\cZ[1] \to T^*_\cZ$ is given by \[
        \alpha^*(1\otimes E_i) = dE_i  \quad \text{for } 1\otimes E_i \in \fg_\cZ[1].
\] With respect to the decomposition of the cotangent module $T^*_\cZ =\fg_\cZ[1] \oplus \iota^*_\cZ T^*_M $, the map $\alpha^*$ is a $\cin(\cZ)$-linear isomorphism $\alpha^* \colon \fg_\cZ[1] \xrightarrow{\sim} \fg_\cZ[1]$ onto the first component.

We set \begin{equation} \label{eq:omegared} \omega_{\mathrm{red}}^{\flat} = \alpha + (\iota^*_\cZ \omega)^{\flat} + \alpha^*.\end{equation}
\end{construction}

\begin{theorem}[Derived symplectic reduction] \label{thm:nondeg}
    The induced map defined by \eqref{eq:omegared} \[\omega_{\mathrm{red}}^{\flat} \colon \Tot(T_{\cZ/G}) \to \Tot(T^*_{\cZ/G})\] is a $G$-equivariant isomorphism of dg-$\cin(\cZ)$-modules. 
\end{theorem}
\begin{proof} 

On the level of graded modules the total complexes decompose as $\Tot(T_{\cZ/G}) = T_{\cZ} \oplus \fg_\cZ[1] = \fg_\cZ[1] \oplus \iota^*_\cZ T_M \oplus \fg^*_\cZ[-1]$ and $\Tot(T^*_{\cZ/G}) = T^*_{\cZ} \oplus \fg^*_\cZ[-1] = \fg_\cZ[1] \oplus \iota^*_\cZ T^*_{M} \oplus \fg^*_\cZ[-1]$. Recall that dg-$\cin(\cZ)$-modules $\fg_\cZ$, $\iota^*_\cZ T_M$, $\iota^*_\cZ T^*_M$ and $\fg^*_\cZ$ inherit the inner differential $\iota_\mu$ from $\cin(\cZ)$. The total differentials $\delta_{\Tot(T_{\cZ/G})}$ and $\delta_{\Tot(T^*_{\cZ/G})}$ can then be written in the matrix form with respect to the decomposition above as \[\delta_{\Tot(T_{\cZ/G})} = \begin{bmatrix}
    -\iota_\mu & 0 & 0 \\
    \rho_0 & \iota_\mu & 0 \\
    \eta & -d\mu & -\iota_\mu
\end{bmatrix}, \quad \delta_{\Tot(T^*_{\cZ/G})} = \begin{bmatrix}
    -\iota_\mu & 0 & 0 \\
    d\mu^* & \iota_\mu & 0 \\
    \eta^* & \rho_0^* & -\iota_\mu
\end{bmatrix}.\]
The signs come from the cone/cocone conventions (see Definition~\ref{def:cone_cocone}) and from the total complex construction. 

The map $\omega^{\flat}_{red} = \alpha + (\iota^*_\cZ \omega)^{\flat} + \alpha^*$ induces a map between the components of the total complexes $\Tot(T_{\cZ/G})$ and $\Tot(T^*_{\cZ/G})$ as follows: 

\[
\begin{tikzcd}[row sep=huge, column sep=large,
               every label/.append style={font=\small}]
    \fg_\cZ[1]
        \arrow[d, "\alpha^*"']
        \arrow[r, "\rho_0"]
        \arrow[rr, bend left=25, "\eta"]
        & \iota^*_\cZ T_M
            \arrow[d, "(\iota^*_\cZ \omega)^\flat"']
            \arrow[r, "-d\mu"]
        & \fg^*_\cZ[-1]
            \arrow[d, "\alpha"] \\
    \fg_\cZ[1]
        \arrow[r, "d\mu^*"']
        \arrow[rr, bend right=25, "\eta^*"']
        & \iota^*_\cZ T^*_M
            \arrow[r, "\rho_0^*"']
        & \fg^*_\cZ[-1]
\end{tikzcd}
\]
First we need to check that $\omega^{\flat}_{red}$ is a chain map, i.e. that \[\omega^{\flat}_{red} \delta_{\Tot(T_{\cZ/G})} = \delta_{\Tot(T^*_{\cZ/G})} \omega^{\flat}_{red}.\] Using the matrix form we compute: \[\omega^{\flat}_{red} \delta_{\Tot(T_{\cZ/G})} = \begin{bmatrix}
    -\alpha^* \iota_\mu & 0 & 0 \\
    (\iota^*_\cZ \omega)^\flat\rho_0 & (\iota^*_\cZ \omega)^\flat\iota_\mu & 0 \\
    \alpha\eta & -\alpha d\mu & -\alpha\iota_\mu
\end{bmatrix}, \quad \delta_{\Tot(T^*_{\cZ/G})} \omega^{\flat}_{red} = \begin{bmatrix}
    -\iota_\mu\alpha^* & 0 & 0 \\
    d\mu^*\alpha^* & \iota_\mu (\iota^*_\cZ \omega)^\flat & 0 \\
    \eta^* \alpha^* & \rho_0^* (\iota^*_\cZ \omega)^\flat & -\iota_\mu \alpha
\end{bmatrix}\]

Clearly, the diagonal elements are equal, since maps $\alpha, (\iota^*_\cZ \omega)^{\flat}$ and $\alpha^*$ are maps of dg-$\cin(\cZ)$-modules. We are left with checking three identities:

\begin{itemize}
\item $(\iota^*_\cZ \omega)^\flat \rho_0 = d\mu^* \alpha^*$ as maps between the trivial dg-module $\fg_\cZ[1]$ and the pullback dg-module $\iota^*_\cZ T^*_M$.

 By $\cin(\cZ)$-linearity, it is enough to prove that on elements of the form $1 \otimes X \in \fg_{\cZ, 0}$, where $X \in \fg$ with the corresponding infinitesimal action vector field $X^\sharp$ on the manifold $M$. Then \[(\iota^*_\cZ \omega)^\flat\rho_0(1\otimes X) = \iota_{X^{\sharp}} \omega = d\langle \mu , X \rangle = d\mu^*(dX) = d\mu^*\alpha^*(1 \otimes X)\]

 \item $-\alpha d\mu= \rho_0^* (\iota^*_\cZ \omega)^\flat$ as maps between $\iota^*_\cZ T_M$ and $\fg^*_\cZ[-1]$.
 
 This statement is dual to the previous one. Again, by $\cin(\cZ)$-linearity it is enough to show that on elements $1\otimes V \in \iota_\cZ^* T_M$, where $V$ is a vector field on the manifold $M$. Then \[\rho_0^* (\iota^*_\cZ \omega)^\flat(1\otimes V) = \rho_0^* (\iota_V \omega) = \{ X \mapsto \iota_{X^{\sharp}}\iota_V \omega \} =- V\langle \mu , \cdot \rangle = -\alpha d\mu(1\otimes V).\] 

 \item $\alpha \eta = \eta^* \alpha^*$ as maps between trivial dg-modules $\fg_\cZ[1]$ and $\fg^*_\cZ[-1]$.
 
 Once again, we will check this identity on elements of the form $1 \otimes E_i \in \fg_{\cZ, 0}$, where $E_i$ is a basis of the Lie algebra $\fg$.
 
 First we want to find a formula for $\eta^*$ - a dual of $\eta$. For this pick a basis $\{E_i\}$ of the Lie algebra $\fg$ and dual basis $\{\sigma_i\}$ of $\fg^*$. Recall from the proof of Proposition~\ref{prop:algebroid} that the homotopy $\eta$ can be written as \[\eta = -\sum_{k,j} c^k_{ij} E_k \otimes \iota_{\sigma_j}\] where $c^k_{ij}$ are structure functions of the Lie algebra $\fg$. Also recall that the anchor map $\rho \colon \fg_\cZ \to T_\cZ$ is given by $\rho = \rho_0 + \eta$.  Then $\rho^* = \rho_0^* + \eta^*$ and \[\langle\rho^*(dE_j), E_i \rangle = \langle dE_j, E_i^{\sharp} + \eta(E_i) \rangle  = -\sum_k c^k_{ij}E_k\] Therefore, \begin{align*} \eta^* \alpha^*(1\otimes E_i) = \eta^*(dE_i) = -\sum_{j,k} c^k_{ij} E_k \otimes \sigma_j =  -\sum_{j,k} c^k_{ij} E_k \otimes \alpha(\iota_{\sigma_j}) = \alpha \eta(E_i), \end{align*} where the last equality follows from the fact that $\alpha$ is a $\cin(\cZ)$-module map.
\end{itemize}

Finally we establish that the induced map $\omega^{\flat}_{red} \colon \Tot(T_{\cZ/G}) \to \Tot(T^*_{\cZ/G})$ is an isomorphism of $\cin(\cZ)$-modules. Indeed, since $\omega^{\flat} \colon T_M \to T^*_M$ is an isomorphism of $\cin(M)$-modules, the map  $(\iota^*_\cZ \omega)^\flat$ induces an isomorphism from $\iota^*_\cZ T_M$ to $\iota^*_\cZ T^*_M$. Furthermore, $\alpha$ induces an isomorphism between modules $\fg^*_\cZ[-1] \subseteq \Tot(T_{\cZ/G})$ and $\fg^*_\cZ[-1] \subseteq \Tot(T^*_{\cZ/G})$, and $\alpha^*$ induces isomorphisms between modules $\fg_\cZ[1] \subseteq \Tot(T_{\cZ/G})$ and $\fg_\cZ[1] \subseteq \Tot(T^*_{\cZ/G})$. Therefore the induced map $\omega_{\mathrm{red}}^{\flat}$ is an isomorphism of $C^{\infty}(\cZ)$-modules of degree 0. This isomorphism is $G$-equivariant, since the maps $\alpha, (\iota^*_\cZ \omega)^{\flat}$ and $\alpha^*$ are $G$-equivariant.

\end{proof}

\begin{remark}
    \label{rmk:BRST_form}

    The map $\alpha \colon T_\cZ \to \fg^*_\cZ[-1]$, which is the infinitesimal version of the multiplicative form $d\theta$, induces a 2-form on the infinitesimal quotient $\cZ/\fg$ (see Remark \ref{rmk:BRST}) as follows. Recall that the bigraded complex $\cin(\cZ/\fg)$ is given by $\cin(M) \otimes \wedge^{-\bullet} \fg \otimes \wedge^\bullet \fg^*$. Pick a basis $ \{E_i\}$ of $\fg$ and dual basis $\{\sigma_i\}$ of $\fg^*$. The map $\alpha$ corresponds to a form $\hat{\alpha}$ of $\Omega^1_{\cin(\cZ)} \otimes S^1(\fg^*_\cZ) \subseteq \Omega^2_{\cin(\cZ/\fg)}$. Under this correspondence, the form $\hat{\alpha}$ is given by $\hat{\alpha} = \sum_i dE_i \otimes d\sigma_i$, where $dE_i \in \Omega^1_{\cin(\cZ)}$ is a 1-form of derived bidegree $(-1,0)$ and $d\sigma_i \in S^1(\fg^*_\cZ)$ is a 1-form of derived bidegree $(0,1)$ in $\Omega^{\bullet}_{\cin(\cZ/\fg)}$. The form $\omega_0 = \omega + \hat{\alpha}$ on $\cZ/\fg$ of form degree 2 and total derived degree 0 gives $\cZ/\fg$ the structure of a 0-shifted symplectic dg NQ-manifold (see \cite[Section 1.3]{PRIDHAM}). The Poisson bracket induced by $\omega_0$ on the bigraded algebra $\cin(\cZ/\fg)$ is the classical BRST Poisson bracket, which is used in the physics literature to describe the reduced Poisson algebra of functions on the symplectic quotient $\mu^{-1}(0)/G$ in the regular case (see \cite{KOSTANTBRST}).

    The correspondance between the form $\omega_{\mathrm{red}} = \iota^*_\cZ \omega + d\theta$ on the action groupoid $\cZ/G$ and the form $\omega_0 = \omega + \hat{\alpha}$ on the infinitesimal quotient $\cZ/\fg$ can be thought of as a particular instance of the derived analogue of the Van Est map. In the non-derived setting Van Est map is a map from the Bott--Shulman complex of a Lie groupoid $G\rightrightarrows M$ to the Weil algebra of its Lie algebroid $A$ (see \cite{VANEST}): \[\mathrm{VE} \colon \Omega^{\bullet}(G_{\bullet}) \to W^{\bullet, \bullet}(A).\] In fact, as noted in \cite{MEHTATHESIS} and \cite{VANEST}, the Weil algebra of a Lie algebroid $A$ can be identified with the de Rham complex of the dg-manifold $(M, \Gamma(\wedge^\bullet A^*), d_A)$. Therefore, in the derived setting we can expect a map from the Bott--Shulman complex of the action dg-groupoid $\cZ/G$ to the de Rham complex of the infinitesimal quotient $\cZ/\fg$:
    \[\mathrm{VE} \colon \Omega^{\bullet}_{\cin(G^{\bullet}\times \cZ)} \to \Omega^{\bullet}_{\cin(\cZ/\fg)}.\]

    For dg-groupoids such a map was built on the level of groupoid cochains by Mehta in \cite[Chapter 6]{MEHTATHESIS}. It would be interesting to extend this map to the level of differential forms and explore under which conditions it induces an isomorphism on cohomology, but we will leave this question for future work.
\end{remark}

\printbibliography

\appendix 
\section{De Rham complex of a dg-manifold} \label{sec:derham}

In this section we will prove technical results regarding the existence of K\"ahler differentials and construction of the de Rham complex for dg-manifolds. Our approach is similar to the one in \cite{LERMAN} and can be easily adapted to dg $\cin$-schemes. See section~\ref{sec:kahler} for relevant definitions.

\begin{proposition}
\label{prop:kahler_ex}
    For any dg-manifold $\cX$, there exists the module of K\"ahler differentials \[d \colon C^{\infty}(\cX) \to \Omega^1_{C^{\infty}(\cX)}\]
\end{proposition}
\begin{proof}

Take a graded set $S = \{S_i\}_{i\in \mathbb{Z}}$, with components $S_i = \{\hat{d}a, \, a\in C^{\infty}(\cX)_i\}$. Then take a free $C^{\infty}(\cX)$-module on $S$ which we will denote by $\textrm{Free}(S)$. We can endow $\textrm{Free}(S)$ with $C^{\infty}(\cX)$-dg-module structure by setting \[\delta_S \left(\sum f_i \hat{d}a_i\right) = \sum (\delta_{\cX}(f_i)\hat{d}a_i + (-1)^{|f_i|}f_i \hat{d}(\delta_{\cX}a_i))\]  We define a graded $\cin(\cX)$-module $\Omega^1_{C^{\infty}(\cX)}$ as a quotient $\textrm{Free}(S)/\sim$ by the following relations:

    \begin{enumerate}
        \item $\hat{d}(ab) = (\hat{d}a)b + a(\hat{d}b)$
        \item For $a_1, \ldots, a_n \in C^{\infty}(\cX)_0$ and $f \in C^{\infty}(\R^n)$ we have \[\hat{d}(f(a_1, \ldots, a_n)) = \sum_{i=1}^n \frac{\partial f}{\partial x_i} (a_1, \ldots a_n) \hat{d}a_i\]
    \end{enumerate}

The differential $\delta_S$ on $\textrm{Free}(S)$ descends to a differential on the graded module $\Omega^1_{\cin(\cX)}$, making it a dg-$\cin(\cX)$-module. The map $d \colon C^{\infty}(\cX) \to \Omega^1_{\cin(\cX)}$ is defined by sending $a$ to the equivalence class of $\hat{d}a$, denoted $da$. This is a dg-morphism by definition, and is a degree 0 derivation, due to the above relations.

We now check the universal property. Suppose that $d' \colon C^{\infty}(\cX) \to (\mathcal{M},\delta_{\mathcal{M}})$ is a degree $n$ derivation. Then we define a map $\phi \colon \textrm{Free}(S) \to \mathcal{M}$ of integer degree $n$, by $\phi(\sum f_i \hat{d}a_i) = \sum f_id'a_i$. It is easy to see that it intertwines the differentials $\delta_S$ and $\delta_{\mathcal{M}}$, since $\delta_{\mathcal{M}}(d'a) = d'\delta_{\cX}a$, and $\delta_S(\hat{d}a)=\hat{d}\delta_\cX a$.

Since $d'$ is a derivation, the map $\phi$ descends to a map $\phi \colon \Omega^1_{\cin(\cX)} \to \mathcal{M}$. Furthermore it's easy to see that the map $\phi$ is unique, because it has to satisfy $\phi(da) = d'a$.

\end{proof}

Recall, that in Definition~\ref{def:dgformsmodule} we defined a dg-$\cin(\cX)$ module $\Omega^k_{\cin(\cX)} = \bigwedge^k_{\cin(\cX)}\Omega^1_{\cin(\cX)}$. We want to extend the differential $d \colon C^{\infty}(\cX) \to \Omega^1_{\cin(\cX)}$ to a degree 1 (with respect to form degree) derivation $d \colon \Omega^{\bullet}_{\cin(\cX)} \to \Omega^{\bullet+1}_{\cin(\cX)}$, turning the complex $(\Omega^{\bullet}_{\cin(\cX)},d, \delta)$ into commutative graded differential dg-algebra. First, we prove the following lemma:

\begin{lemma}
    The universal differential $d \colon C^{\infty}(\cX) \to \Omega^1_{\cin(\cX)}$ gives rise to a unique map of dg-vector spaces $d\colon \Omega^1_{\cin(\cX)} \to \Omega^2_{\cin(\cX)}$ with
    \begin{itemize}
        \item $d(\sum f_i dg_i) = \sum df_i\wedge dg_i$, for $f_i,g_i \in C^{\infty}(\cX)$.
        \item $d(df) = 0$ for all $f \in C^{\infty}(\cX)$
    \end{itemize}
\end{lemma}
\begin{proof} Consider the $\mathbb{R}$-linear map \[\phi \colon \textrm{Free}(S) \to \Omega^2_{\cin(\cX)}, \quad \phi\left(\sum_i f_i \hat{d}g_i\right) = \sum_i df_i \wedge dg_i\] where as in the proof of Proposition \ref{prop:kahler_ex}, $S$ is a graded set $\{\hat{d}a\}_{a \in C^{\infty}(\cX)}$. To see that the map $\phi$ is a dg-morphism, we check: \[\phi \delta(f\hat{d}g) = \phi(\delta(f)\hat{d}g+(-1)^{|f|}f \hat{d}(\delta g)) = d(\delta f)\wedge dg + (-1)^{|f|}df\wedge d(\delta g) = \delta(df\wedge dg) = \delta \phi(f\hat{d}g)\]
The map $\phi$ vanishes on the relation submodule, and therefore it descends to a map $d \colon \Omega^1_{\cin(\cX)} \to \Omega^2_{\cin(\cX)}$.
\end{proof}
\begin{theorem} \label{thm:derham_diff}
    The universal differential extends to a unique degree 1 map of dg-vector spaces \[d \colon \Omega^{\bullet}_{\cin(\cX)} \to \Omega^{\bullet + 1}_{\cin(\cX)}\] so that for all $k>0$ and all $f, g_1, \ldots , g_k \in C^{\infty}(\cX)$ \[d(f dg_1\wedge\ldots\wedge dg_k) = df\wedge dg_1\wedge \ldots \wedge dg_k.\] Consequently $d\circ d =0$.
\end{theorem}
\begin{proof} Consider the map \[\beta \colon \overbrace{\Omega^1_{\cin(\cX)}\times \ldots \times \Omega^1_{\cin(\cX)}}^k \to \Omega^{k+1}_{\cin(\cX)}, \quad \beta(\theta_1, \ldots, \theta_k) = \sum_i (-1)^{i+1}\theta_1\wedge \ldots \wedge d\theta_i \wedge \ldots \wedge \theta_k. \] This is a dg-morphism of graded vector spaces, hence it descends to a map \[\gamma \colon \Omega^1_{\cin(\cX)} \otimes_{\mathbb{R}} \ldots \otimes_{\mathbb{R}} \Omega^1_{\cin(\cX)} \to \Omega^{k+1}_{\cin(\cX)}.\] It is easy to see that the map $\gamma$ is alternating. We check that for any integer $i$ and function $f\in C^{\infty}(\cX)$ and any one-forms $\theta_1, \ldots, \theta_k \in \Omega^1_{\cin(\cX)}$ \[\gamma(\theta_1\otimes \cdots \otimes f\theta_i \otimes \cdots \otimes \theta_k - (-1)^{|\theta_i||f|}\theta_1 \otimes \cdots \otimes f\theta_{i+1} \otimes \cdots \otimes \theta_k) = 0.\] Indeed, \begin{align*}
&\gamma(\theta_1 \otimes \cdots \otimes f\theta_i \otimes \cdots \otimes \theta_k) = \\ &= \sum_{l\neq i} (-1)^{l+1}\theta_1\wedge \ldots \wedge d\theta_l \wedge \ldots \wedge f\theta_i \wedge \ldots \wedge \theta_k + (-1)^{i+1} \theta_1 \wedge \ldots (df \wedge \theta_i + f d\theta_i)\wedge \ldots \wedge \theta_k \\ \ &= \sum_{l\neq i} (-1)^{l+1+|f||\theta_i|}\theta_1\wedge \ldots \wedge d\theta_l \wedge \ldots \wedge f\theta_{i+1} \wedge \ldots \wedge \theta_k + (-1)^{i+2+|f||\theta_i|} \theta_1 \wedge \ldots \wedge \theta_i \wedge df \wedge \ldots \wedge \theta_k \\ &+ (-1)^{i+1 + |f||\theta_i|} \theta_1 \wedge \ldots d\theta_i \wedge f \theta_{i+1} \wedge \ldots \wedge \theta_k \\ &= (-1)^{|f||\theta_i|} \sum_{l \neq i+1} (-1)^{l+1} \theta_1 \wedge \ldots \wedge d\theta_l \wedge \ldots \wedge f\theta_{i+1} \wedge \ldots \wedge \theta_k  \\ &+ (-1)^{i+2} \theta_1 \wedge \ldots \wedge (df\wedge \theta_{i+1} +  f d\theta_{i+1}) \wedge \ldots \wedge \theta_k\\ &= (-1)^{|f||\theta_i|} \gamma(\theta_1 \otimes \cdots \otimes f\theta_{i+1} \otimes \cdots \otimes \theta_k)
\end{align*}
Hence $\gamma$ descends to the desired map  $d\colon \Omega^k_{C^{\infty}(\cX)} \to \Omega^{k+1}_{C^{\infty}(\cX)}$.

\end{proof}

\begin{proposition} \label{prop:pullback_derham}
    Any map of dg-manifolds $(\underline{\phi},\phi^*) \colon (\cX, \mathcal{O}_\cX) \to(\cY, \mathcal{O}_\cY)$ induces a unique map of commutative differential graded dg-algebras $f^* \colon \Omega^{\bullet}_{C^{\infty}(\cY)} \to \Omega^{\bullet}_{C^{\infty}(\cX)}$. Explicitly:
    \begin{equation}
    \label{pullback}
       \phi^*(fdg_1\wedge \ldots \wedge dg_k) = \phi^*(f)d\phi^*(g_1)\wedge \ldots \wedge d\phi^*(g_k)    
    \end{equation}
\end{proposition}
\begin{proof} By universal property of K\"ahler differentials, the map $\phi^* \colon C^{\infty}(\cY) \to C^{\infty}(\cX)$ induces a map of dg-modules $\phi^* \colon \Omega^1_{C^{\infty}(\cY)} \to \Omega^1_{C^{\infty}(\cX)}$. By taking exterior powers of $\phi^*$ we obtain a map $\Lambda^k \phi^* \colon \Omega^k_{C^{\infty}(\cY)} \to \Omega^k_{C^{\infty}(\cX)}$. These patch together into a map of complexes \[\phi^* := \Lambda^{\bullet}\phi^* \colon \Lambda^{\bullet} \Omega^1_{C^{\infty}(\cY)} \to \Lambda^{\bullet} \Omega^1_{C^{\infty}(\cX)}\] satisfying \eqref{pullback}. Now \begin{align*}
    d\phi^*(f dg_1\wedge \ldots \wedge dg_k) &= d(\phi^*(f) d\phi^*(g_1)\wedge \ldots \wedge d\phi^*(g_k)) \\ &= d\phi^*(f) \wedge d\phi^*(g_1)\wedge \ldots \wedge d\phi^*(g_k) \\ &=\phi^*(df \wedge dg_1 \wedge \ldots \wedge dg_k) \\ &= \phi^*d(fdg_1\wedge \ldots \wedge dg_k)
\end{align*}
and 
\begin{multline*}
    \delta_\cX\phi^*(f dg_1\wedge \ldots \wedge dg_k) = \delta_\cX(\phi^*(f) d\phi^*(g_1)\wedge \ldots \wedge d\phi^*(g_k)) = \\ = \delta_\cX(\phi^*(f)) d\phi^*(g_1)\wedge \ldots \wedge d\phi^*(g_k) + \sum_{i=0}^{k-1} (-1)^{|f|+|g_1|+\ldots + |g_i|} \phi^*(f) d\phi^*(g_1)\wedge \ldots \wedge d(\delta_\cX\phi^*(g_{i+1}))\wedge \ldots\wedge d\phi^*(g_k) =\\ =\phi^*(\delta_\cY f) d\phi^*(g_1)\wedge \ldots \wedge d\phi^*(g_k) + \sum_{i=0}^{k-1} (-1)^{|f|+|g_1|+\ldots |g_i|} \phi^*(f) d\phi^*(g_1)\wedge \ldots \wedge d(\phi^*(\delta_\cY g_{i+1}))\wedge \ldots\wedge d\phi^*(g_k) = \\ =\phi^*( (\delta_\cY f )dg_1\wedge \ldots \wedge dg_k + \sum_{i=0}^{k-1} (-1)^{|f|+|g_1|+\ldots |g_i|} fdg_1 \wedge \ldots \wedge d(\delta_\cY g_{i+1})\wedge \ldots \wedge dg_k) = \\ = \phi^* \delta_\cY(f dg_1 \wedge \ldots \wedge dg_k)
\end{multline*}
\end{proof}

We finish the appendix with establishing contractions and Lie derivatives. For manifolds contractions and Lie derivatives are derivations on the algebra of forms. Therefore, we define the module of derivations on forms for dg-maniolds.

\begin{definition}
    We define $\mathrm{Der}^n(\Omega^\bullet_{\cin(\cX)})$ to be the space of derivations of degree $n$ on the algebra of forms $\Omega^\bullet_{\cin(\cX)}$, i.e. maps $D \colon \Omega^\bullet_{\cin(\cX)} \to \Omega^{\bullet + n}_{\cin(\cX)}$ satisfying the Leibniz rule $D(\alpha \wedge \beta) = D(\alpha) \wedge \beta + (-1)^{n\deg(\alpha)} \alpha \wedge D(\beta)$ for all $\alpha, \beta \in \Omega^\bullet_{\cin(\cX)}$.

    The $\cin(\cX)$-module structure on the space $\mathrm{Der}^n(\Omega^\bullet_{\cin(\cX)})$ is defined on one forms by $(fD)(\alpha) = f(D(\alpha))$ and is extended to higher forms to satisfy the Leibniz rule: \[(fD)(\alpha_1 \wedge \ldots \wedge \alpha_k) = \sum_{i=1}^k (-1)^{n(i-1)} \alpha_1 \wedge \ldots \wedge fD(\alpha_i) \wedge \ldots \wedge \alpha_k\]. The dg-structure on the module $\mathrm{Der}^n(\Omega^\bullet_{\cin(\cX)})$ is given by $[\delta, \cdot]$, where $\delta$ is the internal differential in $\Omega^\bullet_{\cin(\cX)}$.
\end{definition}

\begin{proposition}
    \label{prop:contractions}
    There is a map $\iota \colon T_{\cX} \to \mathrm{Der}^{-1}(\Omega^\bullet_{\cin(\cX)})$, which for any vector field $X \in T_\cX$, with degree $|X| = n$, assigns a derivation $\iota_X \colon \Omega^\bullet_{\cin(\cX)} \to \Omega^{\bullet -1}_{\cin(\cX)}$ uniquely defined by 
    \[\iota_X(fdg_1\wedge \ldots \wedge dg_k) = \sum_i (-1)^{i-1} fdg_1\wedge \ldots \wedge  X(dg_i) \wedge \ldots \wedge dg_k.\]

    This map satisfies the following properties:
    \begin{enumerate}
        \item The contraction $\iota_X$ is a derivation of bidegree $(-1,n)$.
        \item We have $\iota_{fX} = f \iota_X$ for all $f \in C^{\infty}(\cX)$.
        \item Define the Lie derivative $\mathcal{L}_X \colon  \Omega^\bullet_{\cin(\cX)} \to  \Omega^\bullet_{\cin(\cX)}$ by the formula
        \[\mathcal{L}_X = d \circ \iota_X + \iota_X \circ d.\] Then $\mathcal{L}_X$ is a derivation of bidegree $(0,n)$ and the following equation is satisfied: \[\iota_{[X,Y]} = [\mathcal{L}_X, \iota_Y]\]

        \item The inner differential $\delta$ on $\Omega^\bullet_{\cin(\cX)}$ can be identified with the Lie derivative along the cohomological vector field $\mathcal{L}_{\delta_\cX}$. 
        \item The contraction map $\iota \colon T_\cX \to \mathrm{Der}^{-1}(\Omega^\bullet_{\cin(\cX)})$ is an isomorphism of dg-$\cin(\cX)$-modules.
    \end{enumerate}

\end{proposition}
\begin{proof} First, given a vector field $X \in T_\cX$, by universal property of K\"ahler differentials there is a map $\iota_X \colon \Omega^1_{\cin(\cX)} \to C^{\infty}(\cX)$ such that $\iota_X(df) = X(f)$ for all $f \in C^{\infty}(\cX)$. We extend this map to higher forms by the formula in the statement of the proposition. It is easy to see that the map $\iota_X$ is a derivation of bidegree $(-1,n)$, and that $\iota_{fX} = f \iota_X$ for all $f \in C^{\infty}(\cX)$.

The fact that the Lie derivative $\mathcal{L}_X$ is a derivation of bidegree $(0,n)$ follows from the fact that it is a graded commutator of two derivations of bidegree $(-1,n)$ and $(1,0)$. Next, we check the identity $\iota_{[X,Y]} = [\mathcal{L}_X, \iota_Y]$. Since both sides of the equation are derivations of form degree $-1$, it is enough to prove the equations on generators $dg$ of $\Omega^1_{\cin(\cX)}$ for $g\in \cin(\cX)$. We compute: \[[\mathcal{L}_X, \iota_Y](dg) = \mathcal{L}_X(Y(g)) - \iota_Y(d(X(g))) = X(Y(g)) - Y(X(g)) = [X,Y](g) = \iota_{[X,Y]}(dg).\]

Recall that inner differential $\delta$ is defined on generators of the module $\Omega^1_{\cX}$ of K\"ahler differentials as $\delta(dg) = d(\delta_{\cX}g)$. Therefore, $\mathcal{L}_{\delta_\cX}(dg) = d(\iota_{\delta_\cX} dg) + \iota_{\delta_\cX} d(dg) = d(\delta_{\cX} g) = \delta(dg)$, and the identity $\mathcal{L}_{\delta_\cX} = \delta$ follows.

Finally, we check that $\iota \colon T_\cX \to \mathrm{Der}^{-1}(\Omega^\bullet_{\cin(\cX)})$ is an isomorphism of dg-$\cin(\cX)$-modules. Since $\iota_{[\delta_\cX, X]} = [\mathcal{L}_{\delta_\cX}, \iota_X]$, we get that $\iota$ is a dg-map. It is easy to see that $\iota$ is injective, since if $X$ is a vector field such that $\iota_X = 0$, then $X(f) = \iota_X(df) = 0$ for all $f \in C^{\infty}(\cX)$, hence $X=0$. To see that $\iota$ is surjective, pick a derivation $D \in \mathrm{Der}^{-1}(\Omega^\bullet_{\cin(\cX)})$. Then we can define a vector field $X_D$ by $X_D(f) = D(df)$ for all $f \in C^{\infty}(\cX)$. It's easy to check that $\iota_{X_D} = D$, hence $\iota$ is surjective.

\end{proof}

\section{Fiber products of dg-manifolds} \label{sec:fiber_prod}

Even though the category of dg-manifolds has arbitrary homotopy fiber products, sometimes it is more convenient to work with strict fiber products, when they exist. In this section we will present an explicit construction of a fiber product in the category of dg-manifolds. We follow the approach from \cite[Chapter 7]{VYSOKY}, which established a similar result for graded manifolds.

\subsection{Tangent spaces of dg-manifolds} \label{sec:dg_tangent}

Let $(\cX, \mathcal{O}_\cX, \delta_\cX)$ be a dg-manifold. For any point $m \in \cX^0$, there is a natural evaluation map $\mathrm{ev}_m \colon C^{\infty}(\cX) \to \R$ given by $\mathrm{ev}_m(f) = \underline{f}(m)$, where $\underline{f} \colon \cX^0 \to \R$ is the underlying smooth function of $f \in \cin(\cX)$. We will write $f(m)$ for $\mathrm{ev}_m(f)$. Note that if the degree of $f$ is non-zero, then $f(m) = 0$ for any point $m \in \cX^0$. The evaluation map descends to a map on stalks $\mathrm{ev}_m \colon \mathcal{O}_{\cX,m} \to \R$. 

\begin{definition} \label{def:tangent_space}
    We define \emph{the tangent space} of a dg-manifold $\cX$ at a point $m \in \cX^0$ to be the graded vector space of derivations at $m$, i.e. \[T_m \cX  = \{D \colon \mathcal{O}_{\cX,m} \to \R \, | \, D(fg) = D(f)g(m) + (-1)^{|f||D|}f(m)D(g)\}.\] $T_m \cX$ is a graded vector space with grading given by the degree of derivations.

    For any map of dg-manifolds $(\underline{\phi}, \phi^*) \colon \cX \to \cY$ and any point $m \in \cX^0$, there is an induced map on tangent spaces \[T_m \phi \colon T_m \cX \to T_{\underline{\phi}(m)} \cY, \quad T_m \phi(D) = D \circ \phi^*_m\] where $\phi^*_m \colon \mathcal{O}_{\cY, \underline{\phi}(m)} \to \mathcal{O}_{\cX,m}$ is the map on stalks induced by $\phi^*$.
\end{definition}

In a coordinate chart around a point $m\in \cX^0$, the tangent space $T_m \cX$ can be identified with the graded vector space spanned by tangent vectors $\{\partial_{x_i}\vert_m\}$, where $x_i$ are the coordinates of the chart and $\partial_{x_i}\vert_m$ are the corresponding derivations at $m$ (for construction see \cite[Proposition 4.6]{VYSOKY}).

\begin{remark}
In general the tangent space $T_m \cX$ is not a dg-vector space. However if the cohomological vector field $\delta_\cX$ vanishes at $m$, then the space $T_m \cX$ inherits a dg-vector space structure from the first term in the Taylor expansion of the cohomological vector field $\delta_\cX$ at $m$. The construction for supermanifolds can be found in \cite{AKSZ}, and it easily generalizes to dg-manifolds.

In particular if $\cX = (E[-1], \iota_s)$ is a dg-manifold associated to a vector bundle $E \to M$ and a section $s \in \Gamma(E)$ as in Construction~\ref{ctr:zerosection}, then the cohomological vector field $\iota_s$ vanishes at a point $m \in M$ if and only if $s(m) = 0$. In this case, the tangent space $T_m \cX$ with the induced differential can be identified with the complex $T_m M \xrightarrow{D_m s} E_m$, where $D_m s$ is the vertical differential of $s$ at $m$. This coincides with the tangent complex defined by Behrend, Liao and Xu in \cite{BLX}.

\end{remark}

\subsection{Inverse image and pullback of sheaves of dg-modules}

This section is devoted to the construction of inverse image and pullback of sheaves of dg-modules along maps of dg-manifolds, which we will be using later when constructing submanifolds and inverse images of submanifolds. The results of this subsection are known, and are carried out in \cite[\href{https://stacks.math.columbia.edu/tag/0FRT}{Tag 0FRT}]{stacks-project}. We will give a brief overview of the theory in the context of dg-manifolds, and refer the reader to \cite{stacks-project} for details.

Let $f \colon \cX \to \cY$ be a morphism of dg-manifolds, with $\underline{f} \colon \cX^0 \to \cY^0$ the underlying map of smooth manifolds. Let $\mathcal{A}$ be a sheaf of dg-$\R$-algebras on the underlying manifold $\cY^0$. 

\begin{definition}
The \emph{inverse image} $\underline{f}^{-1}\cA$ is the sheafification of the presheaf on $\cX^0$ defined by
\[
    U \mapsto \colim_{V \supseteq \underline{f}(U)} \cA(V),
\]
where the colimit runs over open sets $V \subseteq \cY^0$ containing $\underline{f}(U)$, directed by reverse inclusion. 
\end{definition}

\begin{lemma} \label{lem:inverse_algebra}
The sheaf $\underline{f}^{-1}\cA$ is naturally a sheaf of dg-$\R$-algebras on the underlying manifold $\cX^0$.
\end{lemma}
\begin{proof}
    The presheaf $\underline{f}^{-1}_{pre}\cA$ defined by $U \mapsto \colim_{V \supseteq \underline{f}(U)} \cA(V)$ inherits a natural structure of a presheaf of dg-$\R$-algebras, since all the operations and the differential on the sheaf of dg-algebras $\cA$ are compatible with restriction maps. For example, given representatives of elements $a, b \in \underline{f}^{-1}_{pre}\cA(U)$ by pairs $[a,V]$ and $[b,W]$ with $V,W \supseteq \underline{f}(U)$ and $a \in \cA(V)$ and $b \in \cA(W)$, we can define their product in $\underline{f}^{-1}_{pre}\cA(U)$ by $[ab, V\cap W]$, where $ab$ is the product of $a$ and $b$ in $\cA(V\cap W)$. Scalar multiplication and differential are defined similarly. Since all of the operations on the sheaf of dg-algebras $\cA$ are compatible with restriction maps, this construction is independent of the choice of representatives, and it gives $\underline{f}^{-1}_{pre}\cA$ the structure of a presheaf of dg-$\R$-algebras.

    Now the sheafification $\underline{f}^{-1}\cA$ of $\underline{f}^{-1}_{pre}\cA$ inherits a natural structure of a sheaf of dg-$\R$-algebras, since it is a functor from the category of presheaves of dg-algebras to the category of sheaves of dg-algebras.
\end{proof}

\begin{definition}\label{def:inv-image-module}
Let $\cM$ be a sheaf of dg-modules over $\cA$. The \emph{inverse image} $\underline{f}^{-1}\cM$ is the sheafification of the presheaf
\[
    U \mapsto \colim_{V \supseteq \underline{f}(U)} \cM(V).
\]
\end{definition}

\begin{proposition}
$\underline{f}^{-1}\cM$ is naturally a sheaf of  dg-$\underline{f}^{-1}\cA$-modules.
\end{proposition}
\begin{proof}
    The proof proceeds analagously to the proof of Lemma \ref{lem:inverse_algebra}. The $\underline{f}^{-1}_{pre}\cA$-module structure on $\underline{f}^{-1}_{pre}\cM$ is defined by \[ [a,V] \cdot [m,W] = [a\cdot m, V\cap W] \] for  $a \in \cA(V), m \in \cM(W)$ where  $V\supseteq \underline{f}(U)$ and $W \supseteq \underline{f}(U)$. This gives $\underline{f}^{-1}_{pre}\cM$ the structure of a presheaf of  dg-$\underline{f}^{-1}_{pre}\cA$-modules, and the sheafification $\underline{f}^{-1}\cM$ inherits a natural structure of a sheaf of  dg-$\underline{f}^{-1}\cA$-modules.
\end{proof}

\begin{definition}[Pullback]\label{def:pullback}
Let $\cM$ be a sheaf of dg-modules over the sheaf of dg-algebras $\cO_\cY$. The \emph{pullback} of the sheaf $\cM$ along the map $f\colon \cX \to \cY$ is the sheaf of dg-$\cO_\cX$-modules on the underlying manifold $\cX^0$ defined by
\[
    f^*\cM = \cO_\cX \otimes_{f^{-1}\cO_\cY} f^{-1}\cM,
\]
where $\otimes_{f^{-1}\cO_\cY}$ denotes the tensor product of sheaves of dg-$f^{-1}\cO_\cY$-modules: the sheafification of the presheaf
\[
    U \mapsto \cO_\cX(U) \otimes_{f^{-1}\cO_\cY(U)} f^{-1}\cM(U).
\]
\end{definition}

\begin{proposition}\label{prop:dg-module_pullback}
The sheaf $f^*\cM$ is a sheaf of dg-$\cO_\cX$-modules on the underlying manifold $\cX^0$.
\end{proposition}
\begin{proof}
    As in the proof of Lemma \ref{lem:inverse_algebra}, it is enough to check that the presheaf $f^*_{pre} \cM = U \mapsto \cO_\cX(U) \otimes_{f^{-1}\cO_\cY(U)} f^{-1}\cM(U)$ has a natural structure of a presheaf of dg-$\cO_\cX$-modules. Indeed, we have the following structures on the presheaf $f^*_{pre} \cM(U)$:

    \begin{enumerate}[label=(\roman*)]
    \item \emph{Grading:} $f^*_{pre}\cM(U)_n = \bigoplus_{p+q=n} \cO_\cX(U)_p \otimes_{f^{-1}\cO_\cY(U)^0} f^{-1}\cM(U)_q / \sim$, where the equivalence relation is generated by $b f(a) \otimes m \sim b \otimes am$ for $a \in f^{-1}\cO_\cY(U)$, $b \in \cO_\cX(U)$ and $m \in f^{-1}\cM(U)$.

    \item \emph{Left $\cO_\cX(U)$-module structure:} $b' \cdot (b \otimes m) := (b'b) \otimes m$ for $b' \in \cO_\cX(U)$ which is well-defined since
    \[
        b'(b f(a)) \otimes m = b'b f(a) \otimes m = b'b \otimes am.
    \]

    \item \emph{Differential:} Defined on elementary tensors by
    \[
        \delta(b \otimes m) := \delta_\cX(b) \otimes m + (-1)^{|b|} b \otimes \delta_\cM(m).
    \]
    This is well-defined on $f^*_{pre}\cM(U)$: applying $\delta$ to $b f(a) \otimes m - b \otimes am$ gives
    \begin{align*}
        &\delta_\cX(b f(a)) \otimes m + (-1)^{|b|+|a|} b f(a) \otimes \delta_\cM(m)\\
        &\quad - \delta_\cX(b) \otimes am - (-1)^{|b|} b \otimes \delta_\cM(am).
    \end{align*}
    Using $\delta_\cX(b f(a)) = \delta_\cX(b) f(a) + (-1)^{|b|}b\,f(\delta_\cY(a))$ (since $f$ is a dg-algebra map, so $\delta_\cX \circ f = f \circ \delta_\cY$) and
    $\delta_\cM(am) = \delta_\cY(a)m + (-1)^{|a|}a\,\delta_\cM(m)$, these terms cancel, confirming $\delta$ is well-defined.

    \item \emph{$\delta^2 = 0$:} We compute
    \[
        \delta^2(b \otimes m) = \delta(\delta_\cX(b) \otimes m + (-1)^{|b|} b \otimes \delta_\cM(m)).
    \]
    Expanding and using $\delta_\cX^2 = 0$, $\delta_\cM^2 = 0$, and the sign cancellation
    \[
        (-1)^{|b|+1}\delta_\cX(b)\otimes \delta_\cM(m) + (-1)^{|b|}\delta_\cX(b)\otimes \delta_\cM(m) = 0,
    \] 
    we get $\delta^2 = 0$.

    \item \emph{Leibniz rule for the $\cO_\cX(U)$-module structure:} For $b' \in \cO_\cX(U)$,
    \begin{align*}
        \delta(b' \cdot (b\otimes m)) &= \delta(b'b \otimes m) = \delta_\cX(b'b)\otimes m + (-1)^{|b'|+|b|}b'b\otimes \delta_\cM(m)\\
        &= (\delta_\cX(b')b + (-1)^{|b'|}b'\,\delta_\cX(b))\otimes m + (-1)^{|b'|+|b|}b'b\otimes \delta_\cM(m)\\
        &= \delta_\cX(b') \cdot (b\otimes m) + (-1)^{|b'|}b'\cdot \delta(b\otimes m).
    \end{align*}
\end{enumerate}

\end{proof}

\begin{construction}[Functoriality of pullback]\label{ctr:pullback_functoriality}
Let $f \colon \cX \to \cY$ be a map of dg-manifolds and $\phi \colon \cM \to \cN$ be a morphism of sheaves of dg-$\cO_\cY$-modules. Then there is an induced morphism of sheaves of dg-$\cO_\cX$-modules $f^*\phi \colon f^*\cM \to f^*\cN$.

Sectionwise the map $f^*\phi$ is defined by $f^*\phi(b \otimes m) = b \otimes f^{-1}\phi(m)$ for $b \in \cO_\cX(U)$ and $m \in f^{-1}\cM(U)$. Here $f^{-1}\phi \colon f^{-1}\cM \to f^{-1}\cN$ is the morphism of sheaves of dg-$f^{-1}\cO_\cY$-modules induced by the sheafification of the presheaf morphism $f^{-1}_{pre}\phi \colon f^{-1}_{pre}\cM \to f^{-1}_{pre}\cN$ defined by $f^{-1}_{pre}\phi([m,V]) = [\phi(m), V]$ for $m \in \cM(V)$ and $V \supseteq \underline{f}(U)$.

The map $f^*\phi$ has the following properties:
\begin{itemize}
\item \emph{Well-defined}: $bf(a)\otimes m \mapsto bf(a)\otimes \phi(m) = b\otimes (a\cdot\phi(m)) = b\otimes\phi(am)$ for $a \in \cO_\cY(U)$, $b \in \cO_\cX(U)$, and $m \in f^{-1}\cM(U)$, using the $\cO_\cY$-linearity of $\phi$.
    \item \emph{Commutes with $\delta$:} $(f^*\phi)(\delta(b\otimes m)) = (f^*\phi)(\delta_\cX(b)\otimes m + (-1)^{|b|}b\otimes \delta_\cM(m)) = \delta_\cX(b)\otimes\phi(m) + (-1)^{|b|}b\otimes\phi(\delta_\cM(m)) = \delta_\cX(b)\otimes\phi(m) + (-1)^{|b|}b\otimes \delta_\cN(\phi(m)) = \delta(f^*\phi(b\otimes m))$, using $\phi \circ \delta_\cM = \delta_\cN \circ \phi$.
    \item \emph{$\cO_\cX$-linear:} $(f^*\phi)(b'\cdot(b\otimes m)) = b'\cdot(f^*\phi)(b\otimes m)$ is immediate.
\end{itemize}

\end{construction}

\subsection{Dg-submanifolds}

\begin{definition}
    We say that a map of dg-manifolds $(\underline{\phi}, \phi^*) \colon \cS \to \cY$ is an \emph{immersion} if for any point $m \in \cS^0$, the tangent map $T_m \phi \colon T_m \cS \to T_{\underline{\phi}(m)} \cY$ is an injective map of graded vector spaces.
    
    We say that a pair $(\cS, \phi)$ is a \emph{(closed) embedded dg-submanifold} of a dg-manifold $\cY$ if the map $\phi$ is an immersion and the underlying smooth map $\underline{\phi} \colon \cS^0 \to \cY^0$ is a closed embedding of smooth manifolds.

    We also say that two embedded dg-submanifolds $(\cS, \phi)$ and $(\cS', \phi')$ of $\cY$ are \emph{equivalent} if there is an isomorphism of dg-manifolds $\psi \colon \cS \to \cS'$ such that $\phi' \circ \psi = \phi$.
\end{definition}

Given an embedded dg-submanifold $(\cS, \phi)$ of $\cY$ and an induced map of sheaves $\phi^* \colon \mathcal{O}_\cY \to \underline{\phi}_* \mathcal{O}_\cS$, we define the corresponding sheaf of ideals $\mathcal{J}_\cS = \ker \phi^*$.

\begin{proposition} \label{prop:emb_sub_ideal}
    The sheaf of ideals $\mathcal{J}_\cS$ satisfies the following properties:
    \begin{enumerate}
        \item The sheaf $\mathcal{J}_\cS$ is a sheaf of dg-ideals, i.e. $\delta_\cY(\mathcal{J}_\cS(U)) \subset \mathcal{J}_\cS(U)$ for any open subset $U \subset \cY^0$.
        \item For any point $m \in \underline{\phi}(\cS^0)$, there is a chart $U$ around $m$ with coordinates $\{x_i\}_{i\in I}$ and a subset $J \subset I$ such that $\mathcal{J}_\cS(U)$ is generated by $\{x_j\}_{j \in J}$.
        \item The map $\phi^*$ induces an isomorphism of sheaves of dg-algebras  \[\tilde{\phi}^* \colon \underline{\phi}^{-1}(\mathcal{O}_\cY/\mathcal{J}_\cS) \xrightarrow{\sim} \mathcal{O}_\cS\].
    \end{enumerate}
\end{proposition}

\begin{proof}
    The first property follows from the fact that the pullback map $\phi^*$ is a morphism of dg-algebras, and so its kernel is preserved by the differential.

    The second property is proved in \cite[Proposition 7.30]{VYSOKY}.

    Finally, for the third statement, first take an inverse presheaf $\underline{\phi}^{-1}_{pre}(\mathcal{O}_\cY/\mathcal{J}_\cS)$, which is defined by \[\underline{\phi}^{-1}_{pre}(\mathcal{O}_\cY/\mathcal{J}_\cS)(V) = \lim_{V \subset U} \mathcal{O}_\cY(U)/\mathcal{J}_\cS(U),\] where the limit is taken over all open subsets $U$ of the underlying manifold $\cY^0$ containing open set $V$. The map $\tilde{\phi}^* \colon \underline{\phi}^{-1}_{pre}(\mathcal{O}_\cY/\mathcal{J}_\cS) \to \mathcal{O}_\cS$ is defined by sending a section $s \in \underline{\phi}^{-1}(\mathcal{O}_\cY/\mathcal{J}_\cS)(V)$, which is represented by a section $\tilde{s} \in \mathcal{O}_\cY(U)/\mathcal{J}_\cS(U)$ for some open set $U$ containing $V$, to the section $\phi^*(\tilde{s})$. In \cite[Proposition 7.30]{VYSOKY} it is shown that the map $\tilde{\phi}^*$ is an isomorphism of presheaves of graded algebras. Since $\mathcal{O}_\cS$ is a sheaf, the map  $\tilde{\phi}^*$ factors through the sheafification $\underline{\phi}^{-1}(\mathcal{O}_\cY/\mathcal{J}_\cS)$, and the resulting map $\tilde{\phi}^* \colon \underline{\phi}^{-1}(\mathcal{O}_\cY/\mathcal{J}_\cS) \to \mathcal{O}_\cS$ is an isomorphism of sheaves of graded algebras. Finally, since $\tilde{\phi}^*$ by construction intertwines the differentials, it is an isomorphism of sheaves of dg-algebras.
    
\end{proof}

\begin{definition}
    Let $\mathcal{I}$ be a sheaf of dg-ideals inside the sheaf of dg-algebras $\mathcal{O}_\cY$. Denote by $Z(\mathcal{I})$ the subset of $\cY^0$ consisting of points $m$ such that for any section $s \in \mathcal{I}(U)$ defined in a neighborhood $U$ of $m$, we have $\mathrm{ev}_m(s) = 0$. We say that $\mathcal{I}$ is a \emph{sheaf of regular dg-ideals} if for any point $m \in Z(\mathcal{I})$, there is a chart $U$ around $m$ with coordinates $\{x_i\}_{i\in I}$ and a subset $J \subset I$ such that $\mathcal{I}(U)$ is generated by $\{x_j\}_{j \in J}$.
\end{definition}

Clearly if $(\cS, \phi)$ is an embedded dg-submanifold of a dg-manifold $\cY$, then the sheaf of ideals $\mathcal{J}_\cS$ is a sheaf of regular dg-ideals by Proposition~\ref{prop:emb_sub_ideal}. The following proposition gives a converse statement.

\begin{proposition} \label{prop:dg-submanifold}
    Let $\mathcal{I}$ be a sheaf of regular dg-ideals in $\mathcal{O}_\cY$. Then there is an embedded dg-submanifold $(\cS, \phi)$ of the dg-manifold $\cY$ such that $\mathcal{I} = \mathcal{J}_\cS$. Moreover, any two embedded dg-submanifolds of the dg-manifold $\cY$ with the same sheaf of ideals are equivalent.
\end{proposition}
\begin{proof}
    By \cite[Theorem 7.35]{VYSOKY} there is a unique graded submanifold $(\cS,\phi)$ of the dg-manifold $\cY$ with a structure sheaf $\mathcal{O}_\cS = \underline{\phi}^{-1}(\mathcal{O}_\cY/\mathcal{I})$. Since $\mathcal{I}$ is a dg-ideal, the differential on the sheaf of dg-algebras $\mathcal{O}_\cY$ descends to a differential on quotient sheaf $\mathcal{O}_\cS$, making $(\cS, \phi)$ into a dg-submanifold of $\cY$. Finally, if $(\cS', \phi')$ is another embedded dg-submanifold of $\cY$ with the same sheaf of ideals $\mathcal{I}$, then by Proposition~\ref{prop:emb_sub_ideal} there are isomorphisms of sheaves of dg-algebras $\tilde{\phi}^* \colon \underline{\phi}^{-1}(\mathcal{O}_\cY/\mathcal{I}) \to \mathcal{O}_\cS$ and $\tilde{\phi}'^* \colon \underline{\phi}'^{-1}(\mathcal{O}_\cY/\mathcal{I}) \to \mathcal{O}_{\cS'}$. Hence we get an isomorphism of sheaves of dg-algebras $\psi^* = (\tilde{\phi}')^* \circ (\tilde{\phi}^*)^{-1} \colon \mathcal{O}_\cS \to \mathcal{O}_{\cS'}$, which induces an isomorphism of dg-manifolds $\psi \colon \cS' \to \cS$ such that $\phi' = \phi \circ \psi$.
\end{proof}

\subsection{Products of dg-manifolds}

The category of dg-manifolds possesses finite products, which are given by the Cartesian product of the underlying manifolds. The underlying sheaf is more complicated than the tensor product of the sheaves of the factors, which can be seen even for manifolds, as $\cin(M\times N)$ is not isomorphic to $\cin(M)\otimes_{\R} \cin(N)$. One has to take a suitable completion of the tensor product.

\begin{proposition} \label{prop:dg-product}
    Let $(\cX, \mathcal{O}_\cX, \delta_\cX)$ and $(\cY, \mathcal{O}_\cY, \delta_\cY)$ be two dg-manifolds. Then there exists a dg-manifold $\cX \times \cY = (\cX^0 \times \cY^0, \mathcal{O}_{\cX\times \cY}, \delta_{\cX\times \cY})$ and dg-maps $\pi_\cX \colon \cX \times \cY \to \cX$ and $\pi_\cY \colon \cX \times \cY \to \cY$ such that $\underline{\pi_\cX}$ and $\underline{\pi_\cY}$ are smooth projections of $\cX^0 \times \cY^0$ onto $\cX^0$ and $\cY^0$ respectively. Moreover, $(\cX \times \cY, \pi_\cX, \pi_\cY)$ is a binary product of $\cX$ and $\cY$ in the category of dg-manifolds.
\end{proposition}

\begin{proof}
    If we consider $\cX$ and $\cY$ as graded manifolds, then by \cite[Proposition 3.41]{VYSOKY}, there is a graded binary product $\cX \times \cY = (\cX^0 \times \cY^0, \mathcal{O}_{\cX\times \cY})$, together with graded projection $\pi_\cX \colon \cX \times \cY \to \cX$ and $\pi_\cY \colon \cX \times \cY \to \cY$. We define a vector field $\delta_{\cX\times \cY}$ on $\cX \times \cY$ by the formula \[\delta_{\cX\times \cY} = \delta_\cX \otimes 1 + 1 \otimes \delta_\cY\] where $\delta_\cX \otimes 1$ and $1 \otimes \delta_\cY$ are natural lifts of $\delta_\cX$ and $\delta_\cY$ to the product $\cX \times \cY$ (see \cite[Example 4.25]{VYSOKY}). Let $U$ be a chart on $\cX$ with coordinates $\{x_i\}$ and $V$ be a chart on $\cY$ with coordinates $\{y_j\}$. Then the product $U \times V$ is a chart on the product graded manifold $\cX \times \cY$ with coordinates $\{x_i, y_j\}$. In these coordinates we can write $\delta_\cX = \sum_i f_i \partial_{x_i}$ and $\delta_\cY = \sum_j g_j \partial_{y_j}$, where $f_i \in \mathcal{O}_\cX(U)$ and $g_j \in \mathcal{O}_\cY(V)$. Hence we get \begin{equation} \label{eq:local_delta_prod} \delta_{\cX\times \cY}\vert_{U\times V} = \sum_i \pi_\cX^*(f_i) \partial_{x_i} + \sum_j \pi_\cY^*(g_j) \partial_{y_j}.\end{equation} Since $\partial_{x_i} \pi^*_\cY(g_j) = 0$ and $\partial_{y_j} \pi^*_\cX(f_i) = 0$, and $\delta_\cX^2 =0$ and $\delta_\cY^2 = 0$, we get that $\delta_{\cX\times \cY}^2\vert_{U\times V} = 0$. Hence $\delta_{\cX\times \cY}$ is a cohomological vector field on the product space $\cX \times \cY$, and $(\cX \times \cY, \delta_{\cX\times \cY})$ is a dg-manifold. By construction, the projections $\pi_\cX$ and $\pi_\cY$ are maps of dg-manifolds. Finally if $\cZ$ is another object with dg-maps $\phi_\cX \colon \cZ \to \cX$ and $\phi_\cY\ \colon \cZ \to cY$, then by the universal property of the product of graded manifolds there is a unique map of graded manifolds $\psi \colon \cZ \to \cX \times \cY$ such that $\pi_\cX \circ \psi = \phi_\cX$ and $\pi_\cY \circ \psi = \phi_\cY$. In local chart $U \times V$ on $\cX \times \cY$,  with coordinates $\{x_i, y_j\}$, the map $\psi$ is given by $\psi^*(x_i) = \phi_\cX^*(x_i)$ and $\psi^*(y_j) = \phi_\cY^*(y_j)$. Since $\phi_\cX$ and $\phi_\cY$ are maps of dg-manifolds, $\psi^*$ is also a dg-map. 
\end{proof}

\subsection{Transversality and preimages}
Preimages of submanifolds under smooth maps are a special case of fiber products. They exist if the map is transverse to the submanifold. We give a similar criterion for the existence of preimages of embedded dg-submanifolds under maps of dg-manifolds.

\begin{definition}
    Let $\phi \colon \cX \to \cZ$ and $\varphi \colon \cY \to \cZ$ be two maps of dg-manifolds. We say that $\phi$ and $\varphi$ are \emph{transverse} if for any point $m \in \cX^0$ and $n \in \cY^0$ such that $\underline{\phi}(m) = \underline{\varphi}(n)$, the images of the tangent maps $T_m \phi \colon T_m \cX \to T_{\underline{\phi}(m)} \cZ$ and $T_n \varphi \colon T_n \cY \to T_{\underline{\varphi}(n)}\cZ$ together generate the whole tangent space $T_{\underline{\phi}(m)}\cZ$, i.e. \[\mathrm{Im}(T_m\phi) + \mathrm{Im}(T_n\varphi) = T_{\underline{\phi}(m)}\cZ.\] We write $\phi \pitchfork \varphi$ to denote that $\phi$ and $\varphi$ are transverse.
\end{definition}

\begin{proposition} \label{prop:inverse_image}
    Let $(\cS, \iota)$ be an embedded dg-submanifold of a dg-manifold $\cZ$ and $\phi \colon \cX \to \cZ$ be a map of dg-manifolds such that $\phi$ and $\iota$ are transverse $\phi \pitchfork \iota$. Then there is an embedded dg-submanifold $(\phi^{-1}(\cS), \tilde{\iota})$ of the dg-manifold $\cX$ and a dg-map $\tilde{\phi} \colon \phi^{-1}(\cS) \to \cS$ fitting into the following commutative diagram:  \[\begin{tikzcd}
\phi^{-1}(\cS) \arrow[r, "\tilde{\phi}"] \arrow[d, "\tilde{\iota}"] & \cS \arrow[d, "\iota"] \\
\cX \arrow[r, "\phi"] & \cZ. 
\end{tikzcd}\] Moreover this diagram is a pullback square in the category of dg-manifolds.
\end{proposition}
\begin{proof}
    Our proof is an adaptation of the proof of graded case in \cite[Theorem 7.43]{VYSOKY}. We need to show the following:
    \begin{enumerate}[label=(\roman*)]
        \item \emph{Construction of $\phi^{-1}(\cS)$}. Let $\mathcal{J}_\cS$ be the sheaf of ideals corresponding to the embedded dg-submanifold $(\cS, \iota)$. In \cite[Theorem 7.43]{VYSOKY} Vysoky constructs a sheaf of graded ideals $\mathcal{I}$ in $\mathcal{O}_\cX$ corresponding to the preimage of $\cS$ under $\phi$. The construction works as follows: let $j \colon \mathcal{J}_\cS \to \mathcal{O}_\cZ$ be the inclusion of the sheaf of ideals, and let $\phi^*j \colon \phi^* \mathcal{J}_\cS \to \phi^* \mathcal{O}_\cZ \cong \mathcal{O}_\cX$ be the pullback of this inclusion (see Construction~\ref{ctr:pullback_functoriality}). Then the sheaf of graded ideals $\mathcal{I}$ is defined as the image of $\phi^*j$. The transversality condition $\phi \pitchfork \iota$ implies that $\mathcal{I}$ is a sheaf of regular graded ideals (this follows from Vysoky's proof). Furthermore, since $\mathcal{J}_\cS$ is a dg-ideal, the pullback $\phi^* \mathcal{J}_\cS$ is a sheaf of dg-$\cO_\cX$-modules (see Proposition \ref{prop:dg-module_pullback}), and the image $\mathcal{I}$ of $\phi^*j$ is a sheaf of dg-ideals. Therefore by Proposition~\ref{prop:dg-submanifold} there is an embedded dg-submanifold $(\phi^{-1}(\cS), \tilde{\iota})$ of $\cX$ corresponding to the sheaf of ideals $\mathcal{I}$. 
        
        \item \emph{Construction of $\tilde{\phi}\colon \phi^{-1}(\cS) \to \cS$}. Consider a map $\phi \circ \tilde{\iota} \colon \phi^{-1}(\cS) \to \cZ$. In order to show that it descends to a map $\tilde{\phi} \colon \phi^{-1}(\cS) \to \cS$, it is enough to show that the pullback of the sheaf of ideals $\mathcal{J}_\cS$ along $\phi \circ \tilde{\iota}$ is zero. This is checked in \cite[Theorem 7.43]{VYSOKY}.
        
        \item \emph{Pullback diagram}. Finally, we need to show that the diagram is a pullback square in the category of dg-manifolds. Let $(\cW, \alpha, \beta)$ be another cone over $\cX$ and $\cS$. We need to show that there is a unique map of dg-manifolds $\psi \colon \cW \to \phi^{-1}(\cS)$ such that $\tilde{\iota} \circ \psi = \alpha$ and $\tilde{\phi} \circ \psi = \beta$. Once again, it is enough to show that the map $\alpha \colon \cW \to \cX$ factors through $\phi^{-1}(\cS)$, i.e. that the pullback of the sheaf of ideals $\mathcal{I}$ along $\alpha$ is zero. This is also shown in \cite[Theorem 7.43]{VYSOKY}.
    \end{enumerate}

\end{proof}

\subsection{Fiber products}

\begin{construction}
    For any dg-manifold $\cZ$, there is a diagonal map $\Delta_\cZ \colon \cZ \to \cZ \times \cZ$ defined by $\pi^1_\cZ \circ \Delta_\cZ = \pi^2_\cZ \circ \Delta_\cZ = \mathrm{id}_\cZ$, where $\pi^1_\cZ$ and $\pi^2_\cZ$ are the projections from $\cZ \times \cZ$ to the first and second factor respectively. The diagonal map is an embedding (see \cite[Example 7.49]{VYSOKY}).
\end{construction}

\begin{proposition} \label{prop:dg_fiber_products}
    Let $\phi \colon \cX \to \cZ$ and $\varphi \colon \cY \to \cZ$ be two maps of dg-manifolds. Then the transversality $\phi \pitchfork \varphi$ holds if and only if the map $\phi \times \varphi \colon \cX \times \cY \to \cZ \times \cZ$ is transverse to the diagonal $\Delta_\cZ$. In this case, there exists an embedded dg-submanifold $\cX \times_\cZ \cY$ of $\cX \times \cY$ defined as the preimage of $\Delta_\cZ$ under $\phi \times \varphi$, and the following diagram is a pullback square in the category of dg-manifolds: \[\begin{tikzcd}
\cX \times_\cZ \cY \arrow[r, "\pi_\cY"] \arrow[d, "\pi_\cX"] & \cY \arrow[d, "\varphi"] \\
\cX \arrow[r, "\phi"] & \cZ.
\end{tikzcd}\]
\end{proposition}
\begin{proof}
This proposition consists of three statements:
\begin{enumerate}[label=(\roman*)]
    \item \emph{$\phi \pitchfork \varphi$ if and only if $(\phi \times \varphi) \pitchfork \Delta_\cZ$.} This is a straightforward check using the definition of transversality (see \cite[Proposition 7.51]{VYSOKY}).
    \item \emph{If $(\phi \times \varphi) \pitchfork \Delta_\cZ$, then there exists an embedded dg-submanifold $\cX \times_\cZ \cY$ of $\cX \times \cY$ defined as the preimage of the diagonal $\Delta_\cZ$ under $\phi \times \varphi$.} This follows from Proposition~\ref{prop:inverse_image}.
    \item \emph{The diagram is a pullback square in the category of dg-manifolds.} The proof follows \cite[Proposition 7.52]{VYSOKY} \textit{mutatis mutandis}. The idea is to rewrite the universal property of the fiber product in terms of the universal property of the preimage (established in Proposition~\ref{prop:inverse_image}) of the diagonal under the map $\phi \times \varphi$.
\end{enumerate}
\end{proof}

\section{Serre--Swan for dg-vector bundles} \label{sec:serreswan}

The goal of this appendix is to show that no information is lost when passing from locally free sheaves of dg-modules to global sections. More precisely, we show that the global sections functor induces an equivalence between the category of locally free sheaves of dg-modules of constant rank and the category of projective finitely generated dg-modules over the dg-manifold. This result was established for graded manifolds by Smolka and Vysoky in \cite{SMOLKA}, so we won't give a detailed proof, but rather explain how to adapt the arguments of \cite{SMOLKA} to the dg-setting.

\subsection{Locally free sheaves of dg-modules}
Let $(\cX, \cO_\cX, \delta_{\cX})$ be a dg-manifold with a base manifold $\cX^0$.

\begin{definition}
    We say that the sheaf $(\cS, \delta_{\cS})$ is a \emph{sheaf of dg-modules} over the sheaf of dg-algebras $(\cO_\cX, \delta_{\cX})$ if
    \begin{itemize}
        \item The sheaf $\cS$ is a sheaf of graded $\R$-vector spaces over the underlying manifold $\cX^0$ and $\delta_{\cS} : \cS \to \cS$ is a degree $+1$ sheaf morphism such that $\delta_{\cS}^2 = 0$.
        
        \item For any open set $U \subset \cX^0$, $(\cS(U), \delta_{\cS}\vert_U)$ is a $\cO_\cX(U)$-dg-module, i.e. $\cS(U)$ is a graded $\cO_\cX(U)$-module and \[\delta_{\cS}\vert_U(f\cdot m) = \delta_{\cX}\vert_U (f)\cdot m + (-1)^{|f|}f \cdot \delta_{\cS}\vert_U (m)\] for any $f \in \cO_\cX(U)$ and $m \in \cS(U)$.

    \end{itemize}

    A \emph{morphism of sheaves of dg-modules} $\phi : (\cS, \delta_{\cS}) \to (\cP, \delta_{\cP})$ is a sheaf morphism $\phi : \cS \to \cP$ such that for any open set $U \subset \cX^0$, $\phi\vert_U : (\cS(U), \delta_{\cS}\vert_U) \to (\cP(U), \delta_{\cP}\vert_U)$ is a morphism of $\cO_\cX(U)$-dg-modules.
\end{definition}

\begin{construction}
Given a sheaf of dg-modules $(\cS, \delta_{\cS})$ on a dg-manifold $(\cX, \cO_\cX, \delta_\cX)$,
one can form the \emph{dual sheaf} $\cS^* = \underline{\Hom}_{\cO_\cX}(\cS, \cO_\cX)$.
We recall what this means precisely. For each open $U \subseteq X$, the sections in degree $k$ are
\[
    \cS^*(U)^k
    \;=\;
    \Hom^k_{\cO_\cX(U)}\!\bigl(\cS(U),\, \cO_\cX(U)\bigr)
    \;=\;
    \prod_{n \in \Z} \Hom_{\cO_\cX(U)}\!\bigl(\cS(U)^n,\, \cO_\cX(U)^{n+k}\bigr),
\]
that is, $\cO_\cX(U)$-linear maps $\alpha \colon \cS(U) \to \cO_\cX(U)$ of total degree $k$, meaning $\alpha$ maps the degree $n$ component $\cS(U)^n$ into degree $n+k$ component $\cO_\cX(U)^{n+k}$ for every $n$. We write $\langle \alpha, s \rangle$ for the pairing. The assignment $U \mapsto \cS^*(U)$ is made into a sheaf via the restriction maps of $\cS$ and $\cO_\cX$: for $V \subseteq U$, we set $\langle \alpha|_V, s|_V \rangle = \langle \alpha, s \rangle|_V$. We endow $\cS^*$ with the dual differential $\delta_{\cS^*}$ defined on sections $\alpha \in \cS^*(U)^{|\alpha|}$ and $s \in \cS(U)$ by
\[
    \langle \delta_{\cS^*}(\alpha),\, s \rangle
    \;=\;
    \delta_\cX\!\bigl(\langle \alpha, s \rangle\bigr)
    \;-\;
    (-1)^{|\alpha|}\,\bigl\langle \alpha,\, \delta_\cS(s) \bigr\rangle.
\]
\end{construction}

\begin{definition}
We say that a sheaf $\cS$ of graded $\cO_\cX$-modules is \emph{locally free of finite rank} if there exists  an open cover $\{U_i\}$ of the base manifold $\cX^0$ such that on each open set $U_i$ there is a finite collection of local generators $e_1, \ldots, e_N \in  \cS(U_i)$ of degrees $|e_1|, \ldots, |e_N| \in \Z$ such that every section $s \in \cS(U_i)$ can be written uniquely as
\[
    s \;=\; \sum_{a=1}^N f^a \cdot e_a, \qquad f^a \in \cO_\cX(U_i),
\]
with $|f^a| = |s| - |e_a|$. Equivalently, one has an isomorphism of graded $\cO_\cX|_{U_i}$-modules
\[
    \cS|_{U_i}
    \;\cong\;
    \bigoplus_{a=1}^N \cO_\cX|_{U_i}[{-|e_a|}].
\]

We say that $\cS$ has \emph{constant graded rank} if the multiset of generator degrees $(|e_1|, \ldots, |e_N|)$ is the same on every trivialising open set $U_i$.
\end{definition}

\begin{definition}
    We define the category $\catname{dgVB}$ of \emph{dg-vector bundles} consisting of pairs $(\cX, \cS)$, where $\cX$ is a dg-manifold and $\cS$ is a locally free sheaf of dg-modules of constant graded rank over $\cX$. A morphism from $(\cX, \cS)$ to $(\cY, \cP)$ is a pair $(\phi, \Lambda)$, where $\phi \colon \cX \to \cY$ is a morphism of dg-manifolds and $\Lambda \colon \cP^* \to \phi_*(\cS^*)$ is a morphism of sheaves of dg-modules.
\end{definition}

\begin{remark}
    The dual $(\cS^*, \delta_{\cS^*})$ of a locally free sheaf of dg-modules of constant graded rank is again a locally free sheaf of dg-modules of the same graded rank (with degrees negated, since the dual frame
    $e^1, \ldots, e^N$ defined by $\langle e^a, e_b \rangle = \delta^a_b$ satisfies
    $|e^a| = -|e_a|$). 
\end{remark}

\begin{example}[The tangent sheaf as a dg-vector bundle]
    A $\cO_\cX(U)$-linear map $X \colon \cO_\cX(U) \to \cO_\cX(U)$ of degree $k$ is called a \emph{vector field of degree $k$} on $U \subseteq \cX^0$ if it satisfies the graded
    Leibniz rule
    \[
        X(f \cdot g) \;=\; X(f) \cdot g \;+\; (-1)^{k|f|} f \cdot X(g)
    \]
    for all $f, g \in \cO_\cX(U)$. Such maps are also called \emph{graded derivations} of $\cO_\cX(U)$, and we set
    \[
        \mathfrak{X}_\cX(U)^k
        \;=\;
        \mathrm{Der}^k(\cO_\cX(U))
        \;=\;
        \bigl\{\, X \colon \cO_\cX(U) \to \cO_\cX(U) \text{ linear},\;
        |X| = k,\; X(fg) = X(f)g + (-1)^{k|f|}fX(g) \,\bigr\}.
    \]
    The assignment $U \mapsto \mathfrak{X}_\cX(U) = \bigoplus_k \mathfrak{X}_\cX(U)^k$ defines a sheaf of graded $\cO_\cX$-modules, the \emph{tangent sheaf} of the dg-manifold $\cX$, where the $\cO_\cX(U)$-module structure is given by $(f \cdot X)(g) = f \cdot X(g)$.

    The cohomological vector field $\delta_\cX \in \mathfrak{X}_\cX(U)^1$ acts on the sheaf of vector fields $\mathfrak{X}_\cX$ via the \emph{Lie derivative}
    \[
        L_{\delta_\cX}(X) \;=\; [\delta_\cX, X]
        \;=\; \delta_\cX \circ X - (-1)^{|X|} X \circ \delta_\cX,
    \]
    which has degree $|X| + 1$, so that $L_{\delta_\cX}$ is an odd operator on $\mathfrak{X}_\cX(U)$. Furthermore, $L_{\delta_\cX}^2 = \frac{1}{2}[\delta_\cX, [\delta_\cX, -]] = 0$ by the graded Jacobi identity and $\delta_\cX^2 = 0$. Thus $(\mathfrak{X}_\cX, L_{\delta_\cX})$ is a sheaf of $\cO_\cX$-dg-modules.

    It is locally free of constant graded rank: on a coordinate chart $U$ with local coordinates $(x^i)$ of degrees $|x^i|$, the derivations $\partial/\partial x^i$ of degree $-|x^i|$ form a local frame, so
    \[
        \mathfrak{X}_\cX\big|_U
        \;\cong\;
        \bigoplus_i \cO_\cX|_U\bigl[|x^i|\bigr],
    \]
    and the graded rank is the multiset $\{-|x^i|\}$, which is constant across overlapping charts by compatibility of the dg-manifold atlas. Hence $(\mathfrak{X}_\cX, L_{\delta_\cX}) \in \catname{dgVB}$.
\end{example}

Next proposition states that morphisms of sheaves of dg-modules are determined by their global component, which will be important for the proof of Serre--Swan equivalence. 
\begin{proposition}[{cf.\ \cite[Proposition~3.4]{SMOLKA}}]
\label{prop:glob_sections}
    Let $(\cS, \delta_\cS)$ and $(\cP, \delta_\cP)$ be locally free sheaves of dg-modules of constant graded rank over dg-manifolds $(\cX, \cO_\cX, \delta_\cX)$ and $(\cY, \cO_\cY, \delta_\cY)$ respectively. Then the following data are equivalent:
    \begin{enumerate}[label=(\roman*)]
        \item A morphism $(\phi, \Lambda) \colon (\cX, \cS) \to (\cY, \cP)$ of dg-vector bundles over $\phi \colon \cX \to \cY$.
        \item A pair $(\phi, \Lambda_0)$, where $\phi \colon \cX \to \cY$ is a morphism of dg-manifolds and $\Lambda_0 \colon \cP^*(\cY^0) \to \cS^*(\cX^0)$ is a degree-zero linear map satisfying
              \begin{align}
                  \Lambda_0(f \cdot \alpha)
                  &= \phi^*(f) \cdot \Lambda_0(\alpha)
                  \label{eq:Lambda0_linear} \\
                  \delta_{\cS^*(\cX^0)} \circ \Lambda_0
                  &= \Lambda_0 \circ \delta_{\cP^*(\cY^0)},
                  \label{eq:Lambda0_dg}
              \end{align}
        for all $f \in \cO_\cY(\cY^0)$ and $\alpha \in \cP^*(\cY^0)$.
    \end{enumerate}
\end{proposition}

\begin{proof}
    \textbf{(i) $\Rightarrow$ (ii).} Given a morphism $(\phi, \Lambda)$, set $\Lambda_0 = \Lambda|_{\cP^*(\cY^0)} \colon \cP^*(\cY^0) \to \cS^*(\cX^0)$. Conditions~\eqref{eq:Lambda0_linear} and~\eqref{eq:Lambda0_dg} are the $\cO_\cY(\cY^0)$-linearity of $\Lambda$ restricted to global sections, and the intertwining of the global differentials, both of which hold by assumption.

    \textbf{(ii) $\Rightarrow$ (i).} Given $(\phi, \Lambda_0)$, we first construct a morphism of the underlying graded sheaves. By~\cite[Proposition~3.4]{SMOLKA}, applied to the graded case, the map $\Lambda_0$ satisfying~\eqref{eq:Lambda0_linear} determines a unique morphism $(\phi, \Lambda) \colon \cS \to \cP$ of graded sheaves whose restriction to global dual sections recovers $\Lambda_0$, i.e. $\Lambda|_{\cP^*(\cY^0)} = \Lambda_0$. It remains to verify that the map $\Lambda$ intertwines the sheaf differentials, i.e.\ that
    \[
        \delta_{\cS^*} \circ \Lambda \;=\; \Lambda \circ \delta_{\cP^*}
    \]
    holds on the level of sheaves.

    To this end we use the following locality property: each of the sheaf maps $\delta_{\cP^*}$, $\delta_{\cS^*}$, and $\Lambda$ (and hence their compositions $\delta_{\cS^*} \circ \Lambda$ and $\Lambda \circ \delta_{\cP^*}$) is a local operator, in the sense that if $\alpha \in \cP^*(V)$ vanishes on some open $W \subseteq V$, then both $\delta_{\cP^*}(\alpha)$ and $\Lambda(\alpha)$ vanish on $W$ (resp.\ on $\phi^{-1}(W)$). For the differentials this follows from the graded Leibniz rule: if $\alpha|_W = 0$ and $\rho \in \cO_\cY$ is a bump function supported in $V \setminus W$, then $\alpha = \rho \cdot \alpha$ near $W$ and $\delta_{\cP^*}(\rho \cdot \alpha) = \delta_\cY(\rho) \cdot \alpha + (-1)^{|\rho|}\rho \cdot \delta_{\cP^*}(\alpha)$ vanishes on $W$. For $\Lambda$ it follows from $\cO_\cY$-linearity via~\eqref{eq:Lambda0_linear}.

    Since $\delta_{\cS^*} \circ \Lambda$ and $\Lambda \circ \delta_{\cP^*}$ are both local, they are each determined by their values on global sections. Indeed, any section $\alpha \in \cP^*(V)$ over an open $V \subseteq \cY^0$ can be compared with a global section after multiplying by a suitable bump function, and locality ensures the compositions agree wherever the bump function is identically $1$. Since $\Lambda_0$ satisfies~\eqref{eq:Lambda0_dg} on global sections, it follows that $\delta_{\cS^*} \circ \Lambda = \Lambda \circ \delta_{\cP^*}$ as sheaf morphisms. Hence $(\phi, \Lambda)$ is a morphism of sheaves of dg-modules.
\end{proof}

\subsection{Projective dg-modules}

Throughout this section, $(\cX, \cO_\cX, \delta_\cX)$ denotes a dg-manifold with underlying smooth manifold $X = \cX^0$, and $A = \cO_\cX(X)$ denotes its dg-algebra of global sections.

\begin{definition}
    \label{def:finitely_gen}
    A dg-module $(P, \delta_P)$ over a dg-algebra $(A, \delta_A)$ is \emph{finitely generated over $A$} if there exists a finite set of $\{p_i\}_{i=1}^n \subset P$,  such that every element $p \in P$ can be written as a finite sum
    \[
        p \;=\; \sum_{i=1}^n f_i \cdot p_i, \qquad f_i \in A,\quad |f_i| = |p| - |p_i|.
    \]
\end{definition}

\begin{definition}
    \label{def:graded_projective}
    A dg-module $(P, \delta_P)$ over a dg-algebra $(A, \delta_A)$ is \emph{graded projective} if it is projective as a graded $A$-module, i.e.\ if for any degree-zero graded $A$-module morphism $\psi \colon P \to M$ and surjective $\phi \colon Q \twoheadrightarrow M$, there exists a degree-zero graded $A$-module morphism $\tilde{\psi} \colon P \to Q$ such that $\phi \circ \tilde{\psi} = \psi$:
    \[
    \begin{tikzcd}
        & Q \arrow[d, "\phi", two heads] \\
        P \arrow[r, "\psi"'] \arrow[ur, dashed, "\tilde\psi"] & M.
    \end{tikzcd}
    \]
\end{definition}

\begin{proposition}[{cf.~\cite[Proposition~4.13]{SMOLKA}}]
\label{prop:dual_projective}
    For any dg-$A$-module $(P, \delta_P)$, the dual dg-module is defined as $P^* = \underline{\Hom}_A(P, A)$ with the differential $\delta_{P^*}$ given by the graded commutator $\delta_{P^*}(\alpha) = [\delta_A, \alpha]$. If $(P, \delta_P)$ is graded projective and finitely generated, then so is the dual module $(P^*, \delta_{P^*})$, and the canonical degree-zero dg-$A$-module map
    \[
        \iota \colon P \to P^{**}, \qquad \iota(p)(\alpha) = (-1)^{|p||\alpha|}\alpha(p),
    \]
    is an isomorphism of dg-$A$-modules.
\end{proposition}

\begin{proof}
    The proof carries over from~\cite[Proposition~4.13]{SMOLKA} without modification.
\end{proof}

\begin{definition}
\label{def:dgProj}
    The category $\catname{dgProj}$ is defined as follows.
    \begin{itemize}
        \item \emph{Objects} are pairs $(\cX, P)$, where $\cX$ is a dg-manifold and $(P, \delta_P)$ is a graded projective finitely generated dg-module over $A = \cO_\cX(\cX^0)$.
        \item \emph{Morphisms} from $(\cX, P)$ to $(\cY, Q)$ are pairs $(\phi, \Lambda_0)$, where $\phi \colon \cX \to \cY$ is a morphism of dg-manifolds and $\Lambda_0 \colon Q^* \to P^*$ is a degree-zero linear map satisfying
              \begin{align}
                  \Lambda_0(f \cdot \alpha)
                  &= \phi^*(f) \cdot \Lambda_0(\alpha),
                  \label{eq:proj_linear} \\
                  \delta_{P^*} \circ \Lambda_0
                  &= \Lambda_0 \circ \delta_{Q^*},
                  \label{eq:proj_dg}
              \end{align}
        for all global sections $f \in \cO_\cY(\cY^0)$ and elements $\alpha \in Q^*$.
    \end{itemize}
\end{definition}

\begin{proposition}[{cf.~\cite[Proposition~4.21]{SMOLKA}}]
\label{prop:functor_upsilon}
    Let $(\cS, \delta_\cS) \in \catname{dgVB}$ be a dg-vector bundle over a dg-manifold
    $\cX$. Then $(\cX,\, \cS(\cX^0))$ is an object of $\catname{dgProj}$. Moreover, the assignment
    \[
        \Upsilon(\cS, \delta_\cS) \;=\; \bigl(\cX,\, \cS(\cX^0)\bigr),
        \qquad
        \Upsilon(\phi, \Lambda) \;=\; (\phi, \Lambda_0),
    \]
    where $\Lambda_0 = \Lambda|_{\cP^*(\cY^0)}$ denotes restriction to global dual sections, defines a fully faithful functor
    \[
        \Upsilon \colon \catname{dgVB} \to \catname{dgProj}.
    \]
\end{proposition}

\begin{proof}
    That $\cS(\cX^0)$ is graded projective and finitely generated over $\cO_\cX(\cX^0)$ follows from~\cite[Proposition~4.21]{SMOLKA} by the same argument as in the graded case.  Full faithfulness of $\Upsilon$ is a direct consequence of Proposition~\ref{prop:glob_sections}.
\end{proof}

\begin{theorem}[Serre--Swan for dg-vector bundles]
\label{thm:serre_swan}
    For every graded projective dg-module of finite type $P \in \catname{dgProj}$ over the dg-manifold $\cX$, there exists a dg-vector bundle $(\cS, \delta_\cS) \in \catname{dgVB}$ over $\cX$, unique up to isomorphism, such that
    \[
        P \;\cong\; \cS(\cX^0)
    \]
    as dg-$\cO_\cX(\cX^0)$-modules. Consequently, the global sections functor
    \[
        \Upsilon \colon \catname{dgVB} \xrightarrow{\;\sim\;} \catname{dgProj}
    \]
    is essentially surjective. Combined with full faithfulness given by Proposition~\ref{prop:functor_upsilon}, this implies that the global section functor $\Upsilon$ is an equivalence of categories.
\end{theorem}

\begin{proof}
    By the graded Serre--Swan theorem~\cite[Theorem~4.34]{SMOLKA}, there exists a locally free sheaf of graded $\cO_\cX$-modules $\cS$ of constant graded rank such that
    $P \cong \cS(\cX^0)$ as graded $\cO_\cX(\cX^0)$-modules. It remains to show that $\cS$ is a sheaf of dg-modules with a differential compatible with the differential on the dg-module $P$.

    Via the isomorphism $P \cong \cS(\cX^0)$, the differential $\delta_P$ defines a differential $\delta_\cS(\cX^0) \colon \cS(\cX^0) \to \cS(\cX^0)$ on global sections satisfying
    the graded Leibniz rule with respect to $\cO_\cX(\cX^0)$ and $\delta_{\cO_\cX(\cX^0)}$. We extend this to a sheaf differential $\delta_\cS \colon \cS \to \cS$ by a classical bump function argument. Take any open $U \subseteq X$ with local frame $(e_1, \ldots, e_N)$, and any section $s \in \cS(U)$. Then on any open $W \subset U$, such that $\overline{W} \subset U$, we can write $s = \rho \cdot \tilde{s}\vert_W$, where $\rho$ is a bump function on $\cX^0$ supported in $U$ and $\tilde{s} \in \cS(\cX^0)$ is a global section. We then set \[\delta_\cS(s|W) = \delta_\cS(\cX^0)(\tilde{s})|_W.\] The $\cO_\cX(\cX^0)$-linearity of $\delta_\cS(\cX^0)$ and the Leibniz rule ensure that this definition is independent of the choice of $\rho$ and $\tilde{s}$. Covering $U$ by such $W$ and using the sheaf axiom for $\cS$ glues these local definitions into a well-defined map $\delta_\cS(U) \colon \cS(U) \to \cS(U)$. These maps are compatible with restrictions, and therefore assemble into sheaf map $\delta_\cS \colon \cS \to \cS$. Furthermore, by construction $\delta_\cS(U)$ satisfies the Leibniz rule and $\delta_\cS^2(U) = 0$. Thus $(\cS, \delta_\cS) \in
    \catname{dgVB}$, and $\Upsilon(\cS, \delta_\cS) = (\cX, P)$ by construction.

    Uniqueness up to isomorphism follows from full faithfulness of $\Upsilon$ (Proposition~\ref{prop:functor_upsilon}). Together with essential surjectivity, this gives the equivalence of categories.
\end{proof}

\subsection*{Summary of notation} \label{sec:sum_notation}
\addcontentsline{toc}{section}{\nameref{sec:sum_notation}}
\begin{center}
\setlength{\tabcolsep}{14pt}
\renewcommand{\arraystretch}{1.5}
  \begin{longtable}{ l p{10cm}}

 $\cX = (\cX^0, \mathcal{O}_\cX, \delta_\cX)$ & Dg-manifold $\cX$ with underlying manifold $\cX^0$, sheaf of functions $\mathcal{O}_\cX$, and differential $\delta_\cX$. See Definition~\ref{def:dgmanifold}.\\

$f = (\underline{f}, f^*) \colon \cX \to \cY$  & Morphism of dg-manifolds $\cX$ and $\cY$, consisting of a smooth map $\underline{f}:\cX^0 \to \cY^0$ and a pullback map $f^*:\mathcal{O}_\cY \to \mathcal{O}_\cX$. See Definition~\ref{def:dgmanifold}.\\

$\cin(\cX)$ & Global sections of structure sheaf $\mathcal{O}_\cX$ of the dg-manifold $\cX$. See Notation~\ref{not:globalfunctions}.\\

$(\Omega^1_{\cin(\cX)}, \delta^1_\cX)$ or $(T^*_\cX, \delta^1_\cX)$ & The module of Kähler differentials on the algebra of global functions $\cin(\cX)$ with the differential $\delta^1_\cX$. See Definition~\ref{def:kahler}.\\

$(T_{\cin(\cX)}, \mathcal{L}_{\delta_\cX})$ or $(T_\cX, \mathcal{L}_{\delta_\cX})$ & The module of vector fields on $\cin(\cX)$ with the Lie derivative $\mathcal{L}_{\delta_\cX}$. See Definition~\ref{def:tangentmodule}.\\

$\Omega^p_{\cin(\cX)}$ & The module of $p$-forms on the dg-manifold $\cX$. See Definition~\ref{def:dgformsmodule}.\\

$d \colon \Omega^p_{\cin(\cX)} \to \Omega^{p+1}_{\cin(\cX)}$ & The exterior derivative of the de Rham complex of $\cin(\cX)$. See Definition~\ref{def:dgformsmodule} and Theorem~\ref{thm:derham_diff}. \\

$\iota^*_\cX \mathcal{M}$ &
A pullback of a $\cin(\cX^0)$-module $\mathcal{M}$ to the dg-manifold $\cX$ given by \[\iota^*_\cX \mathcal{M} = \mathcal{O}_\cX \otimes_{\cin(\cX^0)} \mathcal{M}.\] See Notation~\ref{not:pullback}.\\

$V_\cX$ & The module of sections of a trivial vector bundle with fiber $V$ over the dg-manifold $\cX$ given by $\cin(\cX) \otimes_\R V$. See Notation~\ref{not:pullback}.\\

$\cone(f)$ & The cone of a morphism of dg-$\cin(\cX)$-modules $f \colon A \to B$. See Definition~\ref{def:cone_cocone} \\
$\cocone(f)$ & The cocone of a morphism of dg-$\cin(\cX)$-modules $f \colon A \to B$. See Definition~\ref{def:cone_cocone}. \\

$d\mu \colon T_\cZ \to \fg^*_\cZ$ & The differential of the moment map $\mu$ restricted to the derived zero set $\cZ$. See Proposition~\ref{prop:kahler}.\\

$\cZ$ & The derived zero locus of the moment map $\mu \colon M \to \fg^*$, i.e. \[\cZ = (M, \cin(M)\otimes_\R \wedge^{-\bullet}\fg, \iota_\mu).\] See Equation~\eqref{eq:zerolevelset} and Construction~\ref{ctr:zerosection}.\\

$\cZ/G$ & The action groupoid of the derived zero set $\cZ$ by the Lie group $G$, i.e. $\cZ/G = (G\times \cZ \rightrightarrows \cZ)$. See Construction~\ref{ctr:actiongroupoid}.\\

$u \colon \cZ \to G\times \cZ$ & The unit map of the action groupoid $\cZ/G$. See Construction~\ref{ctr:actiongroupoid_general}.\\

$(\cZ/G)^\bullet$ & The nerve of the action groupoid $\cZ/G$. See Definition~\ref{def:nerve} and Construction~\ref{ctr:actiongroupoid_nerve}.\\

$(C^{\bullet}(\cZ/G), \partial)$ & The complex of cochains on the nerve of the action groupoid $\cZ/G$. See Definition~\ref{def:cochains}.\\

$X \in \fg$ & Generators of the algebra $\cin(\cZ)$ in degree -1. See the discussion after Equation~\eqref{eq:zerolevelset}.\\

$dX \in T^*_\cZ$ & A one-form on $\cZ$ of derived degree -1, which is the image of an element $X \in \fg$ under the exterior derivative $d \colon \cin(\cZ) \to \Omega^1_{\cin(\cZ)}$. See Notation~\ref{not:dg_alg_not}.\\

$\iota_\sigma \in T_\cZ$ & The vector field on $\cZ$ of degree 1 given by contraction $\iota_\sigma \colon \cin(M) \otimes \wedge^{\bullet} \fg \to \cin(M) \otimes \wedge^{\bullet -1}  \fg$ with a constant section $\sigma \in \fg^*$. See Notation~\ref{not:dg_alg_not}.\\

$\rho \colon \cin(\cZ) \otimes \fg \to T_\cZ$ & The anchor map of the Lie algebroid of the action groupoid $\cZ/G$. See Proposition~\ref{prop:algebroid}. \\

$\rho_0 \colon \cin(\cZ) \otimes \fg \to \iota^*_\cZ T_M$ & The infinitesimal action of the Lie algebra $\fg$ on $T_M$ restricted to the derived zero set $\cZ$. See Proposition~\ref{prop:algebroid}.\\

$\eta \colon \cin(\cZ) \otimes \fg \to \cin(\cZ) \otimes \fg^*[-1]$ & A chain homotopy between $d\mu \circ \rho_0$ and $0$ given by $\eta(X) =\textrm{ad}^*_X$. See Proposition~\ref{prop:algebroid}.\\

$E_i$ & A basis of the Lie algebra $\fg$. See the proof of Proposition~\ref{prop:algebroid}.\\

$\{\overline{E}_i\}$ & A left invariant frame of the Lie group $G$ such that $\overline{E}_i(e) = E_i$. See the proof of Proposition~\ref{prop:algebroid}.\\

$\{\sigma_i\}$ & A basis of the dual space $\fg^*$ which is dual to the basis $E_i$ of $\fg$. See the proof of Proposition~\ref{prop:algebroid}.\\

$c_{ij}(g)$ & Coefficients of the adjoint action $\textrm{Ad}_g$ in the basis $E_i$, i.e. $\textrm{Ad}_g = \sum_{i,j} c_{ij}(g) E_i \otimes \sigma_j$. See the proof of Proposition~\ref{prop:algebroid}.\\

$X^{\sharp}$ & The vector field on the base manifold $M$ induced by the infinitesimal action of $X \in \fg$. See the proof of Lemma~\ref{lem:anchor_equivariant}.\\

$\mathrm{DR}^m(\mathcal{G})$ & The de Rham complex of a dg-groupoid $\mathcal{G}$, defined as the total complex of the triple complex $(\Omega^{\bullet}_{\cin(N_{\bullet}\cG)}, d, \delta, \partial^*)$. See Definition~\ref{def:bottshulman}.\\

$\omega_{\mathrm{red}}$ & The reduced form on the action groupoid $\cZ/G$, defined as $\omega_{\mathrm{red}} = \iota^*_\cZ \omega + d\theta \in \mathrm{DR}^2(\cZ/G)$, where $\iota^*_\cZ \omega$ is the original symplectic form on $M$ pulled back to $\cZ$ and $\theta$ is the standard (left) Maurer--Cartan form on $G$ pulled back to $G \times \cZ$. See Construction~\ref{ctr:reducedform}.\\

$T_{\cZ/G}$ & The tangent complex of the action groupoid $\cZ/G$, given by $T_{\cZ/G} = \fg_\cZ \xrightarrow{\rho} T_\cZ$. See Construction~\ref{ctr:tangent_complex_action}.\\

$T^*_{\cZ/G}$ & The cotangent complex of the action groupoid $\cZ/G$, given by $T^*_{\cZ/G} = T^*_\cZ \xrightarrow{\rho^*} \fg^*_\cZ$. See Construction~\ref{ctr:tangent_complex_action}.\\

$\alpha \colon T_\cZ[1] \to \fg^*_\cZ$ & The map induced by the contraction with $d\theta$, given on generators by $\alpha(V) = 0$ for $V \in \mathfrak{X}(M)$ and $\alpha(\iota_{\sigma_i}) = \sigma_i$ for $\iota_{\sigma_i} \in T_\cZ$. See Construction~\ref{ctr:inducedmap}.\\

$\alpha^* \colon T^*_\cZ[-1] \to \fg_\cZ$ & The dual map of $\alpha$, given on generators by $\alpha^*(\beta) = 0$ for $\beta \in \Omega^1(M)$ and $\alpha^*(dE_i) = E_i$ for $dE_i \in T^*_\cZ$. See Construction~\ref{ctr:inducedmap}.\\

$\omega_{\mathrm{red}}^{\flat} \colon \Tot(T_{\cZ/G}) \to \Tot(T^*_{\cZ/G})$ & The map induced by the reduced form $\omega_{\mathrm{red}}$, defined as $\omega_{\mathrm{red}}^{\flat} = \alpha + (\iota^*_\cZ \omega)^{\flat} + \alpha^*$. See Construction~\ref{ctr:inducedmap}.\\

 \end{longtable}
\end{center}\mbox{}\\
\end{document}